\documentclass[reqno]{amsart}

\usepackage{amsmath}
\usepackage{amssymb}
\usepackage{enumerate}
\usepackage[all]{xy}
\usepackage{amsthm}
\usepackage{amscd}
\usepackage{amsfonts}
\usepackage{latexsym}
\usepackage{amscd}
\usepackage{graphicx}
\usepackage{tikz}
\usepackage{setspace}
\usepackage{hyperref}
\usetikzlibrary{decorations.pathmorphing}

\definecolor{darkred}{HTML}{993333}

\makeatletter

\newcommand{\arxiv}[1]{\href{http://arxiv.org/abs/#1}{\tt arXiv:\nolinkurl{#1}}}

\def\hw{{\operatorname{hw}}}

\def\op{\operatorname{op}}

\def\L{{\mathcal L}}
\def\Lp{{\mathcal L'}}
\def\Lm{{\mathcal L_{\min}}}
\def\dotLm{{\dot{\mathcal L}_{\min}}}
\def\dotLmp{{\dot{\mathcal L}'_{\min}}}

\def\Lmp{{\mathcal L'_{\min}}}

\def\Pol{\operatorname{Pol}}
\def\Sym{\operatorname{Sym}}

\newtheorem{theorem}{Theorem}[section]
\newtheorem{lemma}[theorem]{Lemma}

\newtheorem{corollary}[theorem]{Corollary} 
\theoremstyle{definition}  
\newtheorem{definition}[theorem]{Definition}
\newtheorem{example}[theorem]{Construction}

\newtheorem{remark}[theorem]{Remark}

\def\ob{\operatorname{ob}}
\@addtoreset{equation}{section}

\def\Vec{\mathcal{V}ec}
\def\GVec{\underline{\mathcal{GV}ec}}
\def\rMod{\operatorname{Mod-\!}}
\def\lMod{\operatorname{\!-Mod}}
\def\rpMod{\operatorname{pMod-\!}}
\def\lpMod{\operatorname{\!-pMod}}
\def\rdMod{\operatorname{lfdMod-\!}}
\def\ldMod{\operatorname{\!-lfdMod}}
\def\rfMod{\operatorname{fgMod-\!}}
\def\lfMod{\operatorname{\!-fgMod}}
\def\Cat{\mathcal{C}at}
\def\GCat{\mathcal{GC}at}

\def\CCat{\mathfrak{C}\mathrm{at}}

\def\pC{\operatorname{p}\!\mathcal C}

\def\fC{\operatorname{fg}\!\mathcal C}
\def\dC{\operatorname{lfd}\!\mathcal C}

\def\CC{\mathfrak{C}}
\def\UU{\mathfrak{U}}
\def\H{\mathcal{H}}
\def\Hm{\mathcal{H}_{\min}}

\def\hhh{h}
\def\bi{\text{\boldmath$i$}}
\def\bj{\text{\boldmath$j$}}
\def\V{\mathcal V}
\def\A{\mathcal A}
\def\M{\mathcal M}
\def\L{\mathcal L}
\def\I{\mathcal I}

\def\RR{\mathcal R}
\def\SS{\mathcal S}
\def\R{\mathbb R}
\def\SSS{\mathbb S}
\def\ev{\operatorname{ev}}
\def\coev{\operatorname{coev}}

\newcommand{\End}{\operatorname{End}}

\newcommand{\Hom}{\operatorname{Hom}}

\newcommand{\g}{\mathfrak{g}}
\newcommand{\Z}{\mathbb{Z}}
\newcommand{\N}{\mathbb{N}}
\newcommand{\B}{\mathbf{B}}

\renewcommand{\k}{\Bbbk}
\newcommand{\eps}{\varepsilon}

\newcommand{\rad}{{\operatorname{rad}}}
\def\la{\lambda}

\def\e{{\mathrm{e}}}

\def\h{{\mathrm{h}}}

\def\C{\mathcal{C}}

\def\clubsuit{\diamondsuit}
\def\spadesuit{\diamondsuit}

\def\0{{\bar{0}}}
\def\1{{\bar{1}}}

\DeclareFontFamily{OT1}{pzc}{}
\DeclareFontShape{OT1}{pzc}{m}{it}{<-> s * [1.10] pzcmi7t}{}
\DeclareMathAlphabet{\mathpzc}{OT1}{pzc}{m}{it}

\DeclareMathOperator\head{hd}
\DeclareMathOperator\soc{soc}

\allowdisplaybreaks

\begin{document}

\title{Categorical actions and
  crystals}

\author[J. Brundan]{Jonathan Brundan}
\author[N. Davidson]{Nicholas Davidson}

\address{Department of Mathematics,
University of Oregon, Eugene, OR 97403, USA}

\email{brundan@uoregon.edu, davidson@uoregon.edu}

\thanks{2010 {\it Mathematics Subject Classification}: 17B10, 18D10.}
\thanks{Research supported in part by NSF grant DMS-1161094.}

\begin{abstract}
This is an expository article developing some aspects
of the theory of categorical actions of Kac-Moody algebras in the
spirit of works of Chuang--Rouquier, Khovanov--Lauda, Webster, and many others.
%We will also formulate some extensions of the existing
%definitions in order to weaken
%some finiteness hypotheses.
%In particular, we explain how the set of isomorphism classes of irreducible objects 
%in any {\em nilpotent locally Schurian 2-representation} of a Kac-Moody
%2-category carries the structure of a
%{\em crystal} in the sense of Kashiwara.
%\tableofcontents
\end{abstract}

\maketitle  

\section{Introduction}

This work is a contribution to the study of
categorifications of Kac-Moody algebras and their integrable modules.
The subject has its roots in Lusztig's construction  of canonical bases of quantum groups
using geometry of quiver varieties \cite{Lp} (which happened around 1990). It is
intimately connected to the rich combinatorial theory of crystal bases
initiated at the same time by Kashiwara \cite{Kp}. 
In the decade after that,
several other examples were studied related to the
representation theory of the symmetric group 
and associated Hecke
algebras \cite{LLT, Ariki, G} (building in particular on ideas of
Bernstein and Zelevinsky \cite{BZ}),
rational representations of the general linear group \cite{BK},
and the Bernstein-Gelfand-Gelfand category $\mathcal O$ associated to the
general linear Lie (super)algebra \cite{BFK, B1}.
The first serious attempt to put these examples into a unified axiomatic framework
was undertaken by Chuang and Rouquier \cite{CR}. They built a powerful
structure theory for studying
categorical actions of $\mathfrak{sl}_2$, which they applied notably
to prove Brou\'e's Abelian Defect Conjecture for the symmetric groups.

Another major breakthrough came in 2008, when Khovanov and Lauda \cite{KL1,KL2,KL3} 
and Rouquier \cite{Rou} independently introduced some new algebras called
{\em quiver Hecke algebras}, and used them to construct {\em Kac-Moody
  2-categories} associated to arbitrary Kac-Moody algebras.
The definitions of Kac-Moody 2-categories given by Khovanov and Lauda
and by
Rouquier look quite different, so that for a while
subsequent works split into two different schools according to which
definition they were following.
In fact, Rouquier's and Khovanov and Lauda's definitions are
equivalent,
as was established
by the first author \cite{B}.

In this (mostly expository) article, we will
revisit some of Rouquier's foundational
definitions in the light of \cite{B}.
We do this using the diagrammatic formalism
of Khovanov and Lauda wherever possible. From the outset, we have systematically incorporated
the better choice of normalization for the second adjunction of the
Kac-Moody 2-category suggested by \cite{BHLW}.
For a survey with greater emphasis on the connections to
geometry, we refer the reader to Kamnitzer's text \cite{Kam}.

Another of our goals is to extend several of the existing results so that they may
be applied in some more general situations.
To explain the novelty, we need some definitions.
Let $\k$ be an algebraically closed field and $\Vec$
be the category
of (small) vector spaces.
A {\em finite-dimensional category} is a small $\k$-linear category
$\A$ all of whose morphism spaces are finite-dimensional.
Let $\rMod\A$ denote the
functor category $\mathcal{H}om(\A^{\op}, \Vec)$
of right modules over $\mathcal A$.
We say that $\A$ is {\em Artinian}
if all of
the finitely generated objects and the finitely cogenerated
objects in $\rMod\A$ have
finite length (see also Remark~\ref{cats}). 
A {\em locally Schurian category}
is an Abelian category
that is equivalent to $\rMod\A$
for some finite-dimensional category $\A$.
If in addition $\A$ is Artinian,
then the full subcategory of $\rMod\A$
consisting of all objects of finite length is a {\em Schurian
%\footnote{Elsewhere in the literature, one also finds 
%``Schurian'' used to describe a finite-dimensional
 % category all of whose morphism spaces are one-dimensional.}
  category}
in the sense of \cite[$\S$2.1]{BLW}.

In the Abelian setting, the general 
structural results about 2-representations of Kac-Moody 2-categories obtained in
\cite{CR, Rou, R2} typically only apply to 
categories
in which all objects have finite length and whose irreducible objects
satisfy Schur's Lemma. If one wants there to be enough projectives
and injectives too,
this means that one is working in a Schurian category in the sense
just defined.
The main new contribution of this paper is to extend some of
these structural results to {locally} Schurian categories.

The motivation for doing this from a Lie theoretic
perspective is as follows.
Let $\g$ be a symmetrizable Kac-Moody algebra with Chevalley
generators $\{e_i, f_i\:|\:i \in I\}$, 
weight lattice $P$, etc...
Recall that a $\g$-module $V$ is {\em integrable} if it 
decomposes into weight spaces as $V = \bigoplus_{\lambda \in P}
V_\lambda$, and each $e_i$ and $f_i$ acts
locally nilpotently.
In order to categorify
an integrable module with finite-dimensional weight spaces,
it is reasonable to hope that one can use
a finite-dimensional category whose blocks are
finite-dimensional algebras, in which case all subsequent constructions can be performed in
the Schurian category consisting of finite-dimensional modules over these algebras.
Examples include the {\em minimal categorification}
$\Lm(\kappa)$ of the integrable
highest weight module $L(\kappa)$ of (dominant) highest weight $\kappa$ defined already by Khovanov, Lauda and
Rouquier via cyclotomic quiver Hecke algebras, and the minimal
categorifications 
$\Lm(\kappa_1,\dots,\kappa_n)$
of tensor products 
$L(\kappa_1)\otimes\cdots\otimes L(\kappa_n)$ of integrable highest
weight modules
introduced by Webser in \cite{Web}.
%(Note though that Rouquier's {\em universal categorification}
%$\L(\kappa)$ of $L(\kappa)$
%from 
%\cite[$\S$4.3.3]{R2} is not of this form as
%$\End_{\L(\kappa)}(1_\kappa)$ is too big.)

In \cite{Wcan}, Webster also investigated categorifications of more general tensor products involving both
integrable lowest weight and highest weight modules; see also \cite{BW} for the
construction of canonical bases in such mixed tensor products.
Away from finite type, these modules have infinite-dimensional weight spaces.
The candidates for their minimal categorifications 
suggested by Webster are finite-dimensional categories
which are not Artinian in general, so that
the locally Schurian setting becomes essential.
In type A, there are some closely related examples
arising from
the cyclotomic oriented Brauer categories of \cite{BCNR}, which in
level one are 
Deligne's categories
$\underline{\operatorname{Re}}\!\operatorname{p}(GL_t)$
(e.g. see \cite{AHS}).
These also fit into the framework of this article.

\iffalse
A remarkable feature of the Kac-Moody 2-category is that it carries an
additional $\Z$-grading
(for homogeneous choices of parameters), so that actually it 
categorifies quantum groups and their integrable modules.
We will not emphasize this aspect of the theory
here since it adds an extra layer of complexity to the definitions
(although it is usually quite routine to incorporate the grading into proofs). We 
will also not discuss any of the connections to Lusztig's
geometric theory of canonical bases; these have been developed in \cite{R2,VV,Wcan}.
\fi

Here is a guide to the organization of the remainder of the article.

In Section 2, we set up the basic algebraic foundations of
  locally Schurian categories. 
Everything here is either well known (e.g. see \cite{M}), or it is an obvious extension of 
classical results. However our language is new.

Section 3 is an exposition of the definition of Kac-Moody
  2-category, based mainly on \cite{B}. 
We also discuss briefly the
  graded version of the Kac-Moody 2-category. This is important
  as it makes 
the connection to quantum groups, although we will not emphasize it
  elsewhere in the article.

Section 4 begins with a review of Rouquier's theory of
2-representations of Kac-Moody 2-categories. We recall his
definition of the universal categorification
$\L(\kappa)$ of $L(\kappa)$
from \cite[$\S$4.3.3]{R2}. The minimal
categorification $\Lm(\kappa)$ is a certain finite-dimensional
specialization of $\L(\kappa)$; it can be realized equivalently in
terms of cyclotomic quiver Hecke algebras. We also introduce a
2-representation $\L(\kappa'|\kappa)$, which is expected to play the role of
universal categorification for the
tensor product $L(\kappa'|\kappa) := L'(\kappa') \otimes L(\kappa)$ 
of the integrable lowest weight module $L'(\kappa')$ of (anti-dominant)
lowest weight $\kappa'$ 
with  the integrable highest weight module $L(\kappa)$ (see Construction~\ref{bad}).
The minimal categorification $\Lm(\kappa'|\kappa)$ 
from \cite[Proposition 5.6]{Wcan} 
is a certain finite-dimensional specialization
of $\L(\kappa'|\kappa)$.
After that, we focus on nilpotent categorical actions on
  locally Schurian categories. Any such structure has
  an {\em associated crystal} in the sense of Kashiwara; for example, the
  crystal associated to $\Lm(\kappa)$ is the highest weight crystal
  $\B(\kappa)$. This has already found many
  striking applications in classical representation theory; e.g. see \cite{FK} (the oldest)
 and \cite{DVV} (the most recent at the time of writing).

\vspace{2mm}
\noindent
{\em Acknowledgements.}
We thank Ben Webster for sharing his ideas 
in \cite{erratum},
and for suggesting the reduction argument used in the proof of
Theorem~\ref{jolly}.
Also we thank Aaron Lauda for giving us the opportunity to write this
survey, and the referee for many helpful suggestions.
 
\vspace{2mm}
\noindent
{\em Notation.}
Throughout,
we work over an algebraically closed field $\k$. This means that all (2-)categories and
(2-)functors will be assumed to be $\k$-linear by default.

\section{Locally Schurian categories}

In this section, we introduce our language of locally Schurian categories.

\subsection{Locally unital algebras}
A {\em locally unital algebra} is an associative
(not necessarily unital) algebra $A$ equipped with a small family
$(1_x)_{x \in X}$ of mutually orthogonal idempotents such that 
$$
A = \bigoplus_{x,y \in X} 1_y A 1_x.
$$
A {\em locally unital homomorphism} (resp. {\em isomorphism}) between two locally unital algebras
is an algebra homomorphism (resp. isomorphism) which takes 
distinguished idempotents to distinguished idempotents.
Also, we say that $A$ is a {\em contraction} of $B$
if there is an algebra isomorphism $A\stackrel{\sim}{\rightarrow} B$ 
sending each distinguished idempotent in $A$ to
a sum of distinguished idempotents in $B$.
 
We say that $A$ is {\em locally Noetherian} (resp. {\em locally
  Artinian})
if all of the left ideals $A 1_x$ and all of the right ideals $1_y A$ satisfy the Ascending Chain Conditon (resp. the
Descending Chain Condition).
One can also define analogs of
{\em left} (resp. {\em right}) Noetherian or Artinian
for locally unital algebras, requiring just that all the left ideals
$A 1_x$ (resp. the
right ideals $1_y A$) satisfy the appropriate chain condition.
Unlike in the unital setting, locally left/right Artinian does not imply
locally left/right Noetherian (but
see Lemma~\ref{munoz} below). The
following example of a locally unital algebra that is locally left Artinian
but not locally left Noetherian is taken from the end of \cite[$\S$3]{M}:
consider the locally unital algebra of upper triangular matrices over $\k$ with rows
and columns indexed by the totally ordered set $\N \cup
\{\infty\}$, all but finitely many of whose entries are zero.
%It is a locally unital algebra with distinguished idempotents given by
%the diagonal matrix units. 

All {\em modules} over a locally unital algebra will be assumed to be locally unital
without further mention; for a right module $V$ 
this means that $V=\bigoplus_{x \in X} V 1_x$ as a direct sum of subspaces.
If $V$ is any $A$-module satisfying ACC, it is clearly finitely
generated. Conversely, assuming that 
$A$ is locally Noetherian (resp.\
locally Artinian), finitely
generated modules satisfy ACC (resp. DCC). We deduce in the
locally Noetherian
case that submodules of finitely generated modules are finitely generated.

Let $\rMod A$ be the category of all right
$A$-modules. We'll also need the following full subcategories of
$\rMod A$:
\begin{itemize}
\item $\rdMod A$ consisting of all 
{\em locally finite-dimensional modules},
i.e. right modules $V$ with $\dim V 1_x < \infty$ for all $x \in X$;
\item $\rfMod A$ consisting of all {finitely generated modules};
\item $\rpMod A$ consisting of all {finitely generated projective
  modules}.
\end{itemize}
Replacing ``right'' with ``left'' everywhere here, we obtain analogous
categories
$A \lMod$, $A \ldMod$, $A \lfMod$ and $A \lpMod$ of left modules.
There are contravariant equivalences
$$
\circledast:\rdMod A \rightarrow A\ldMod,
\qquad
\#:\rpMod A \rightarrow A\lpMod
$$
defined as follows: the {\em dual} $V^\circledast$ of a locally finite-dimensional right module $V$ is the left module
$\bigoplus_{x \in X} \Hom_{\k}(V 1_x, \k)$;
the {\em dual} $P^\#$ of a finitely generated projective right module $P$ is
the left module
$\Hom_A(P, A)$.
If $V$ and $P$ are left modules instead, their duals
${^\circledast} V$ and ${^\#}P$ are the right modules defined analogously.

\begin{remark}\label{cats}
The data of a locally unital algebra $A$ is the same as the data of a
small category $\A$ with object set $X$ and morphisms
$\Hom_{\A}(x,y) := 1_y A 1_x$. In this incarnation, locally unital
algebra homomorphisms correspond to functors.
A right $A$-module becomes a 
functor $\A^{\op} \rightarrow \Vec$,
and then a module homomorphism is a natural transformation of
functors.
We say $\A$ is {\em Noetherian} (resp. {\em Artinian}) if $A$ is
locally Noetherian (resp. locally Artinian)
in the sense already defined.
All of the other
notions introduced in this subsection can be
recast in this more categorical language too, as was done in \cite{M}. For example, 
the projective module $1_x A$ corresponds to the functor
$\Hom_{\A}(-,x):\A^{\op} \rightarrow \Vec$.
Then the Yoneda Lemma asserts that there is a fully faithful functor
$\A \rightarrow \rpMod A$
sending $x \in \ob \A$ to $1_x A$, and $a \in \Hom_{\A}(x,y)$
to the homomorphism $1_x A \rightarrow 1_y A$ defined by left
multiplication by $a \in 1_y A 1_x$.
This extends canonically to an equivalence of categories
\begin{equation}\label{yoneda2}
\dot \A \rightarrow \rpMod A,
\end{equation}
where $\dot \A$ denotes
the {\em additive Karoubi envelope} of $\A$, that is, the idempotent
completion of the additive envelope of $\A$.
\end{remark}

For locally unital algebras $A$ and $B$
with distinguished idempotents $(1_x)_{x \in X}$ and $(1_y)_{y \in
  Y}$, respectively,
an $(A,B)$-bimodule $M = \bigoplus_{x \in X, y \in Y} 1_x M 1_y$ determines
an adjoint pair $(T_M, H_M)$ of functors
\begin{align*}
T_M := - \otimes_A M&:\rMod A \rightarrow \rMod B,
\\H_M := \bigoplus_{x \in X} \Hom_B(1_x M, -)&:\rMod B \rightarrow
\rMod A.
\end{align*}
Here are a couple of useful facts about tensoring with bimodules.
First, there is a natural isomorphism
\begin{align}\label{see}
V \otimes_A \Hom_B(Q, M)
&\stackrel{\sim}{\rightarrow}
\Hom_B(Q, V \otimes_A M),
&v \otimes f &\mapsto (q \mapsto v \otimes f(q))\\\intertext{for all right $A$-modules $V$ and finitely
generated projective right $B$-modules $Q$; cf. \cite[20.10]{AF}.
Also, given another locally unital algebra $C$
and a right exact functor $E:\rMod B \rightarrow \rMod C$
commuting with direct sums,
$EM$ is an $(A,C)$-bimodule, and
there is a natural isomorphism of right $C$-modules}
\label{dee}
V \otimes_A E M &\stackrel{\sim}{\rightarrow}
E (V \otimes_A M),&
v \otimes n &\mapsto E(f_v)(n)
\end{align}
for all right $A$-modules $V$,
where $f_v:M \rightarrow V \otimes_A M$ is the right $C$-module
homomorphism $m \mapsto v \otimes m$. The proof of this involves
a reduction to the case $V = A$ using the Five Lemma.

\begin{definition}\label{ab}
We say that an $(A,B)$-bimodule $M$ is {\em left rigid} if it has a
left dual
in the 2-category of bimodules; see \cite[$\S$2.10]{EGNO} for our
conventions here.
It means that
there
exists a $(B,A)$-bimodule
$M^\#$, 
an $(A,A)$-bimodule
homomorphism
$\coev:A \rightarrow M \otimes_B M^\#$,
and
a $(B,B)$-bimodule homomorphism
$\ev:M^\# \otimes_A M \rightarrow B$, 
such that the compositions
\begin{align*}
M \stackrel{\operatorname{can}}{\longrightarrow}
A \otimes_A M \stackrel{\coev \otimes 1}{\longrightarrow} &M \otimes_B M^\# \otimes_A M
\stackrel{1 \otimes \ev}{\longrightarrow}  M \otimes_B B
\stackrel{\operatorname{can}}{\longrightarrow} M,\\
M^\# \stackrel{\operatorname{can}}{\longrightarrow}
M^\#\otimes_A A \stackrel{1 \otimes \coev}{\longrightarrow}& M^\#
\otimes_A M \otimes_B M^\#
\stackrel{\ev \otimes 1}{\longrightarrow}  B \otimes_B M^\#
\stackrel{\operatorname{can}}{\longrightarrow} M^\#
\end{align*}
are the identities. In other words,
$T_{M^\#}$ is right adjoint to $T_M$.
We say that $M$ is {\em right rigid} if it has a right dual,
i.e. there
exists a $(B,A)$-bimodule
${^\#}M$ such that
$T_{{^\#}\!M}$ is left adjoint to $T_M$.
Finally, we say $M$ is {\em rigid}
if it is both left and right rigid, and {\em sweet}\footnote{
The language ``sweet bimodule'' appears in \cite[$\S$2.6]{K}, but
(in view of Lemma~\ref{rigidity})
the ones defined there are just what we call rigid bimodules,
since it is not assumed that ${^\#}\!M\cong M^\#$.
Note though that 
the important {examples} constructed in \cite{K} do satisfy this extra
hypothesis, so they are sweet in our sense too.
} if in addition $M^{\#} \cong {^\#}M$ as $(B,A)$-bimodules.
\end{definition}

The following is essentially \cite[Ex. 2.10.16]{EGNO}.

\begin{lemma}\label{rigidity}
Let $A$ and $B$ be locally unital algebras
with distinguished idempotents $(1_x)_{x \in X}$ and $(1_y)_{y \in
  Y}$, respectively.
Let $M$ be an $(A,B)$-bimodule.
\begin{enumerate}
\item
The bimodule $M$ is left rigid if and only if 
$1_x M$ is finitely generated and
projective as a right $B$-module for each $x \in X$.
In that case, $M^\# \cong \bigoplus_{x \in X} (1_x M)^\#$.
\item
It is right rigid if and only if
$M 1_y$
is finitely generated and
projective as a left $A$-module for each $y \in Y$.
In that case, ${^\#}M \cong \bigoplus_{y \in Y} {^\#}(M1_y)$.
\end{enumerate}
\end{lemma}

\begin{proof}
(1)
Suppose that $M$ possesses a left dual $M^\#$.
Then $T_M$ has a right exact right adjoint $T_{M^\#}$, so
$T_M$ sends projectives to projectives.
Hence, $1_x M \cong T_M(1_x A)$ is projective for each $x \in X$.
Let $\coev(1_x) = \sum_{i=1}^{n_x} v_{x,i} \otimes f_{x,i}$
for $v_{x,i} \in 1_x M$ and $f_{x,i} \in (M^\#) 1_x$.
Then any $v \in 1_x M$ is equal to $(1 \otimes \ev)(v_{x,i} \otimes
f_{x,i} \otimes v) \in \sum_{i=1}^{n_i} v_{x,i} A$.
This shows that $1_x M$ is finitely generated.

Conversely, suppose that each $1_x M$ is finitely generated and projective as a
right module.
Then (\ref{see}) implies that $H_M \cong T_{M^\#}$ where
$M^\# := \bigoplus_{x \in X} (1_x M)^\#$.
Hence, we have constructed a bimodule $M^\#$ such that
$T_{M^\#}$ is right adjoint to $T_M$, proving that $M$ is left rigid.

(2) Similar (working with left modules instead of right ones).
\end{proof}

By a {\em projective generating family} for an Abelian category
$\C$, we mean
a small family $(P(x))_{x\in X}$ of compact\footnote{For 
categories of the form $\rMod A$ for some locally unital algebra
  $A$, a projective is compact if and only if it is finitely
  generated.}
 projective objects
such that for each $V \in\ob\C$ there is some $x \in X$ with
$\Hom_{\C}(P(x), V) \neq 0$.
Just like in \cite[Exercise 5.F]{F}, one can show that
an Abelian
category $\C$ is equivalent to $\rMod A$ for some locally
unital algebra $A$ if and only if $\C$ possesses arbitrary direct sums and
has a projective generating family; see also \cite[Theorem 3.1]{M}.
%\cite[Theorem 5.26]{kelly}.
We just need this in the following special case, which is the locally
unital analog of the classical Morita Theorem:

\begin{theorem}\label{mor}
Let $B$ be a locally unital algebra. Suppose that
$(P(x))_{x \in X}$ is a projective generating family
for $\rMod B$. Let
$$
A :=\bigoplus_{x,y \in X} \Hom_{B}(P(x), P(y)),
$$
viewed as a locally unital algebra with distinguished idempotents
$(1_x := 1_{P(x)})_{x \in X}$.
Let $P := \bigoplus_{x\in X} P(x)$, which
is an $(A,B)$-bimodule.
\begin{enumerate}
\item The functors 
$T_P = - \otimes_A P$
and
$H_P = \bigoplus_{x \in X} \Hom_{B}(1_x P, -)$ are quasi-inverse equivalences 
between the categories $\rMod A$ and $\rMod B$.
\item
We have that $H_P \cong T_Q$
where $Q := P^\#$.
\end{enumerate}
Thus, we have constructed a sweet $(A,B)$-bimodule $P$ and a sweet $(B,A)$-bimodule
$Q$ such that $P \otimes_B Q \cong A$ and $Q \otimes_A P \cong B$ as
bimodules.
\end{theorem}

\begin{proof}
The fact that $H_P \circ T_P \cong 1_{\rMod A}$ follows from (\ref{see}). 
Then one deduces that $T_P \circ H_P \cong 1_{\rMod B}$ too by a
standard argument; cf. 
\cite[22.2]{AF}.
Finally Lemma~\ref{rigidity}(1) implies that $H_P \cong T_Q$.
\end{proof}

\begin{corollary}\label{eagle}
For locally unital algebras $A$ and $B$, the following are equivalent:
\begin{enumerate}
\item the categories $\rMod A$ and $\rMod B$ are equivalent;
\item the categories $\rpMod A$ and $\rpMod B$ are equivalent;
\item the categories $A\lpMod$ and $B\lpMod$ are equivalent;
\item the categories $A\lMod$ and $B\lMod$ are equivalent.
\end{enumerate}
\end{corollary}

\begin{proof}
(1) $\Rightarrow$ (2). The restriction of an equivalence $\rMod A \rightarrow \rMod B$
gives an equivalence $\rpMod A \rightarrow \rpMod B$.

(2) $\Rightarrow$ (1). Let $F:\rpMod A \rightarrow \rpMod B$ be an
equivalence of categories. 
Let $(1_x)_{x \in X}$ be the distinguished idempotents in $A$.
Then
$(P(x) := F(1_x A))_{x \in X}$ 
is a projective generating family
for $\rMod B$ such that
$A \cong \bigoplus_{x, y \in X} \Hom_B(P(x), P(y))$.
Now apply Theorem~\ref{mor}.

(3) $\Leftrightarrow$ (4). This is the same as (1) $\Leftrightarrow$
(2) with $A$ and $B$ replaced by the opposite algebras.
 
(2) $\Leftrightarrow$ (3). This follows as $\rpMod A$ (resp. $\rpMod
B$) is contravariantly equivalent to $A \lpMod$ (resp. $B \lpMod$).
\end{proof}

Two locally unital algebras $A$ and $B$ are said to be {\em Morita
  equivalent} if the conditions of Corollary~\ref{eagle} are satisfied.
For example,
if $A$ is a contraction of $B$, then the categories $\rMod A$ and
$\rMod B$
are obviously isomorphic. Hence, $A$ and $B$ are Morita equivalent.
For another simple example, let $N$ be any (possibly infinite but small) set
and $M_N(\k)$ be the algebra of $N \times N$ matrices with entries in
$\k$, all but finitely many of which are zero. 
This is a locally unital algebra with distinguished idempotents
given by the diagonal matrix units $\{e_{i,i}\:|\:i \in N\}$.
Applying Theorem~\ref{mor}
with $B := \k, X := N$ and taking each $P(x)$ to be a copy of $\k$, 
we see that $M_N(\k)$ is Morita equivalent to the ground field $\k$.

\begin{remark}
Suppose that $\A$ and $\mathcal B$ are the categories
associated to locally unital algebras $A$ and $B$ as
in Remark~\ref{cats}.
We say that $\A$ and $\mathcal B$ are 
{\em Morita equivalent} if their additive Karoubi envelopes
$\dot \A$ and $\dot{\mathcal B}$ are
equivalent. In view of Corollary~\ref{eagle} and (\ref{yoneda2}), this
is equivalent to the algebras $A$ and $B$
being Morita equivalent as above. 
\end{remark}

The final theorem in this subsection 
is concerned with adjoint functors.
Again this is classical in the unital setting.

\begin{theorem}\label{name}
Let $B$, $P = \bigoplus_{x \in X} P(x)$ and $A$ be as in
Theorem~\ref{mor}, so
that  $H_P:\rMod B \rightarrow \rMod A$ is an equivalence of categories.
Suppose we are given a functor $E:\rMod B \rightarrow \rMod B$.
Then $E$ possesses a right adjoint if and only if it
is right exact
and commutes with direct sums.
In that case, let
$$
M := \bigoplus_{x, y \in X} \Hom_{B}(P(x), E P(y))
$$
viewed as an $(A,A)$-bimodule in the natural way.
Then the diagram
$$
\begin{CD}
\rMod B&@>H_P>>& \rMod A\\
@VE VV&&@VV T_M V\\
\rMod B&@>>H_{P}>& \rMod A\\
\end{CD}
$$
commutes up to a canonical isomorphism.
\end{theorem}

\begin{proof}
It is standard that functors possessing a right adjoint are right exact and
commute with direct sums. Conversely, suppose that $E$ is right exact
and commutes with direct sums.
Using
(\ref{dee}) then
 (\ref{see}),
we get that 
$$
H_{P} \circ E \circ T_P \cong H_{P} \circ T_{EP} \cong
T_{H_{P}(EP)},
$$ 
which is isomorphic to $T_M$ as $H_{P}(EP) \cong M$.
Thus $H_{P} \circ E \circ T_P \cong T_M$.
Composing on the right with the quasi-inverse $H_P$ of $T_P$,
we deduce that $H_{P} \circ E \cong T_M \circ H_P$.
This proves the final part of the theorem.
Hence
$E$ has a right adjoint as $T_M$ has the right adjoint $H_M$.
\end{proof}

\subsection{Finite-dimensional categories}\label{fc}

Let $A$ be a locally unital algebra with distinguished idempotents
$(1_x)_{x \in X}$. We assume in this subsection that
$A$ is {\em locally finite-dimensional}, i.e.
each subspace $1_y A 1_x$ is finite-dimensional.
Equivalently, the associated category $\A$
from Remark~\ref{cats} is a {\em finite-dimensional category}, i.e. it is a small $\k$-linear category all of whose
morphism spaces are finite-dimensional.
All of the right ideals $1_x A$ and the
left ideals $A1_x$ are locally finite-dimensional. Hence, so are their 
duals $(A 1_x)^\circledast$ and $(1_x A)^\circledast$.
Consequently, all finitely generated modules are locally finite-dimensional, as are all finitely cogenerated modules.

\begin{lemma}\label{munoz}
For a locally finite-dimensional locally unital algebra $A$, locally Artinian implies
locally Noetherian.
\end{lemma}

\begin{proof}
As $A 1_x$ satisfies DCC, its dual $(A 1_x)^\circledast$ satisfies
ACC. Hence, $(A 1_x)^\circledast$ is finitely generated, so satisfies
DCC.
Hence, $A 1_x$ satisfies ACC.
Similarly each $1_x A$ has ACC.
\end{proof}

Here are various other basic facts about modules over a locally finite-dimensional locally unital algebra $A$.
For the most part, these are proved by mimicking
the usual proofs in the setting of finite-dimensional algebras, so we
will be quite brief. 
Fix representatives
$\{L(b)\:|\:b\in \B\}$ for the isomorphism classes of irreducible
right $A$-modules\footnote{We use the notation $\B$ here as
  ultimately this set will carry a crystal structure; cf. Definition~\ref{nc}.}.

 \begin{enumerate}
\item[(L1)] If $V$ is finitely generated (resp. locally finite-dimensional) and $W$ is locally finite-dimensional (resp. finitely cogenerated)
then 
$\Hom_A(V,W)$ is finite-dimensional.
\item[(L2)]
{\em Schur's Lemma} holds:
$\End_A(L(b)) \cong \k$
for each $b \in \B$.
\item[(L3)]
Any
finitely generated (resp. finitely cogenerated) module satisfies the
{\em Krull-Schmidt Theorem}.
\item[(L4)]
The category $\rMod A$ is a Grothendieck category, i.e. it is Abelian,
it possesses arbitrary direct sums,
direct
limits of short exact sequences are exact,
and there is a generator (namely,
the regular module $A$ itself).
Hence, by the general theory of Grothendieck categories,
every $A$-module has an injective hull;
moreover, a module $V$
is finitely cogenerated if and only if
its {\em socle} $\soc(V)$, i.e. the sum of the irreducible submodules of $V$,
is an essential submodule of $V$ of finite length.
\item[(L5)] For $b\in \B$, let 
$A_b := A / \operatorname{Ann}_A(L(b))$, which is a locally unital
algebra with distinguished idempotents $(1_x)_{x \in X}$ that are the images of the ones
in $A$.
Also let $M_b := 
\bigoplus_{x,y\in X} \Hom_\k(L(b) 1_x, L(b)1_y)$, 
viewed as a locally unital algebra with multiplication that is the
opposite of composition. Note that $M_b$ is a contraction of
$M_N(\k)$ where
$N := \{(x, i)\:|\:x\in X, 1 \leq i \leq \dim L(b)1_x\}$.
Then the natural right action of $A$ on $L(b)$ induces a
locally unital isomorphism 
$A_b \stackrel{\sim}{\rightarrow} M_b$.
Hence, $A_b$ is a contraction of a locally unital matrix algebra.
Next, let $J := \bigcap_{b \in \B}
\operatorname{Ann}_A(L(b))$ be the {\em Jacobson radical} of $A$.
The map
$$
A / J \rightarrow \bigoplus_{b \in \B} A_b,\qquad
a+J \mapsto (a + \operatorname{Ann}_A(L(b)))_{b \in \B}
$$
is a well-defined algebra isomorphism.
Hence, $A / J$ is a contraction of a (possibly infinite) direct sum of
locally unital matrix algebras.
It follows
that $A / J$ is {\em semisimple}, i.e. every $A / J$-module is
completely reducible. Moreover, $J$ is the smallest two-sided ideal of $A$ with this property.
\item[(L6)]
For a right $A$-module $V$, its {\em radical} $\operatorname{rad}(V) := VJ$
is the intersection of all of its proper maximal submodules; its {\em head} 
$\operatorname{hd}(V) := V / \operatorname{rad}(V)$
is its largest completely reducible quotient.
Applying $\circledast$ to the statements made in (L4), we deduce that every finitely
generated $A$-module has a projective cover; moreover,
 $V$ is finitely generated if and only if $\operatorname{rad}(V)$ is a
superfluous submodule and $\operatorname{hd}(V)$ is of finite length\footnote{One can give a direct proof of this using the fact that $J$ is
{\em locally nilpotent} in the sense that $eJe$ is a nilpotent ideal
of the finite-dimensional algebra 
$eAe$ for any idempotent $e \in A$.}.

\item[(L7)]
Let $P(b)$
be a projective cover of $L(b)$.
For any right $A$-module $V$,
 the {\em composition multiplicity}
$[V:L(b)]$ is defined as usual to be
the supremum
of the multiplicities $\#\{i=1,\dots,n\:|\:V_i / V_{i-1}\cong L(b)\}$
taken over all filtrations $0 = V_0 < \cdots <V_n=V$ and all $n \in \N$.
By Schur's Lemma, we have that
$[V:L(b)] = 
\dim \Hom_A(P(b), V) \in \N \cup \{\infty\}.$
Noting that $\Hom_A(1_x A, L(b)) \cong L(b) 1_x$, the projective
module $1_x A$ decomposes as
$$
1_x A \cong \bigoplus_{b\in \B} P(b)^{\bigoplus \dim L(b) 1_x}.
$$
All but finitely many summands on the right hand side are zero,
so there are only finitely many $b \in \B$ such that
$L(b) 1_x \neq 0$.
Hence, for any $V$, we have that
$$
\dim V 1_x = \sum_{b \in \B} [V:L(b)] \dim L(b) 1_x.
$$
In particular, we get from this that $V$ is locally finite-dimensional if and only if $[V:L(b)] < \infty$
for all $b \in \B$.
\item[(L8)] 
Given $b \in \B$, we choose $x \in X$ so that $L(b) 1_x \neq 0$.
The decomposition of $1_x A$ derived in (L7) 
implies
that there exists a primitive
idempotent
$1_b \in 1_x A 1_x$ such that $P(b)\cong 1_b A$.
Then $I(b) := (A 1_b)^\circledast$
is an injective hull of $L(b)$. For any $V$, we have that
$[V:L(b)] = \dim \Hom_A(V, I(b)).$
\item[(L9)]
Suppose we are given a family $(A_i)_{i \in I}$ of locally finite-dimensional locally unital algebras, with the distinguished idempotents
in $A_i$ indexed by $X_i$.
Then $A := \bigoplus_{i \in I} A_i$ is a locally finite-dimensonal
locally unital algebra with distinguished idempotents 
indexed by $X := \bigsqcup_{i \in I} X_i$.
Moreover there is an equivalence of categories
$\prod_{i \in I} \rMod A_i \rightarrow \rMod A$
which sends an object $(V_i)_{i \in I}$ 
of
$\prod_{i \in I} \rMod A_i$
to the $A$-module 
$\bigoplus_{i \in I} V_i$.
\item[(L10)]
Suppose 
$\B = \bigsqcup_{i \in I} \B_i$
is a partition
such that
that $\Hom_A(P(b), P(c)) = 0$ for all $b \in \B_i, c \in \B_j$ and $i
\neq j$.
For $x \in X$, 
we can write
$1_x$ uniquely as a sum of mutually orthogonal idempotents $1_x = \sum_{i \in I} 1_{(x,i)}$ 
so that $1_{(x,i)} A \cong \bigoplus_{b \in \B_i} P(b)^{\oplus \dim
  L(b) 1_x}$.
Let $A_i := \bigoplus_{x,y \in X} 1_{(y,i)} A 1_{(x,i)}$, which is
itself a locally unital algebra with idempotents
$(1_{(x,i)})_{x \in X}$ and irreducibles represented by
$\{L(b)\:|\:b \in \B_i\}$.
Then we have that $A = \bigoplus_{i \in I} A_i$. Hence,
$A$ is a contraction of $\bigoplus_{i \in I} A_i$ .
If none of the $\B_i$ can be partitioned any further in this way, we
call this the {\em block decomposition} of $A$, and refer to indecomposable
subalgebras $A_i$ as {\em blocks}.
\end{enumerate}

\subsection{Locally Schurian categories}
Later in the article, we will be interested in categorical
actions on categories of the following form:

\begin{definition}\label{kitty}
We say that a category $\C$ is {\em locally Schurian} if it is
equivalent to $\rMod A$ for some locally finite-dimensional locally
unital algebra $A$.
\end{definition}

Given a locally Schurian category $\C$,
Theorem~\ref{mor} gives a recipe for constructing
a locally finite-dimensional locally unital algebra $A$ such that $\C$ is
equivalent to $\rMod A$: choose a projective generating
family $(P(x))_{x \in X}$ 
for $\C$; set 
\begin{equation}\label{generalone}
A := \bigoplus_{x,y \in X}
\Hom_{\C}(P(x), P(y))
\end{equation} 
viewed as a locally unital algebra with
distinguished
idempotents $(1_x := 1_{P(x)})_{x \in X}$; then the functor
\begin{equation}\label{H}
H := \bigoplus_{x\in X} \Hom_{\C}(P(x), -):\C
\rightarrow \rMod A
\end{equation}
is an equivalence of categories.
Often it is convenient to proceed by choosing representatives
$\{L(b)\:|\:b\in\B\}$
for the isomorphism classes of irreducible object in $\C$
and letting $P(b)$ (resp.\ $I(b)$) be a projective cover (resp.\ an
injective hull) of $L(b)$.
Then we call $(P(b))_{b \in \B}$ a {\em minimal} projective generating family for
$\C$, and 
$\C$ is equivalent to $\rMod B$ where
\begin{equation}\label{basicone}
B := \bigoplus_{b,c \in \B} \Hom_{\C}(P(b), P(c)).
\end{equation}
This a {\em basic} locally unital algebra: the irreducible
$B$-modules are all one dimensional.

An object $V$ in $\C$ is {\em locally finite-dimensional}
if and only if all its composition multiplicities are finite. 
We let $\pC \subseteq \fC \subseteq \dC$ be the full subcategories of $\C$
consisting of finitely generated projective, finitely generated, and
locally finite-dimensional objects;
for $A$ as in (\ref{generalone}),
these are equivalent to the subcategories $\rpMod A,
\rfMod A$ and $\rdMod A$ of $\rMod A$.

We say that $\C$ is {\em Noetherian} if all finitely generated
(resp. cogenerated) objects
satisfy
ACC (resp. DCC); equivalently, the algebra $A$ from (\ref{generalone}) is locally Noetherian.
We say that $\C$ is {\em Artinian} if all finitely generated
(resp. cogenerated)
objects satisfy DCC (resp. ACC); equivalently, the algebra $A$ is
locally Artinian.
We say that $\C$ is {\em finite} if there are only finitely many
isomorphism classes of irreducible object; equivalently, the basic algebra
$B$ from (\ref{basicone}) is finite-dimensional.
By Lemma~\ref{munoz}, Artinian implies Noetherian. 

If $\C$ is Artinian then
$\fC$ is a {\em Schurian category} in the sense of
\cite[$\S$2.1]{BLW}:
it is
Abelian, all objects are of finite length, there are enough projectives and injectives,
and 
the endomorphism
algebras of the irreducible objects are one dimensional.
Moreover, for
$V \in \ob \C$,
the following are equivalent:
\begin{itemize}
\item
$V$ is finitely
generated; 
\item
$V$ is finitely cogenerated; 
\item 
$V$ has finite length.
\end{itemize}

\subsection{Sweet endofunctors}
We will consider categorical actions on locally Schurian categories
involving functors of the following form:

\begin{definition}\label{swe}
Let $\C$ be a locally Schurian category.
We say that an endofunctor $E$ of $\C$
is {\em sweet} if there is an endofunctor $F$
which is biadjoint to $E$.
\end{definition}

Recalling the definition of sweet bimodule from Definition~\ref{ab},
the following theorem gives an algebraic characterization of 
sweet endofunctors.

\begin{theorem}\label{sef}
Let $E$ be an endofunctor of a locally Schurian category
$\C$.
Fix an equivalence $H:\C \rightarrow \rMod A$ as in
(\ref{H}).
Then $E$ is sweet if and only if there is a sweet bimodule $M$ such that
 $H \circ E \cong T_M \circ H$. In that case, we have that
\begin{equation}\label{as}
M \cong \bigoplus_{x,y \in X} \Hom_{\C}(P(x), E P(y)).
\end{equation}
Moreover, $E$ is exact, continuous and cocontinuous, and it preserves the
sets of 
locally finite-dimensional, finitely generated,
finitely cogenerated, projective and injective objects.
\end{theorem}

\begin{proof}
If $M$ is a sweet bimodule such that $H \circ E \cong T_M
\circ H$, then $E$ is sweet since
$T_M$ is a sweet endofunctor of $\rMod A$
and $H$ is an equivalence.
Conversely, suppose that $E$ possesses a
biadjoint $F$. 
Theorem~\ref{name} shows that
$H \circ E \cong T_M \circ H$ for $M$ as in (\ref{as});
similarly $H \circ F \cong T_N \circ H$ for some bimodule $N$.
Then $T_M$ and $T_N$ are biadjoint. Hence, $N$ is both right and left dual to
$M$, i.e. $M$ is a sweet bimodule.
It follows at once that $E$ and $F$ both 
send finitely generated objects to finitely generated
objects, as $T_M$ and $T_N$ clearly do.

Since $E$ has a biadjoint,
it is exact, continuous and cocontinous. Also $F$ is exact, 
so $E$ preserves projectives and injectives. 
To see that $E$ preserves locally finite-dimensional objects, 
we observe for locally finite-dimensional $V$ that 
$$
\dim \Hom_{\C}(P(b), EV) =
\dim \Hom_{\C}(FP(b), V) < \infty$$ for all $b\in \B$.
Similarly $F$ preserves locally finite-dimensional objects.
Finally, 
since $E$ preserves finitely generated objects, we have that
$$
\dim \Hom_{\C}(E P(b), L(c)) =\dim\Hom_{\C}(P(b), F L(c))
= [F L(c):L(b)]$$
 is zero for all but finitely many $c$.
Hence, $$\dim \Hom_{\C}(L(c), E I(b)) =\dim\Hom_{\C}(F L(c),
I(b))
= [F L(c):L(b)]$$ is zero for all but finitely many $c$.
This implies that $E I(b)$ is a finite direct sum of $I(c)$'s.
Hence, $E$ preserves finitely cogenerated objects, and similarly for $F$.
\end{proof}

\begin{lemma}\label{swee}
Suppose that $F$ and $G$ are sweet endofunctors of a locally Schurian
category $\C$. Let $\eta:F \Rightarrow G$ be a natural
transformation.
If $\eta_L:FL \rightarrow GL$ is an isomorphism 
for each irreducible object $L \in \ob \C$, then $\eta$ is an isomorphism.
\end{lemma}

\begin{proof}
We may assume that $\C = \rMod B$ for a basic locally unital
algebra $B$. This means that the irreducible $B$-modules are parametrized by the
same set $\B$ as indexes its distinguished idempotents, 
and $[V:L(b)] = \dim \Hom_B(P(b), V) = \dim V 1_b$ for each $b \in
\B$.
The main step is to prove that $\eta_V:FV \rightarrow GV$ is an
isomorphism for each locally finite-dimensional $B$-module $V$.
Assuming this, the lemma may be deduced as follows: given any
$B$-module $V$, consider
a two-step projective resolution $Q
\rightarrow P \rightarrow V \rightarrow 0$;
since $P$ and $Q$ are direct sums of finitely generated projectives, and
$F$ and $G$ commute with arbitrary direct sums, the locally finite-dimensional result shows that
$\eta_P$ and
$\eta_Q$ are isomorphisms. Hence,
$\eta_V$ is 
an isomorphism too by the Five Lemma.

So now suppose that $V$ is locally finite-dimensional.
It suffices to show for each fixed $a \in \B$ that the restriction of $\eta_V$ defines a linear isomorphism
between 
$(FV) 1_a$ and $(GV) 1_a$.
Let $$
X := \{b \in \B\:|\:(F L(b)) 1_a \neq 0\}=
\{b \in \B\:|\:\Hom_B(P(a), F L(b)) \neq
0\}.
$$
Fixing a left adjoint $E$ to $F$, we have 
that
$X = \{b \in \B\:|\:\Hom_B(E P(a), L(b)) \neq 0\}$.
Since $E P(a)$ is a finitely generated projective, we deduce from this
that $X$ is finite; moreover, $E P(a)$ a direct sum of indecomposable
projectives of the form 
$P(b)$ for $b \in X$. 
Now we proceed by induction on
$n := \sum_{b \in X} \dim V 1_b \in \N$. 
In case $n = 0$, we have that $\Hom_B(P(a), F V) \cong \Hom_B(E P(a),
V) = 0$. Hence, $(FV) 1_a
= 0$; similarly, $(GV) 1_a = 0$. So the desired conclusion that $(FV) 1_a \cong (GV) 1_a$
is trivial.
For the induction step, we take a vector $0 \neq v \in V 1_b$ for some
$b \in X$.
Let $W := v B$ and $W' := \rad(W)$, so that we have a filtration
$0 \leq W' < W \leq V$ with $W / W' \cong L(b)$.
By induction, $\eta_{W'}$ and $\eta_{V / W}$ 
restrict to 
isomorphisms $(F W') 1_a \stackrel{\sim}{\rightarrow} (GW') 1_a$
and $(FV / FW) 1_a \stackrel{\sim}{\rightarrow} (GV / GW) 1_a$.
Also $\eta_{W / W'}$ is an isomorphism as $W / W'$ is irreducible.
Hence, $\eta_V$ defines an isomorphism $(FV)
1_a\stackrel{\sim}{\rightarrow} (GV) 1_a$ as required.
\end{proof}

\subsection{Serre quotients}\label{sq}
Finally in this section, we review briefly the standard notions of {\em Serre subcategory} and {\em
  Serre quotient category} in the setting of locally Schurian
categories.
Let $\C$ be a locally Schurian category with irreducible
objects represented by
$\{L(b)\:|\:b \in \B\}$ as above.
Let $\B'$ be any subset of $\B$ and $\C'$
be the full subcategory of $\C$ consisting of all the objects whose
irreducible subquotients are isomorphic to $L(b)$ for $b \in \B'$. 
It is a Serre subcategory of $\C$, i.e. it is
closed under taking subobjects, quotients and extensions.
Moreover it is itself a locally Schurian category with irreducible
objects represented by $\{L(b)\:|\:b \in \B'\}$. To see this, 
define $B$
according to (\ref{basicone}) so that
$\C$ is equivalent to $\rMod B$.
Then $\C'$ is equivalent to $\rMod B'$ where $B'$ is the
quotient of $B$ by the two-sided ideal generated by the idempotents
$\{1_b\:|\:b \in \B\setminus\B'\}$.
The exact inclusion functor $\iota:\C' \rightarrow \C$
corresponds to the natural inflation functor from $\rMod B'$ to
$\rMod B$, and it has a left adjoint $\iota^!$ (resp.\ a right
adjoint $\iota^*$) which sends an object to its largest quotient
(resp.\ subobject) belonging to $\C'$.
We have that $\iota^! \circ \iota \cong 1_{\C'} \cong \iota^* \circ \iota$.

The {\em Serre quotient}\, $\C /
\C'$ is an Abelian category 
equipped with an exact
{\em quotient functor}
$\pi:\C \rightarrow \C / \C'$ 
satisfying the
following universal property:
if $F:\C \rightarrow \mathcal D$ is any exact functor to an
Abelian category $\mathcal D$ then there is a unique 
functor $\bar F: \C / \C' \rightarrow D$
such that $F = \bar F \circ \pi$.
In fact $\C / \C'$ is another locally Schurian
category 
with irreducibles represented
by $\{\pi L(b)\:|\:b \in \B \setminus \B'\}$.
Again this is easy to see in terms of the algebra $B$: the category
$\C / \C'$ is equivalent to modules over the algebra 
$e B e := \bigoplus_{b, c \in \B \setminus \B'} 1_b B 1_c$. The quotient functor
$\pi$ corresponds to the obvious truncation 
functor $e:\rMod B \rightarrow \rMod eBe$
sending a $B$-module $V$ to $Ve := \bigoplus_{b \in\B \setminus\B'} V
1_b$. Consequently,
$\pi$ has a left adjoint $\pi^!:\C / \C'
\rightarrow \C$ and a right adjoint
$\pi^*:\C / \C'$ corresponding to the functors
$-\otimes_{eBe} eB$ and $\bigoplus_{b \in \B\setminus\B'} \Hom_{eBe}(B 1_b,-)$,
respectively.
We have that $\pi \circ \pi^! \cong 1_{\C / \C'} \cong \pi \circ
\pi^*$.

\begin{lemma}\label{qlem}
In the above setup, assume that we are given $V, W \in \ob \C$ 
such that $V$ is finitely generated, $W$ is finitely cogenerated, and
all constituents of $\head(V)$ and $\soc(W)$ are of the form $L(b)$ for
$b \in \B \setminus \B'$.
Then the functor $\pi$ induces an isomorphism
$$
\Hom_{\C}(V,W) \stackrel{\sim}{\rightarrow}
\Hom_{\C / \C'}(\pi V, \pi W).
$$
\end{lemma}

\begin{proof}
The counit of adjunction defines a morphism
$f: \pi^! \pi V \rightarrow V$. By the assumptions on $V$, $f$ is an epimorphism.
Moreover $f$ becomes an isomorphism on applying $\pi$, hence $\ker f$
belongs to $\C'$.
Using also the assumptions on $W$, we deduce that
$\Hom_{\C / \C'}(\pi V, \pi W) \cong 
\Hom_{\C}(\pi^! \pi V, W) \cong \Hom_\C(V,W)$.
\end{proof}

\section{Kac-Moody 2-categories}

In this section, we review Rouquier's definition of 
Kac-Moody 2-category from \cite{Rou}. Then, following \cite{B}, we explain the relationship
between this and the 2-category introduced by Khovanov and Lauda in
\cite{KL3},
and discuss the graded version.
Note our exposition uses a slightly different normalization for the second
adjunction compared to \cite{B}
based on the idea of \cite{BHLW}.

\subsection{Kac-Moody data}\label{km}
Let $I$ be a finite index set\footnote{With minor adjustments to the
  basic definitions, everything here 
can be extended to infinite $I$ too; type $A_\infty$ is particularly important in
  applications.} and
$A = (-d_{ij})_{i,j \in I}$ be a generalized Cartan matrix,
so $d_{ii} = -2$, $d_{ij} \geq 0$ for $i \neq j$, and $d_{ij} = 0
\Leftrightarrow d_{ji} = 0$. 
We assume that $A$ is {\em symmetrizable}, so that there exist 
positive integers $(d_i)_{i \in I}$ such that $d_i d_{ij} = d_j
d_{ji}$ for all $i,j \in I$.
Pick a finite-dimensional complex vector space $\mathfrak{h}$
and
linearly independent subsets
$\{\alpha_i\:|\:i \in I\}$ and $\{h_i\:|\:i \in I\}$ of
$\mathfrak{h}^*$ and $\mathfrak{h}$, respectively,
such that 
$\langle h_i, \alpha_j\rangle = -d_{ij}$ for all $i,j \in I$.
Let
\begin{align*}
P &:= \{\lambda \in \mathfrak{h}^*\:|\:\langle h_i, \lambda \rangle \in
\Z \text{ for all }i \in I\},&
Q &:= \bigoplus_{i \in I} \Z \alpha_i,\\
P^+ &:= \{\lambda \in \mathfrak{h}^*\:|\:\langle h_i,\lambda \rangle \in
  \N\text{ for all }i \in I\},
&
Q^+ &:= \bigoplus_{i \in I} \N \alpha_i.
\end{align*}
We refer to $P$ and $Q$ as the {\em weight lattice} and the {\em root
  lattice}, respectively.
We view $P$ as an interval-finite poset via the usual {\em dominance ordering}:
$\lambda \leq \mu \Leftrightarrow \mu - \lambda \in Q^+$.

Let $\g$ be the associated {\em Kac-Moody algebra}
with {Cartan subalgebra} $\mathfrak{h}$.
Thus, $\g$ is the Lie algebra generated by 
$\mathfrak{h}$ and elements $e_i, f_i\:(i \in I)$
subject to the usual Serre relations:
for $h,h' \in \mathfrak{h}$ and $i,j \in I$ we have that
\begin{align}\label{weyl}
[h,h'] &= 0,
&[e_i,f_j]&=\delta_{i,j} h_i,\\
[h,e_i] &= \langle h,\alpha_i \rangle e_i,
&
[h,f_i] &= -\langle h,\alpha_i \rangle f_i,\\
(\operatorname{ad}\, e_i)^{d_{ij}+1} (e_j)&= 0,
&
(\operatorname{ad}\, f_i)^{d_{ij}+1} (f_j)&= 0.\label{serre}
\end{align}
Let $U(\g)$ be its {\em universal enveloping algebra}.
Actually it is often more convenient to work with the idempotented version $\dot{U}(\g)$
of $U(\g)$, which is a certain locally unital
algebra with 
distinguished idempotents $(1_\lambda)_{\lambda \in P}$. It is defined by
analogy with \cite[$\S$23.1]{Lubook} (which treats the quantum case).
As well as being an algebra, $\dot U(\g)$ is a $(U(\g),U(\g))$-bimodule with 
$h 1_\lambda = \langle h,\lambda \rangle 1_\lambda = 1_\lambda h$, $e_i
1_\lambda= 1_{\lambda+\alpha_i} e_i$ and $f_i 1_\lambda =
1_{\lambda-\alpha_i} f_i$. The elements $\{1_\lambda, e_i 1_\lambda, f_i
1_\lambda\:|\:i \in I, \lambda \in P\}$
generate $\dot{U}(\g)$
subject to the relations derived from (\ref{weyl})--(\ref{serre}).
Weight modules for $\g$, i.e. $\g$-modules $V$ such that
$V = \bigoplus_{\lambda \in P} V_\lambda$, are just the same as (locally unital)
$\dot{U}(\g)$-modules.

By an {\em upper} (resp.\ {\em lower}) {\em integrable module},
we mean a weight module with finite-dimensional weight spaces
all of whose weights lie in a finite union of sets of the form
$\lambda - Q^+$ (resp. $\lambda+Q^+$). By the classical theory from
\cite[Chapters 9--10]{Kac},
a $\g$-module is upper integrable if and only if it is isomorphic to a
finite direct sum of
the irreducible 
modules\footnote{All of these assertions depend on the assumption that $A$ is
symmetrizable.}
\begin{equation}\label{hi}
L(\kappa) := \dot{U}(\g) 1_\kappa / \langle e_i 1_\kappa , f_i^{1+\langle
  h_i,\kappa\rangle} 1_\kappa \:|\: i \in I \rangle
\end{equation}
for $\kappa \in P^+$;
these are the {\em integrable highest weight modules}.
Similarly, a $\g$-module is lower integrable if and only if it is a
finite direct sum of the (irreducible) {\em integrable lowest weight modules}
\begin{equation}\label{lo}
L'(\kappa')
:= \dot{U}(\g) 1_{\kappa'}  / \langle e_i^{1-\langle h_i,\kappa'\rangle}
1_{\kappa'}, f_i 1_{\kappa'}\:|\:i
\in I \rangle.
\end{equation}
for $\kappa' \in -P^+$.
Generalizing (\ref{hi})--(\ref{lo}), we let
\begin{equation}\label{mid}
L(\kappa'|\kappa) := \dot{U}(\g) 1_{\kappa+\kappa'} / \langle
e_i^{1-\langle h_i, \kappa' \rangle} 1_{\kappa+\kappa'},
f_i^{1+\langle h_i, \kappa \rangle} 1_{\kappa+\kappa'}\:|\:i \in
I\rangle
\end{equation}
for $\kappa \in P^+$ and $\kappa' \in -P^+$.
These modules are not so well studied, but they play an important role
in Lusztig's construction of canonical bases for $\dot U(\g)$
from \cite[Part IV]{Lubook}.
They are integrable modules but they may have infinite-dimensional weight
spaces outside of finite type, so that they are neither upper nor lower integrable.
The next lemma is the classical counterpart of
\cite[Proposition 23.3.6]{Lubook}.

\begin{lemma}\label{count}
There is a $\g$-module isomorphism
$L(\kappa'|\kappa) \stackrel{\sim}{\rightarrow} L(\kappa') \otimes
L(\kappa)$ such that $u \bar 1_{\kappa+\kappa'} \mapsto u(\bar 1_{\kappa'}
\otimes \bar 1_\kappa)$ for $u \in \dot U(\g)$.
\end{lemma}

The following lemma about annihilators will be useful later on.

\begin{lemma}\label{faith}
We have that 
\begin{align*}
\displaystyle\bigcap_{\substack{\kappa \in P^+ \\ \kappa' \in -P^+}} \operatorname{Ann}_{\dot
  U(\g)} (L(\kappa'|\kappa)) &= \{0\}.\\\intertext{Moreover if $\g$ is
  of finite type then}
\displaystyle\bigcap_{\lambda \in P^+} \operatorname{Ann}_{\dot
  U(\g)} (L(\lambda)) &= \{0\}.
\end{align*}
\end{lemma}

\begin{proof}
The proof of the first statement reduces using Lemma~\ref{count} to
checking that the maps 
\begin{align*}
U^-(\g) \rightarrow \bigoplus_{\kappa \in P^+}
L(\kappa), \qquad
&u \mapsto (u \bar 1_\kappa)_{\kappa \in P^+},\\
U^+(\g) \rightarrow \bigoplus_{\kappa' \in -P^+}
L'(\kappa'), \qquad
&u \mapsto (u \bar 1_{\kappa'})_{\kappa' \in -P^+}
\end{align*}
are injective, where $U^{\pm}(\g)$ are the positive and negative parts
of $U(\g)$ generated by the $e_i$ and $f_i$, respectively.
These are well-known facts; e.g. they can be deduced in a
non-classical way from \cite[Proposition 19.3.7]{Lubook}.
The second statement (which is even better known) follows from the
first since each $L(\kappa'|\kappa)$ is a finite direct sum of
$L(\lambda)$'s in view of Lemma~\ref{count} and complete reducibility.
\end{proof}

\subsection{Strict 2-categories}\label{S2c}
Let $\Cat$ be the category of (small) $\k$-linear categories and
$\k$-linear functors.
It is a monoidal category with tensor functor
$\boxtimes:\Cat \times \Cat \rightarrow \Cat$
defined on objects ($=$ categories) by letting $\C \boxtimes \C'$ be the category with objects
that are pairs $(\lambda,\lambda') \in \ob \C \times \ob \C'$,
morphisms 
$$
\Hom_{\C \boxtimes \C'}((\lambda,\lambda'),
(\mu,\mu')) := 
\Hom_\C(\lambda,\mu) \otimes_\k \Hom_\C(\lambda',\mu'),
$$
and composition law defined by $(g \otimes g') \circ (f \otimes f') := (g \circ f)
\otimes (g' \circ f')$.
The definition of $\boxtimes$ on morphisms ($=$ functors) is obvious:
$F \boxtimes F'$ is the functor that sends $(\lambda,\lambda')\mapsto
(F\lambda, F'\lambda')$ and $g \otimes g' \mapsto F g \otimes F' g'$.

\begin{definition}\label{strict2cat}
A {\em strict 2-category} is a category enriched in 
$\Cat$.
Thus, for objects $\lambda,\mu$ in a strict 2-category
$\CC$,
there is given a
category $\mathcal{H}om_{\CC}(\lambda,\mu)$ of
morphisms from $\lambda$ to $\mu$, whose objects $F,G$
are the {\em 1-morphisms} of $\CC$, and whose morphisms
$x:F \Rightarrow G$ are
the {\em 2-morphisms} of $\CC$.
\end{definition}

For example, $\Cat$ can be viewed as a strict 2-category $\CCat$
with 2-morphisms being natural transformations.

Given a strict 2-category $\CC$, we use the shorthand $\Hom_{\CC}(F,G)$ for the vector space
$\Hom_{\mathcal{H}om_{\CC}(\lambda,\mu)}(F,G)$ of all 2-morphisms $x:F
\Rightarrow G$.
Let us also briefly recall the ``string calculus'' for 2-morphisms in 
$\CC$;
e.g. see \cite[$\S$2]{Lauda}. We represent a 2-morphism
$x \in \Hom_{\CC}(F,G)$
by the picture
$$
\mathord{
\begin{tikzpicture}[baseline = 0]
	\draw[-,thick,darkred] (0.08,-.4) to (0.08,-.13);
	\draw[-,thick,darkred] (0.08,.4) to (0.08,.13);
      \draw[thick,darkred] (0.08,0) circle (4pt);
   \node at (0.08,0) {\color{darkred}$\scriptstyle{x}$};
   \node at (0.08,-.53) {$\scriptstyle{F}$};
   \node at (0.52,0) {$\scriptstyle{\lambda}.$};
   \node at (-0.32,0) {$\scriptstyle{\mu}$};
   \node at (0.08,.53) {$\scriptstyle{G}$};
\end{tikzpicture}
}
$$
The vertical composition $y \circ x$ of $x$ with another 2-morphism
$y \in \Hom_{\CC}(G,H)$ is obtained by vertically stacking
pictures:
$$
\mathord{
\begin{tikzpicture}[baseline = 0]
	\draw[-,thick,darkred] (0.08,-.4) to (0.08,-.13);
	\draw[-,thick,darkred] (0.08,.4) to (0.08,.13);
	\draw[-,thick,darkred] (0.08,.87) to (0.08,.6);
	\draw[-,thick,darkred] (0.08,1.37) to (0.08,1.13);
   \node at (0.08,1.01) {\color{darkred}$\scriptstyle{y}$};
      \draw[thick,darkred] (0.08,1) circle (4pt);
      \draw[thick,darkred] (0.08,0) circle (4pt);
   \node at (0.08,0) {\color{darkred}$\scriptstyle{x}$};
   \node at (0.08,-.53) {$\scriptstyle{F}$};
   \node at (0.62,.5) {$\scriptstyle{\lambda}.$};
   \node at (-0.42,0.5) {$\scriptstyle{\mu}$};
   \node at (0.08,1.5) {$\scriptstyle{H}$};
   \node at (0.08,.5) {$\scriptstyle{G}$};
\end{tikzpicture}
}
$$
Given 2-morphisms 
$x:F \Rightarrow H,
y:G \Rightarrow K$ between 1-morphisms $F, H:\lambda \rightarrow\mu,
G,K:\mu \rightarrow\nu$,
we denote their horizontal composition by $yx:GF\Rightarrow KH$,
and represent it by horizontally stacking pictures:
$$
\mathord{
\begin{tikzpicture}[baseline = 0]
	\draw[-,thick,darkred] (0.08,-.4) to (0.08,-.13);
	\draw[-,thick,darkred] (0.08,.4) to (0.08,.13);
      \draw[thick,darkred] (0.08,0) circle (4pt);
   \node at (0.08,0) {\color{darkred}$\scriptstyle{x}$};
   \node at (0.08,-.53) {$\scriptstyle{F}$};
   \node at (0.58,0) {$\scriptstyle{\lambda}.$};
   \node at (-0.37,0) {$\scriptstyle{\mu}$};
   \node at (0.08,.53) {$\scriptstyle{H}$};
	\draw[-,thick,darkred] (-.8,-.4) to (-.8,-.13);
	\draw[-,thick,darkred] (-.8,.4) to (-.8,.13);
      \draw[thick,darkred] (-.8,0) circle (4pt);
   \node at (-.8,0) {\color{darkred}$\scriptstyle{y}$};
   \node at (-.8,-.53) {$\scriptstyle{G}$};
   \node at (-1.22,0) {$\scriptstyle{\nu}$};
   \node at (-.8,.53) {$\scriptstyle{K}$};
\end{tikzpicture}
}
$$
When confusion seems unlikely, we will use the same notation for a $1$-morphism $F$
as for its identity $2$-morphism. 
With this convention, 
we have that $yH \circ Gx = yx = Kx \circ yF$, or in pictures:
$$
\mathord{
\begin{tikzpicture}[baseline = 0]
   \node at (0.08,-.53) {$\scriptstyle{F}$};
   \node at (0.58,0) {$\scriptstyle{\lambda}.$};
   \node at (-0.37,0) {$\scriptstyle{\mu}$};
   \node at (0.08,.53) {$\scriptstyle{H}$};
   \node at (-.8,-.53) {$\scriptstyle{G}$};
   \node at (-1.22,0) {$\scriptstyle{\nu}$};
   \node at (-.8,.53) {$\scriptstyle{K}$};
	\draw[-,thick,darkred] (0.08,-.4) to (0.08,-.23);
	\draw[-,thick,darkred] (0.08,.4) to (0.08,.03);
      \draw[thick,darkred] (0.08,-0.1) circle (4pt);
   \node at (0.08,-0.1) {\color{darkred}$\scriptstyle{x}$};
	\draw[-,thick,darkred] (-.8,-.4) to (-.8,-.03);
	\draw[-,thick,darkred] (-.8,.4) to (-.8,.23);
      \draw[thick,darkred] (-.8,0.1) circle (4pt);
   \node at (-.8,.1) {\color{darkred}$\scriptstyle{y}$};
%   \node at (.58,0) {$\scriptstyle{\lambda}$};
%   \node at (-.37,0) {$\scriptstyle{\mu}$};
%   \node at (-1.22,0) {$\scriptstyle{\nu}$};
\end{tikzpicture}
}
\quad=\quad
\mathord{
\begin{tikzpicture}[baseline = 0]
   \node at (0.08,-.53) {$\scriptstyle{F}$};
   \node at (0.58,0) {$\scriptstyle{\lambda}.$};
   \node at (-0.37,0) {$\scriptstyle{\mu}$};
   \node at (0.08,.53) {$\scriptstyle{H}$};
   \node at (-.8,-.53) {$\scriptstyle{G}$};
   \node at (-1.22,0) {$\scriptstyle{\nu}$};
   \node at (-.8,.53) {$\scriptstyle{K}$};
	\draw[-,thick,darkred] (0.08,-.4) to (0.08,-.13);
	\draw[-,thick,darkred] (0.08,.4) to (0.08,.13);
      \draw[thick,darkred] (0.08,0) circle (4pt);
   \node at (0.08,0) {\color{darkred}$\scriptstyle{x}$};
	\draw[-,thick,darkred] (-.8,-.4) to (-.8,-.13);
	\draw[-,thick,darkred] (-.8,.4) to (-.8,.13);
      \draw[thick,darkred] (-.8,0) circle (4pt);
   \node at (-.8,0) {\color{darkred}$\scriptstyle{y}$};
%   \node at (.58,0) {$\scriptstyle{\lambda}$};
%   \node at (-.37,0) {$\scriptstyle{\mu}$};
%   \node at (-1.22,0) {$\scriptstyle{\nu}$};
\end{tikzpicture}
}
\quad=\quad
\mathord{
\begin{tikzpicture}[baseline = 0]
   \node at (0.08,-.53) {$\scriptstyle{F}$};
   \node at (0.58,0) {$\scriptstyle{\lambda}.$};
   \node at (-0.37,0) {$\scriptstyle{\mu}$};
   \node at (0.08,.53) {$\scriptstyle{H}$};
   \node at (-.8,-.53) {$\scriptstyle{G}$};
   \node at (-1.22,0) {$\scriptstyle{\nu}$};
   \node at (-.8,.53) {$\scriptstyle{K}$};
%   \node at (.58,0) {$\scriptstyle{\lambda}$};
%   \node at (-.37,0) {$\scriptstyle{\mu}$};
%   \node at (-1.22,0) {$\scriptstyle{\nu}$};
	\draw[-,thick,darkred] (0.08,-.4) to (0.08,-.03);
	\draw[-,thick,darkred] (0.08,.4) to (0.08,.23);
      \draw[thick,darkred] (0.08,0.1) circle (4pt);
   \node at (0.08,0.1) {\color{darkred}$\scriptstyle{x}$};
	\draw[-,thick,darkred] (-.8,-.4) to (-.8,-.23);
	\draw[-,thick,darkred] (-.8,.4) to (-.8,.03);
      \draw[thick,darkred] (-.8,-0.1) circle (4pt);
   \node at (-.8,-.1) {\color{darkred}$\scriptstyle{y}$};
\end{tikzpicture}
}.
$$
This is the {\em interchange law}; it means that diagrams for 2-morphisms are
invariant under rectilinear 
isotopy.

We note that any strict 2-category $\CC$ has an {\em additive envelope}
constructed by taking the additive envelope of each of the morphism
categories in $\CC$.
The {\em additive Karoubi envelope}
$\dot\CC$ is the strict 2-category obtained by taking idempotent
completions after that.
Finally, we define the {\em Grothendieck ring} 
\begin{equation}\label{gring}
K_0(\dot\CC) := \bigoplus_{\lambda,\mu \in \ob \CC}
K_0(\mathcal{H}om_{\dot\CC}(\lambda,\mu)),
\end{equation}
where the latter $K_0$ denotes the usual split Grothendieck group of
an additive category. Horizontal composition induces a multiplication
making $K_0(\dot\CC)$ into a locally unital ring with distinguished
idempotents $(1_\lambda)_{\lambda \in \ob\CC}$.

\subsection{Quiver Hecke categories and the nil Hecke algebra}\label{snow}
The data of a strict monoidal category $\C$ is equivalent to that of a strict 2-category $\CC$ 
with one object;
the objects and morphisms in $\C$ correspond to the 1-morphisms and 2-morphisms in
$\CC$. For strict monoidal categories, we will use the same diagrammatic
formalism as explained in the previous subsection; the only difference is that there no
need to label the regions of the diagrams by objects since there is
only one.

In the next definition, we introduce the {\em quiver Hecke category}, which is ``half'' of the Kac-Moody
2-category $\UU(\g)$ to be defined in the next subsection.
Everything from this point on
depends
on the Kac-Moody data from $\S$\ref{km} plus
some additional parameters: we fix
\begin{itemize}
\item
units $t_{ij} \in \k^\times$ such that 
$t_{ii} = 1$ and
$d_{ij}=0 \Rightarrow t_{ij} = t_{ji}$;
\item
scalars\footnote{In \cite{B} and elsewhere, scalars $s_{ij}^{pq}$ are 
incorporated into the relations also for $p=0$ or $q=0$;
we don't allow so much freedom here because it makes it impossible to prove that
dots are nilpotent in the cyclotomic quotients
discussed below (see Lemma~\ref{nilp}).}
 $s_{ij}^{pq} \in \k$ for $0 < p < d_{ij}$, $0 < q <
d_{ji}$ such that
$s_{ij}^{pq} = s_{ji}^{qp}$.
\end{itemize}

\begin{definition}\label{pd}
The (positive) {\em quiver Hecke category}
$\H$ is the strict
monoidal category
generated by objects $I$ and 
morphisms
$\mathord{
\begin{tikzpicture}[baseline = -2]
	\draw[->,thick,darkred] (0.08,-.15) to (0.08,.3);
      \node at (0.08,0.05) {$\color{darkred}\bullet$};
   \node at (0.08,-.25) {$\scriptstyle{i}$};
\end{tikzpicture}
}:i \rightarrow i$ and
$\mathord{
\begin{tikzpicture}[baseline = -2]
	\draw[->,thick,darkred] (0.18,-.15) to (-0.18,.3);
	\draw[->,thick,darkred] (-0.18,-.15) to (0.18,.3);
   \node at (-0.18,-.25) {$\scriptstyle{i}$};
   \node at (0.18,-.25) {$\scriptstyle{j}$};
\end{tikzpicture}
}:i \otimes j \rightarrow j \otimes i$
subject to the following relations:
\begin{align*}
\mathord{
\begin{tikzpicture}[baseline = 0]
	\draw[->,thick,darkred] (0.28,.4) to[out=90,in=-90] (-0.28,1.1);
	\draw[->,thick,darkred] (-0.28,.4) to[out=90,in=-90] (0.28,1.1);
	\draw[-,thick,darkred] (0.28,-.3) to[out=90,in=-90] (-0.28,.4);
	\draw[-,thick,darkred] (-0.28,-.3) to[out=90,in=-90] (0.28,.4);
  \node at (-0.28,-.4) {$\scriptstyle{i}$};
  \node at (0.28,-.4) {$\scriptstyle{j}$};
\end{tikzpicture}
}
&=
\left\{
\begin{array}{ll}
0&\text{if $i=j$,}\\
t_{ij}\mathord{
\begin{tikzpicture}[baseline = 0]
	\draw[->,thick,darkred] (0.08,-.3) to (0.08,.4);
	\draw[->,thick,darkred] (-0.28,-.3) to (-0.28,.4);
   \node at (-0.28,-.4) {$\scriptstyle{i}$};
   \node at (0.08,-.4) {$\scriptstyle{j}$};
\end{tikzpicture}
}&\text{if $d_{ij}=0$,}\\
 t_{ij}
\mathord{
\begin{tikzpicture}[baseline = 0]
	\draw[->,thick,darkred] (0.08,-.3) to (0.08,.4);
	\draw[->,thick,darkred] (-0.28,-.3) to (-0.28,.4);
   \node at (-0.28,-.4) {$\scriptstyle{i}$};
   \node at (0.08,-.4) {$\scriptstyle{j}$};
      \node at (-0.28,0.05) {$\color{darkred}\bullet$};
      \node at (-0.5,0.2) {$\color{darkred}\scriptstyle{d_{ij}}$};
\end{tikzpicture}
}
+
t_{ji}
\mathord{
\begin{tikzpicture}[baseline = 0]
	\draw[->,thick,darkred] (0.08,-.3) to (0.08,.4);
	\draw[->,thick,darkred] (-0.28,-.3) to (-0.28,.4);
   \node at (-0.28,-.4) {$\scriptstyle{i}$};
   \node at (0.08,-.4) {$\scriptstyle{j}$};
     \node at (0.08,0.05) {$\color{darkred}\bullet$};
     \node at (0.32,0.2) {$\color{darkred}\scriptstyle{d_{ji}}$};
\end{tikzpicture}
}
+\!\! \displaystyle\sum_{\substack{0 < p < d_{ij}\\0 < q <
    d_{ji}}} \!\!\!\!\!s_{ij}^{pq}
\mathord{
\begin{tikzpicture}[baseline = 0]
	\draw[->,thick,darkred] (0.08,-.3) to (0.08,.4);
	\draw[->,thick,darkred] (-0.28,-.3) to (-0.28,.4);
   \node at (-0.28,-.4) {$\scriptstyle{i}$};
   \node at (0.08,-.4) {$\scriptstyle{j}$};
      \node at (-0.28,0.05) {$\color{darkred}\bullet$};
      \node at (0.08,0.05) {$\color{darkred}\bullet$};
      \node at (-0.43,0.2) {$\color{darkred}\scriptstyle{p}$};
      \node at (0.22,0.2) {$\color{darkred}\scriptstyle{q}$};
\end{tikzpicture}
}
&\text{otherwise,}\\
\end{array}
\right. 
\\
\mathord{
\begin{tikzpicture}[baseline = 0]
	\draw[<-,thick,darkred] (0.25,.6) to (-0.25,-.2);
	\draw[->,thick,darkred] (0.25,-.2) to (-0.25,.6);
  \node at (-0.25,-.26) {$\scriptstyle{i}$};
   \node at (0.25,-.26) {$\scriptstyle{j}$};
      \node at (-0.13,-0.02) {$\color{darkred}\bullet$};
\end{tikzpicture}
}
-
\mathord{
\begin{tikzpicture}[baseline = 0]
	\draw[<-,thick,darkred] (0.25,.6) to (-0.25,-.2);
	\draw[->,thick,darkred] (0.25,-.2) to (-0.25,.6);
  \node at (-0.25,-.26) {$\scriptstyle{i}$};
   \node at (0.25,-.26) {$\scriptstyle{j}$};
      \node at (0.13,0.42) {$\color{darkred}\bullet$};
\end{tikzpicture}
}
&=
\delta_{i,j}
\mathord{
\begin{tikzpicture}[baseline = -5]
 	\draw[->,thick,darkred] (0.08,-.3) to (0.08,.4);
	\draw[->,thick,darkred] (-0.28,-.3) to (-0.28,.4);
   \node at (-0.28,-.4) {$\scriptstyle{i}$};
   \node at (0.08,-.4) {$\scriptstyle{j}$};
\end{tikzpicture}
},\\
\mathord{
\begin{tikzpicture}[baseline = 0]
 	\draw[<-,thick,darkred] (0.25,.6) to (-0.25,-.2);
	\draw[->,thick,darkred] (0.25,-.2) to (-0.25,.6);
  \node at (-0.25,-.26) {$\scriptstyle{i}$};
   \node at (0.25,-.26) {$\scriptstyle{j}$};
      \node at (-0.13,0.42) {$\color{darkred}\bullet$};
\end{tikzpicture}
}
-
\mathord{
\begin{tikzpicture}[baseline = 0]
	\draw[<-,thick,darkred] (0.25,.6) to (-0.25,-.2);
	\draw[->,thick,darkred] (0.25,-.2) to (-0.25,.6);
  \node at (-0.25,-.26) {$\scriptstyle{i}$};
   \node at (0.25,-.26) {$\scriptstyle{j}$};
      \node at (0.13,-0.02) {$\color{darkred}\bullet$};
\end{tikzpicture}
}
&=
\delta_{i,j}
\mathord{
\begin{tikzpicture}[baseline = -5]
 	\draw[->,thick,darkred] (0.08,-.3) to (0.08,.4);
	\draw[->,thick,darkred] (-0.28,-.3) to (-0.28,.4);
   \node at (-0.28,-.4) {$\scriptstyle{i}$};
   \node at (0.08,-.4) {$\scriptstyle{j}$};
\end{tikzpicture}
},\\
\mathord{
\begin{tikzpicture}[baseline = 0]
	\draw[<-,thick,darkred] (0.45,.8) to (-0.45,-.4);
	\draw[->,thick,darkred] (0.45,-.4) to (-0.45,.8);
        \draw[-,thick,darkred] (0,-.4) to[out=90,in=-90] (-.45,0.2);
        \draw[->,thick,darkred] (-0.45,0.2) to[out=90,in=-90] (0,0.8);
   \node at (-0.45,-.45) {$\scriptstyle{i}$};
   \node at (0,-.45) {$\scriptstyle{j}$};
  \node at (0.45,-.45) {$\scriptstyle{k}$};
\end{tikzpicture}
}
\!\!-
\!\!\!
\mathord{
\begin{tikzpicture}[baseline = 0]
	\draw[<-,thick,darkred] (0.45,.8) to (-0.45,-.4);
	\draw[->,thick,darkred] (0.45,-.4) to (-0.45,.8);
        \draw[-,thick,darkred] (0,-.4) to[out=90,in=-90] (.45,0.2);
        \draw[->,thick,darkred] (0.45,0.2) to[out=90,in=-90] (0,0.8);
   \node at (-0.45,-.45) {$\scriptstyle{i}$};
   \node at (0,-.45) {$\scriptstyle{j}$};
  \node at (0.45,-.45) {$\scriptstyle{k}$};
\end{tikzpicture}
}
\!&=
\left\{
\begin{array}{ll}
\!\!\displaystyle
\sum_{\substack{r,s \geq 0 \\ r+s=d_{ij}-1}}
\!\!\!\!t_{ij}
\mathord{
\begin{tikzpicture}[baseline = 0]
	\draw[->,thick,darkred] (0.44,-.3) to (0.44,.4);
	\draw[->,thick,darkred] (0.08,-.3) to (0.08,.4);
	\draw[->,thick,darkred] (-0.28,-.3) to (-0.28,.4);
   \node at (-0.28,-.4) {$\scriptstyle{i}$};
   \node at (0.08,-.4) {$\scriptstyle{j}$};
   \node at (0.44,-.4) {$\scriptstyle{k}$};
     \node at (-0.28,0.05) {$\color{darkred}\bullet$};
     \node at (0.44,0.05) {$\color{darkred}\bullet$};
      \node at (-0.43,0.2) {$\color{darkred}\scriptstyle{r}$};
      \node at (0.55,0.2) {$\color{darkred}\scriptstyle{s}$};
\end{tikzpicture}
}
+ \displaystyle
\!\!\sum_{\substack{0 < p < d_{ij}\\0 < q <
    d_{ji}\\r,s \geq 0\\r+s=p-1}}
s_{ij}^{pq}
\mathord{
\begin{tikzpicture}[baseline = 0]
	\draw[->,thick,darkred] (0.44,-.3) to (0.44,.4);
	\draw[->,thick,darkred] (0.08,-.3) to (0.08,.4);
	\draw[->,thick,darkred] (-0.28,-.3) to (-0.28,.4);
   \node at (-0.28,-.4) {$\scriptstyle{i}$};
   \node at (0.08,-.4) {$\scriptstyle{j}$};
   \node at (0.44,-.4) {$\scriptstyle{k}$};
     \node at (-0.28,0.05) {$\color{darkred}\bullet$};
     \node at (0.44,0.05) {$\color{darkred}\bullet$};
      \node at (-0.43,0.2) {$\color{darkred}\scriptstyle{r}$};
     \node at (0.55,0.2) {$\color{darkred}\scriptstyle{s}$};
     \node at (0.08,0.05) {$\color{darkred}\bullet$};
      \node at (0.2,0.2) {$\color{darkred}\scriptstyle{q}$};
\end{tikzpicture}
}
&\text{if $i=k \neq j$,}\\
0&\text{otherwise.}
\end{array}
\right.
\end{align*}
(We depict the $n$th power of $\mathord{
\begin{tikzpicture}[baseline = -2]
	\draw[->,thick,darkred] (0.08,-.15) to (0.08,.3);
      \node at (0.08,0.05) {$\color{darkred}\bullet$};
   \node at (0.08,-.25) {$\scriptstyle{i}$};
\end{tikzpicture}
}$ under vertical composition by 
labelling the dot with $\color{darkred}n$.)
There is also the (negative) {\em quiver Hecke category} $\H'$
generated by objects $I$ and 
morphisms
$\mathord{
\begin{tikzpicture}[baseline = 1]
	\draw[<-,thick,darkred] (0.08,-.15) to (0.08,.3);
      \node at (0.08,0.15) {$\color{darkred}\bullet$};
   \node at (0.08,.4) {$\scriptstyle{i}$};
\end{tikzpicture}
}:i \rightarrow i$ and
$\mathord{
\begin{tikzpicture}[baseline = 0]
	\draw[<-,thick,darkred] (0.18,-.15) to (-0.18,.3);
	\draw[<-,thick,darkred] (-0.18,-.15) to (0.18,.3);
   \node at (-0.18,.4) {$\scriptstyle{j}$};
   \node at (0.18,.4) {$\scriptstyle{i}$};
\end{tikzpicture}
}:i \otimes j \rightarrow j \otimes i$
subject to the relations obtained by reversing the directions of
all the arrows in the above, then switching the order of the terms on the left hand
sides of the second, third and fourth relations. In fact, $\H'$ is isomorphic to $\H$, but the different
normalization of generators is sometimes more convenient.
\end{definition}

For objects $\bi = i_n \otimes\cdots \otimes i_1 \in I^{\otimes n}$ and $\bj = j_m \otimes\cdots
\otimes j_1 \in I^{\otimes m}$, there are no morphisms
$\bi \rightarrow \bj$ in $\H$ 
unless $m=n$.
The endomorphism algebra
\begin{equation}
H_n := \bigoplus_{\bi, \bj \in I^{\otimes n}} \Hom_{\H}(\bi, \bj)
\end{equation}
is the (positive) {\em quiver Hecke algebra}
which was introduced independently in \cite{Rou} and \cite{KL1}.
There is also the negative version
$H_n' := \bigoplus_{\bi, \bj \in I^{\otimes n}} \Hom_{\H'}(\bi, \bj)$, 
which is isomorphic to $H_n$ with a different normalization of generators.

In the special case that $I$ is a singleton, the
quiver Hecke algebra $H_n$ is the {\em nil Hecke algebra} $NH_n$,
which plays a crucial role in the general theory. 
Numbering strands by $1,\dots,n$ from right to left, let us
write $X_i$ for the element of $NH_n$ corresponding to a dot on the
$i$th strand, and $T_i$ for the element corresponding to the crossing of the $i$th and $(i+1)$th
strands. Let $S_n$ be the symmetric group with its usual simple
reflections $s_1,\dots,s_{n-1}$, 
length function $\ell$ and longest element $w_n$.
It acts naturally on the polynomial algebra $\Pol_n :=
\k[X_1,\dots,X_n]$; we write $\Sym_n$ for the subalgebra of invariants.
The following are well known; e.g. see \cite[$\S$2]{KL1},
\cite[$\S$2]{R2} or \cite[$\S$2]{Bsurvey}.
\begin{itemize}
\item[(N1)]
There is a faithful representation of $NH_n$ on $\Pol_n$
defined by $X_i \cdot f := X_i f$ and $T_i \cdot f := 
\frac{s_i(f)-f}{X_{i} - X_{i+1}}$ (the $i$th {\em Demazure operator}).
\item[(N2)] 
For any $w \in S_n$, let $T_w$ be the corresponding element of $NH_n$ defined via a reduced
expression for $w$.  
Then $NH_n$ is a free left $\Pol_n$-module on basis $\{T_w \:|\:w \in S_n\}$. 
In particular, $\Pol_n \hookrightarrow NH_n$.
\item[(N3)] The algebra
$\Pol_n$ is a free $\Sym_n$-module on basis $\{b_w\:|\:w \in S_n\}$
where $b_w := (-1)^{\ell(w)}T_w
\cdot X_1^{n-1} X_2^{n-2} \cdots X_{n-1}$.
%For any $f \in \Pol_n$, we have that 
%\begin{equation}\label{shrink}
%T_{w_n} \cdot f = \left(\displaystyle\sum_{w \in S_n}
%  (-1)^{\ell(w)} w(f)\right)\bigg/\left(\displaystyle\prod_{1 \leq i < j \leq n}
%  (X_j - X_i)\right).
%\end{equation}
Moreover $b_{w_n} = 1$.
\item[(N4)] 
There is an algebra isomorphism $NH_n
  \stackrel{\sim}{\rightarrow} \End_{\Sym_n}(\Pol_n)$
induced by
the action of $NH_n$ on $\Pol_n$.
Hence, $NH_n \cong M_{n!}(\Sym_n)$ and $Z(NH_n) = \Sym_n$.
\item[(N5)] 
  The element $\pi_n := (-1)^{\ell(w_n)} X_1^{n-1} X_2^{n-2} \cdots X_{n-1} T_{w_n}$
acts on the basis for $\Pol_n$ from (N3) by $\pi_n \cdot b_w =
\delta_{w,1} b_w$. Hence, it is a primitive idempotent in $NH_n$,
and $NH_n \cong (NH_n \pi_n)^{\oplus n!}$ as a left $NH_n$-module.
%\item[(N6)]
%We have that $(-1)^{\ell(w_n)}T_{w_n} X_1^{n-1}X_2^{n-2} \cdots X_{n-1}  T_{w_n} = T_{w_n}$;
%this follows because both sides act in the same way on the $b_w$'s.
%Also $T_{w_n} f T_{w_n}  = 0$ for any polynomial $f$ of
%smaller degree than
%$X_1^{n-1} X_2^{n-2}\cdots X_{n-1}$.
\end{itemize}

In the remainder of the subsection, we wish to give a first indication
of the power of the quiver Hecke relations.
Let $\C$ be some category which is additive and idempotent-complete.
Suppose that we are given a categorical action of the quiver Hecke category $\H$
on $\C$, i.e. there is
a strict monoidal functor
\begin{equation}
\Phi:\H \rightarrow \mathcal{E}\!nd(\C)
\end{equation}
where $\mathcal{E}\!nd(\C)$ denotes the strict monoidal category of all
endofunctors of $\C$.
Let $E_i := \Phi(i)$ and
$x_i := \Phi\big(\mathord{
\begin{tikzpicture}[baseline = -2]
	\draw[->,thick,darkred] (0.08,-.15) to (0.08,.3);
      \node at (0.08,0.05) {$\color{darkred}\bullet$};
   \node at (0.08,-.25) {$\scriptstyle{i}$};
\end{tikzpicture}
}\big) \in \End(E_i)$.
% and $\tau_{ij} := \Phi\big(\mathord{
%\begin{tikzpicture}[baseline = -2]
%	\draw[->,thick,darkred] (0.18,-.15) to (-0.18,.3);
%	\draw[->,thick,darkred] (-0.18,-.15) to (0.18,.3);
%   \node at (-0.18,-.25) {$\scriptstyle{i}$};
%   \node at (0.18,-.25) {$\scriptstyle{j}$};
%\end{tikzpicture}
%}\big)$
%for all $i,j \in I$.
Since $E_i$ is $\k$-linear (as always), it is additive, hence it induces an endomorphism
$e_i := [E_i]$ of the split Grothendieck group $K_0(\C)$.
More generally, for $i \in I$ and $n \geq 1$, 
we can obviously identify the nil Hecke algebra $NH_n$ with $\End_{\mathcal
  H}(i^{\otimes n})$;
then the image under $\Phi$ of the idempotent $\pi_n$ from (N5)
gives us an idempotent $\pi_{i,n} \in \End(E_i^n)$.
Let $E_i^{(n)} := \pi_{i,n} E_i^n$, i.e. it is the endofunctor of $\C$ that sends an object $V$ to the
image of $(\pi_{i,n})_V \in \End_\C(E_i^n V)$, and a morphism $f:V \rightarrow W$ to the
restriction of $E_i^n f$.
The following lemma shows that this categorifies the divided power $e_i^{(n)} := e_i^n / n!$.

\begin{lemma}\label{divpower}
We have that $E_i^n \cong \left(E_i^{(n)}\right)^{\bigoplus n!}$.
\end{lemma}

\begin{proof}
By (N5), the identity element of $NH_n$ is a sum of $n!$ primitive
idempotents, each of which is conjugate to $\pi_n$.
\end{proof}

\begin{lemma}\label{proud}
Suppose for some $V \in \ob \C$
that there is a monic polynomial $f(t)$ of degree $n$ such that
$f((x_i)_V) = 0$. 
Then $E_i^{(n+1)} V = 0$.
In particular, if $\C$ is finite-dimensional, then all $e_i$
act locally nilpotently on $K_0(\C)$.
\end{lemma}

\begin{proof}
The second statement follows from the first: if $\C$ is
finite-dimensional then, for any $V \in \ob \C$, we have that $(x_i)_V$ is an element of the finite-dimensional algebra $\End_\C(E_i V)$. 
Hence, it certainly satisfies some polynomial relation.
To prove the first statement, 
we first note the following identity in $NH_{n+1}$:
$$
(-1)^{\ell(w_{n+1})} \pi_{n+1} f(X_1) X_2^{n-1} \cdots X_n T_{w_{n+1}} = \pi_{n+1}.
$$
This holds because $\pi_{n+1}^2 = \pi_{n+1}$ and moreover $T_{w_{n+1}} X_1^m X_2^{n-1}\cdots X_n
T_{w_{n+1}} = 0$ for $m < n$ by degree considerations.
Now as above we identify $NH_{n+1}$ with 
$\End_{\H}(i^{\otimes(n+1)})$, apply $\Phi$ to our identity, then evaluate the
resulting natural transformations at $V$.
By assumption, $\Phi(f(X_1))_V = E_i^n f((x_i)_V) = 0$. Hence, the left hand side
vanishes,
and we deduce that $\Phi(\pi_{n+1})_V = 0$. 
This is the identity endomorphism of $E_i^{(n+1)} V$, so the latter
object is
isomorphic to zero.
\end{proof}

Perhaps most striking of all, we have the following, which is an
immediate consequence of the even stronger categorical Serre relations
from \cite[Proposition 4.2]{Rou}:

\begin{lemma}\label{catserre}
The endomorphisms $e_i$ of $K_0(\C)$ satisfy the Serre
relations from (\ref{serre}).
\end{lemma}

\subsection{Kac-Moody 2-categories}\label{ji}
We are ready to formulate Rouquier's definition of the
Kac-Moody 2-category $\UU(\g)$; cf. \cite[$\S$4.1.3]{Rou}.

\begin{definition}\label{def1}
The {\em Kac-Moody 2-category}
 $\UU(\g)$
is the strict 2-category\footnote{Some authors require it is additive
from the outset but we don't assume this.}
with objects $P$;
generating 
1-morphisms 
$E_i 1_\lambda:\lambda \rightarrow \lambda+\alpha_i$ and
$F_i 1_\lambda:\lambda \rightarrow \lambda-\alpha_i$ for each $i \in I$ and
$\lambda \in P$,
whose identity 2-morphisms will be
 represented diagrammatically by
$
{\scriptstyle\lambda+\alpha_i}
\mathord{

}
{\scriptstyle\lambda}$, respectively;
and generating
2-morphisms
$
\mathord{
%
}
\,{\scriptstyle\lambda}
:E_i F_i 1_\lambda \rightarrow 1_\lambda$.
The generating 2-morphisms are subject to the following relations.
First, we have the {\em positive quiver Hecke relations}
(cf. Definition~\ref{pd}):
\begin{align*}
\mathord{

}
&=
\left\{
\begin{array}{ll}
0&\text{if $i=j$,}\\
t_{ij}\mathord{
\begin{tikzpicture}[baseline = 0]
	\draw[->,thick,darkred] (0.08,-.3) to (0.08,.4);
	\draw[->,thick,darkred] (-0.28,-.3) to (-0.28,.4);
   \node at (-0.28,-.4) {$\scriptstyle{i}$};
   \node at (0.08,-.4) {$\scriptstyle{j}$};
   \node at (.3,.05) {$\scriptstyle{\lambda}$};
\end{tikzpicture}
}&\text{if $d_{ij}=0$,}\\
 t_{ij}
\mathord{
\begin{tikzpicture}[baseline = 0]
	\draw[->,thick,darkred] (0.08,-.3) to (0.08,.4);
	\draw[->,thick,darkred] (-0.28,-.3) to (-0.28,.4);
   \node at (-0.28,-.4) {$\scriptstyle{i}$};
   \node at (0.08,-.4) {$\scriptstyle{j}$};
   \node at (.3,-.05) {$\scriptstyle{\lambda}$};
      \node at (-0.28,0.05) {$\color{darkred}\bullet$};
      \node at (-0.5,0.2) {$\color{darkred}\scriptstyle{d_{ij}}$};
\end{tikzpicture}
}
+
t_{ji}
\mathord{
\begin{tikzpicture}[baseline = 0]
	\draw[->,thick,darkred] (0.08,-.3) to (0.08,.4);
	\draw[->,thick,darkred] (-0.28,-.3) to (-0.28,.4);
   \node at (-0.28,-.4) {$\scriptstyle{i}$};
   \node at (0.08,-.4) {$\scriptstyle{j}$};
   \node at (.3,-.05) {$\scriptstyle{\lambda}$};
     \node at (0.08,0.05) {$\color{darkred}\bullet$};
     \node at (0.32,0.2) {$\color{darkred}\scriptstyle{d_{ji}}$};
\end{tikzpicture}
}
+\!\! \displaystyle\sum_{\substack{0 < p < d_{ij}\\0 < q <
    d_{ji}}} \!\!\!\!\!s_{ij}^{pq}
\mathord{
\begin{tikzpicture}[baseline = 0]
	\draw[->,thick,darkred] (0.08,-.3) to (0.08,.4);
	\draw[->,thick,darkred] (-0.28,-.3) to (-0.28,.4);
   \node at (-0.28,-.4) {$\scriptstyle{i}$};
   \node at (0.08,-.4) {$\scriptstyle{j}$};
   \node at (.3,-.05) {$\scriptstyle{\lambda}$};
      \node at (-0.28,0.05) {$\color{darkred}\bullet$};
      \node at (0.08,0.05) {$\color{darkred}\bullet$};
      \node at (-0.43,0.2) {$\color{darkred}\scriptstyle{p}$};
      \node at (0.22,0.2) {$\color{darkred}\scriptstyle{q}$};
\end{tikzpicture}
}
&\text{otherwise,}\\
\end{array}%
\right.
\end{align*}
Next we have the {\em right adjunction relations}
\begin{align*}
\mathord{
%
},
\end{align*}
which imply that $F_i 1_{\lambda+\alpha_i}$ is a
right dual of $E_i 1_\lambda$.
Finally there are some {\em inversion relations}. To formulate these, define a new 2-morphism
$
\mathord{
%
}.
\end{align*}
Then the inversion relations assert that the following
are isomorphisms:
\begin{align*}
\mathord{
%
}
&
:E_i F_i 1_\lambda \oplus 
1_\lambda^{\oplus -\langle h_i,\lambda\rangle}
\stackrel{\sim}{\rightarrow}
 F_i E_i 1_\lambda&\text{if $\langle h_i,\lambda\rangle \leq
  0$},
\end{align*}
the last two being 2-morphisms in the additive envelope of
$\UU(\g)$.
\end{definition}

\begin{remark}\rm
More formally, the inversion relations
mean that there are some additional
generating 2-morphisms
$\mathord{
%
}
\,{\scriptstyle\lambda}:
F_i E_i 1_\lambda \rightarrow 1_\lambda$
for $0 \leq n  < \langle h_i,\lambda\rangle$ and $0 \leq m < -\langle
h_i,\lambda\rangle$ such that the following hold:
\begin{align*}
\mathord{
%
}
\bigg)^{-1}
&&
\text{if $\langle h_i,\lambda \rangle \leq 0$.}
\end{align*}
\end{remark}

As usual with objects defined by generators and relations, one then
needs to
play the game of deriving consequences from the defining relations.
Here we record some which were established in
\cite{B}; we also cite below the more recent exposition in \cite{BE2} since that uses
exactly the same normalization as here.
\begin{itemize}
\item[(K1)] {\em Negative quiver Hecke relations.}
Define 2-morphisms
$\mathord{
%
}.
\end{align*}
On rotating the positive quiver Hecke
relations clockwise through $180^\circ$, one sees that 
these satisfy the negative quiver Hecke relations:
\begin{align*}
\mathord{

}
&=
\left\{
\begin{array}{ll}
0&\text{if $i=j$,}\\
\mathord{
\begin{tikzpicture}[baseline = 0]
	\draw[<-,thick,darkred] (0.08,-.3) to (0.08,.4);
	\draw[<-,thick,darkred] (-0.28,-.3) to (-0.28,.4);
   \node at (-0.28,.5) {$\scriptstyle{i}$};
   \node at (0.08,.5) {$\scriptstyle{j}$};
   \node at (.3,.05) {$\scriptstyle{\lambda}$};
\end{tikzpicture}
}&\text{if $d_{ij}=0$,}\\
t_{ij}\mathord{
\begin{tikzpicture}[baseline = 0]
	\draw[<-,thick,darkred] (0.08,-.3) to (0.08,.4);
	\draw[<-,thick,darkred] (-0.28,-.3) to (-0.28,.4);
   \node at (-0.28,.5) {$\scriptstyle{i}$};
   \node at (0.08,.5) {$\scriptstyle{j}$};
   \node at (.3,-.05) {$\scriptstyle{\lambda}$};
      \node at (-0.28,0.05) {$\color{darkred}\bullet$};
      \node at (-0.5,0.2) {$\color{darkred}\scriptstyle{d_{ij}}$};
\end{tikzpicture}
}
+
t_{ji}
\mathord{
\begin{tikzpicture}[baseline = 0]
	\draw[<-,thick,darkred] (0.08,-.3) to (0.08,.4);
	\draw[<-,thick,darkred] (-0.28,-.3) to (-0.28,.4);
   \node at (-0.28,.5) {$\scriptstyle{i}$};
   \node at (0.08,.5) {$\scriptstyle{j}$};
   \node at (.3,-.05) {$\scriptstyle{\lambda}$};
     \node at (0.08,0.05) {$\color{darkred}\bullet$};
     \node at (0.32,0.2) {$\color{darkred}\scriptstyle{d_{ji}}$};
\end{tikzpicture}
}
+\!\! \displaystyle\sum_{\substack{0 < p < d_{ij}\\0 < q <
    d_{ji}}} \!\!\!\!\!s_{ij}^{pq}
\mathord{
\begin{tikzpicture}[baseline = 0]
	\draw[<-,thick,darkred] (0.08,-.3) to (0.08,.4);
	\draw[<-,thick,darkred] (-0.28,-.3) to (-0.28,.4);
   \node at (-0.28,.5) {$\scriptstyle{i}$};
   \node at (0.08,.5) {$\scriptstyle{j}$};
   \node at (.3,-.05) {$\scriptstyle{\lambda}$};
      \node at (-0.28,0.05) {$\color{darkred}\bullet$};
      \node at (0.08,0.05) {$\color{darkred}\bullet$};
      \node at (-0.43,0.2) {$\color{darkred}\scriptstyle{p}$};
      \node at (0.22,0.2) {$\color{darkred}\scriptstyle{q}$};
\end{tikzpicture}
}
&\text{otherwise,}\\
\end{array}%
\right.
\end{align*}
\item[(K2)] {\em Second adjunction.}
We next introduce 2-morphisms
$\mathord{
%
}
\,{\scriptstyle\lambda}:F_i E_i 1_\lambda \rightarrow 1_\lambda$.
The definition of these was suggested already by Rouquier in
\cite[$\S$4.1.4]{Rou}. Following the idea of \cite{BHLW}
we will normalize them in a different way which
depends on an additional choice of
units $c_{\lambda;i} \in \k^\times$ for each $i \in I$ and
  $\lambda \in P$ such that 
$c_{\lambda+\alpha_j;i} = t_{ij}
  c_{\lambda;i}$.
Fixing such a choice from now on,
we set
\begin{align*}
\mathord{

}
&:=
\left\{
\begin{array}{rl}
c_{\lambda;i}
\mathord{
\begin{tikzpicture}[baseline = 0]
	\draw[-,thick,darkred] (0.4,0.4) to[out=-90, in=0] (0.1,0);
	\draw[->,thick,darkred] (0.1,0) to[out = 180, in = -90] (-0.2,0.4);
    \node at (0.4,.5) {$\scriptstyle{i}$};
  \node at (0.5,0.15) {$\scriptstyle{\lambda}$};
      \node at (0.12,-0.21) {$\color{darkred}\scriptstyle{\langle h_i,\lambda\rangle-1}$};
      \node at (0.12,0.01) {$\color{darkred}\scriptstyle{\spadesuit}$};
\end{tikzpicture}
}\hspace{7.5mm}
&\text{if $\langle h_i,\lambda \rangle > 0$,}\\
c_{\lambda;i}\,
\mathord{
\begin{tikzpicture}[baseline = 0]
	\draw[-,thick,darkred] (0.28,.6) to[out=240,in=90] (-0.28,-.1);
	\draw[<-,thick,darkred] (-0.28,.6) to[out=300,in=90] (0.28,-0.1);
   \node at (0.28,.7) {$\scriptstyle{i}$};
   \node at (.45,.2) {$\scriptstyle{\lambda}$};
	\draw[-,thick,darkred] (0.28,-0.1) to[out=-90, in=0] (0,-0.4);
	\draw[-,thick,darkred] (0,-0.4) to[out = 180, in = -90] (-0.28,-0.1);
      \node at (0.89,-0.1) {$\color{darkred}\scriptstyle{-\langle h_i,\lambda\rangle}$};
      \node at (0.27,-0.1) {$\color{darkred}\bullet$};
\end{tikzpicture}
}
&\text{if $\langle h_i,\lambda \rangle \leq 0$,}
\end{array}%
\right.
\end{align*}
Then by \cite[Theorem 4.3]{B} (or \cite[Proposition 6.2]{BE2} with the present normalization) we have
the {\em left adjunction relations}:
\begin{equation*}
\mathord{
%
}.
\end{equation*}
This means that $F_i 1_{\lambda+\alpha_i}$
is also a left dual of $E_i 1_\lambda$.
\item[(K3)] {\em Pitchfork relations and cyclicity.}
The following relations follow from the definitions and \cite[Theorem 5.3]{B} (or
\cite[Proposition 4.1]{BE2} and \cite[Proposition 7.2]{BE2}):
\begin{align*}
\mathord{
%
}.
\end{align*}
It follows that 
\begin{align*}
&\mathord{
%
}
{\scriptstyle\lambda}
$ 
are {\em cyclic}
i.e. their right mates are equal to their left mates:
\begin{align*}
\mathord{
%
}.
\end{align*}
This is the main advantage of the normalization of the second
adjunction from \cite{BHLW} as chosen in
(K2).
\item[(K4)] {\em Infinite Grassmannian relations.}
Let $\Sym$ be the algebra of symmetric functions over $\k$.
Recall $\Sym$ is generated both by the elementary symmetric functions $\e_n\:(n
\geq 1)$ and by the complete symmetric functions $\h_n\:(n \geq 1)$.
Adopting the convention that $\e_0 = \h_0 = 1$ and $\e_n = \h_n = 0$ for $n < 0$,
these two families of generators are related by the equation
\begin{equation}\label{igr}
\sum_{r+s=n}(-1)^r \e_r \h_s = 0
\text{ for all $n > 0$}.
\end{equation}
Take $i \in I$, $\lambda \in P$ and set $h := \langle h_i,\lambda \rangle$.
Then the infinite Grassmannian relations assert that there is a
well-defined
homomorphism
$$
\beta_{\lambda;i}:\Sym \rightarrow
\operatorname{End}_{\UU(\g)}(1_\lambda)
$$
such that 
\begin{align}\label{ee}
\qquad\quad \beta_{\lambda;i}(\h_n)
&
=c_{\lambda;i}^{-1}
\mathord{
\begin{tikzpicture}[baseline = 0]
  \draw[<-,thick,darkred] (0,0.4) to[out=180,in=90] (-.2,0.2);
  \draw[-,thick,darkred] (0.2,0.2) to[out=90,in=0] (0,.4);
 \draw[-,thick,darkred] (-.2,0.2) to[out=-90,in=180] (0,0);
  \draw[-,thick,darkred] (0,0) to[out=0,in=-90] (0.2,0.2);
 \node at (0,-.1) {$\scriptstyle{i}$};
   \node at (0.3,0.2) {$\scriptstyle{\lambda}$};
   \node at (-0.2,0.2) {$\color{darkred}\bullet$};
   \node at (-0.75,0.2) {$\color{darkred}\scriptstyle{n+h-1}$};
\end{tikzpicture}
}&\text{if $n> -h$,}\\\label{hh}
\beta_{\lambda;i}(\e_n)
&=
(-1)^n c_{\lambda;i}\mathord{
\begin{tikzpicture}[baseline = 0]
  \draw[->,thick,darkred] (0.2,0.2) to[out=90,in=0] (0,.4);
  \draw[-,thick,darkred] (0,0.4) to[out=180,in=90] (-.2,0.2);
\draw[-,thick,darkred] (-.2,0.2) to[out=-90,in=180] (0,0);
  \draw[-,thick,darkred] (0,0) to[out=0,in=-90] (0.2,0.2);
 \node at (0,-.08) {$\scriptstyle{i}$};
   \node at (-0.3,0.2) {$\scriptstyle{\lambda}$};
   \node at (0.2,0.2) {$\color{darkred}\bullet$};
   \node at (0.75,0.2) {$\color{darkred}\scriptstyle{n-h-1}$};
\end{tikzpicture}
}&\text{if $n > h$.}
\end{align}
This was proved originally by Lauda\footnote{In fact, Lauda showed for $\g = \mathfrak{sl}_2$ that 
this homomorphism is an isomorphism. In general, 
the product of the homomorphisms $\beta_{\lambda;i}$ over all
$i \in I$ should given an isomorphism $\bigotimes_{i \in I} \Sym
\stackrel{\sim}{\rightarrow}
\End_{\UU(\g)}(1_\lambda)$, but the proof of this assertion 
depends on the Nondegeneracy Condition discussed later in
the subsection.} in \cite[Proposition 8.2]{L}; see \cite[Proposition
5.1]{BE2} where it is established using our normalization.
It motivated Lauda's introduction also of certain {\em negatively dotted
  bubbles}, which are 2-morphisms in $\End_{\UU(\g)}(1_\lambda)$ defined
so that (\ref{ee})--(\ref{hh}) hold for all $n\in\Z$.
\item[(K5)] {\em Dual inversion relations.}
The following 2-morphisms are invertible:
\begin{align*}
\mathord{

}
&
:F_i E_i 1_\lambda\stackrel{\sim}{\rightarrow}
 E_i F_i 1_\lambda \oplus 
1_\lambda^{\oplus -\langle h_i,\lambda\rangle}
 &\text{if $\langle h_i,\lambda\rangle \leq
  0$}.
\end{align*}
This may be deduced from the definitions above using also the following relations proved in 
\cite[Corollary 3.3]{B} (or \cite[Corollary 5.2]{BE2}):
$$
\mathord{
%
}.
$$
Note here we are using the negatively dotted bubbles.
Another consequence of the last relations displayed, plus the description
of the leftward crossing given in (K3), is that 
all 2-morphisms in $\UU(\g)$ are generated (under both vertical
and horizontal composition) by upward dots, 
upward crossings, and leftward and rightward cups and caps.
\item[(K6)] {\em Curl relations.}
For $n \geq 0$ 
we have:
$$
\mathord{
%
}\:
{\scriptstyle\lambda}.
$$
These are proved e.g. in \cite[Lemma 3.2]{BHLW} or \cite[Corollary 5.4]{BE2}.
\item[(K7)] {\em Bubble slides.}
For the next relations, 
we adopt the following convenient shorthand:
\begin{equation}\label{hfg}
\mathord{
%
}.
\end{equation}
Then following hold for all $n \in \Z$.
First:
\begin{align*}
\mathord{
%
}\,
{\scriptstyle\lambda}.
\end{align*}
In the simply-laced case, these were proved in
\cite[Propositions 3.3--3.5]{KL3}.
In general, they are recorded in various places in the literature;
e.g. see \cite[Proposition 2.8]{Web} or \cite[$\S$3.2]{BHLW}.
When $d_{ij} > 1$, the proof of
the final two bubble slides above is not as straightforward as
those references may suggest as it requires also an application of 
the deformed braid relation; we refer to \cite[Proposition
7.3]{BE2} for the detailed argument.
\item[(K8)] {\em Alternating crossings.}
Finally we record the following relations from \cite[Corollary
5.3]{BE2} and \cite[Proposition 7.6]{BE2}; the first two are immediate
from the definitions and the ``diamond relations'' recorded in (K5);
the last one was derived already in \cite{KL3}.
\begin{align*}
\mathord{
\begin{tikzpicture}[baseline = 0]
	\draw[->,thick,darkred] (0.28,0) to[out=90,in=-90] (-0.28,.7);
	\draw[-,thick,darkred] (-0.28,0) to[out=90,in=-90] (0.28,.7);
	\draw[<-,thick,darkred] (0.28,-.7) to[out=90,in=-90] (-0.28,0);
	\draw[-,thick,darkred] (-0.28,-.7) to[out=90,in=-90] (0.28,0);
  \node at (-0.28,-.8) {$\scriptstyle{i}$};
  \node at (0.28,.8) {$\scriptstyle{j}$};
  \node at (.43,0) {$\scriptstyle{\lambda}$};
\end{tikzpicture}
}
&=
(-1)^{\delta_{i,j}}
\mathord{
\begin{tikzpicture}[baseline = 0]
	\draw[<-,thick,darkred] (0.08,-.3) to (0.08,.4);
	\draw[->,thick,darkred] (-0.28,-.3) to (-0.28,.4);
   \node at (-0.28,-.4) {$\scriptstyle{i}$};
   \node at (0.08,.5) {$\scriptstyle{j}$};
   \node at (.3,.05) {$\scriptstyle{\lambda}$};
\end{tikzpicture}
}+\delta_{i,j}\!\!\sum_{n=0}^{\langle h_i, \lambda \rangle-1}
\sum_{r \geq 0}
\mathord{
\begin{tikzpicture}[baseline = 0]
	\draw[-,thick,darkred] (0.3,0.7) to[out=-90, in=0] (0,0.3);
	\draw[->,thick,darkred] (0,0.3) to[out = 180, in = -90] (-0.3,0.7);
    \node at (0.3,.8) {$\scriptstyle{i}$};
    \node at (0.4,-0.32) {$\scriptstyle{\lambda}$};
  \draw[->,thick,darkred] (0.2,0) to[out=90,in=0] (0,0.2);
  \draw[-,thick,darkred] (0,0.2) to[out=180,in=90] (-.2,0);
\draw[-,thick,darkred] (-.2,0) to[out=-90,in=180] (0,-0.2);
  \draw[-,thick,darkred] (0,-0.2) to[out=0,in=-90] (0.2,0);
 \node at (-0.3,0) {$\scriptstyle{i}$};
   \node at (0.2,0) {$\color{darkred}\bullet$};
   \node at (0.9,0) {$\color{darkred}\scriptstyle{-n-r-2}$};
   \node at (-0.25,0.45) {$\color{darkred}\bullet$};
   \node at (-0.45,0.45) {$\color{darkred}\scriptstyle{r}$};
	\draw[<-,thick,darkred] (0.3,-.7) to[out=90, in=0] (0,-0.3);
	\draw[-,thick,darkred] (0,-0.3) to[out = 180, in = 90] (-0.3,-.7);
    \node at (-0.3,-.8) {$\scriptstyle{i}$};
   \node at (-0.25,-0.5) {$\color{darkred}\bullet$};
   \node at (-.4,-.5) {$\color{darkred}\scriptstyle{n}$};
\end{tikzpicture}
},\\
\mathord{
\begin{tikzpicture}[baseline = 0]
	\draw[-,thick,darkred] (0.28,0) to[out=90,in=-90] (-0.28,.7);
	\draw[->,thick,darkred] (-0.28,0) to[out=90,in=-90] (0.28,.7);
	\draw[-,thick,darkred] (0.28,-.7) to[out=90,in=-90] (-0.28,0);
	\draw[<-,thick,darkred] (-0.28,-.7) to[out=90,in=-90] (0.28,0);
  \node at (0.28,-.8) {$\scriptstyle{i}$};
  \node at (-0.28,.8) {$\scriptstyle{j}$};
  \node at (.43,0) {$\scriptstyle{\lambda}$};
\end{tikzpicture}
}
&=
(-1)^{\delta_{i,j}}\mathord{
\begin{tikzpicture}[baseline = 0]
	\draw[->,thick,darkred] (0.08,-.3) to (0.08,.4);
	\draw[<-,thick,darkred] (-0.28,-.3) to (-0.28,.4);
   \node at (-0.28,.5) {$\scriptstyle{j}$};
   \node at (0.08,-.4) {$\scriptstyle{i}$};
   \node at (.3,.05) {$\scriptstyle{\lambda}$};
\end{tikzpicture}
}+\delta_{i,j}\!\!\sum_{n=0}^{-\langle h_i, \lambda \rangle-1}
\sum_{r \geq 0}
\mathord{
\begin{tikzpicture}[baseline=0]
	\draw[-,thick,darkred] (0.3,-0.7) to[out=90, in=0] (0,-0.3);
	\draw[->,thick,darkred] (0,-0.3) to[out = 180, in = 90] (-0.3,-.7);
    \node at (0.3,-.8) {$\scriptstyle{i}$};
   \node at (0.25,-0.5) {$\color{darkred}\bullet$};
   \node at (.4,-.5) {$\color{darkred}\scriptstyle{r}$};
  \draw[<-,thick,darkred] (0,0.2) to[out=180,in=90] (-.2,0);
  \draw[-,thick,darkred] (0.2,0) to[out=90,in=0] (0,.2);
 \draw[-,thick,darkred] (-.2,0) to[out=-90,in=180] (0,-0.2);
  \draw[-,thick,darkred] (0,-0.2) to[out=0,in=-90] (0.2,0);
 \node at (0.25,0.1) {$\scriptstyle{i}$};
   \node at (0.4,-0.2) {$\scriptstyle{\lambda}$};
   \node at (-0.2,0) {$\color{darkred}\bullet$};
   \node at (-0.9,0) {$\color{darkred}\scriptstyle{-n-r-2}$};
	\draw[<-,thick,darkred] (0.3,.7) to[out=-90, in=0] (0,0.3);
	\draw[-,thick,darkred] (0,0.3) to[out = -180, in = -90] (-0.3,.7);
   \node at (0.27,0.5) {$\color{darkred}\bullet$};
   \node at (0.44,0.5) {$\color{darkred}\scriptstyle{n}$};
    \node at (-0.3,.8) {$\scriptstyle{i}$};
\end{tikzpicture}
},\\
\mathord{
\begin{tikzpicture}[baseline = 0]
	\draw[<-,thick,darkred] (0.45,.8) to (-0.45,-.4);
	\draw[->,thick,darkred] (0.45,-.4) to (-0.45,.8);
        \draw[<-,thick,darkred] (0,-.4) to[out=90,in=-90] (-.45,0.2);
        \draw[-,thick,darkred] (-0.45,0.2) to[out=90,in=-90] (0,0.8);
   \node at (-0.45,-.5) {$\scriptstyle{j}$};
   \node at (0,.9) {$\scriptstyle{i}$};
  \node at (0.45,-.5) {$\scriptstyle{k}$};
   \node at (.5,-.1) {$\scriptstyle{\lambda}$};
\end{tikzpicture}
}-
\mathord{
\begin{tikzpicture}[baseline = 0]
	\draw[<-,thick,darkred] (0.45,.8) to (-0.45,-.4);
	\draw[->,thick,darkred] (0.45,-.4) to (-0.45,.8);
        \draw[<-,thick,darkred] (0,-.4) to[out=90,in=-90] (.45,0.2);
        \draw[-,thick,darkred] (0.45,0.2) to[out=90,in=-90] (0,0.8);
   \node at (-0.45,-.5) {$\scriptstyle{j}$};
   \node at (0,.9) {$\scriptstyle{i}$};
  \node at (0.45,-.5) {$\scriptstyle{k}$};
   \node at (.5,-.1) {$\scriptstyle{\lambda}$};
\end{tikzpicture}
}
&=
\delta_{i,j} \delta_{i,k}
\displaystyle\sum_{r,s,t \geq 0}
\left(
\mathord{
\begin{tikzpicture}[baseline = 0]
	\draw[-,thick,darkred] (0.3,0.7) to[out=-90, in=0] (0,0.3);
	\draw[->,thick,darkred] (0,0.3) to[out = 180, in = -90] (-0.3,0.7);
    \node at (0.3,.8) {$\scriptstyle{i}$};
    \node at (0.98,-.8) {$\scriptstyle{i}$};
    \node at (1.4,-0.32) {$\scriptstyle{\lambda}$};
  \draw[->,thick,darkred] (0.2,0) to[out=90,in=0] (0,0.2);
  \draw[-,thick,darkred] (0,0.2) to[out=180,in=90] (-.2,0);
\draw[-,thick,darkred] (-.2,0) to[out=-90,in=180] (0,-0.2);
  \draw[-,thick,darkred] (0,-0.2) to[out=0,in=-90] (0.2,0);
 \node at (-0.3,0) {$\scriptstyle{i}$};
   \node at (0.2,0) {$\color{darkred}\bullet$};
   \node at (.96,0) {$\color{darkred}\scriptstyle{-r-s-t-3}$};
   \node at (-0.23,0.43) {$\color{darkred}\bullet$};
   \node at (-0.43,0.43) {$\color{darkred}\scriptstyle{r}$};
	\draw[<-,thick,darkred] (0.3,-.7) to[out=90, in=0] (0,-0.3);
	\draw[-,thick,darkred] (0,-0.3) to[out = 180, in = 90] (-0.3,-.7);
    \node at (-0.3,-.8) {$\scriptstyle{i}$};
   \node at (-0.25,-0.5) {$\color{darkred}\bullet$};
   \node at (-.4,-.5) {$\color{darkred}\scriptstyle{s}$};
	\draw[->,thick,darkred] (.98,-0.7) to (.98,0.7);
   \node at (0.98,0.4) {$\color{darkred}\bullet$};
   \node at (1.15,0.4) {$\color{darkred}\scriptstyle{t}$};
\end{tikzpicture}
}
+
\mathord{
\begin{tikzpicture}[baseline = 0]
	\draw[<-,thick,darkred] (0.3,0.7) to[out=-90, in=0] (0,0.3);
	\draw[-,thick,darkred] (0,0.3) to[out = 180, in = -90] (-0.3,0.7);
    \node at (-0.3,.8) {$\scriptstyle{i}$};
    \node at (-0.98,-.8) {$\scriptstyle{i}$};
    \node at (.6,-0.32) {$\scriptstyle{\lambda}$};
  \draw[-,thick,darkred] (0.2,0) to[out=90,in=0] (0,0.2);
  \draw[<-,thick,darkred] (0,0.2) to[out=180,in=90] (-.2,0);
\draw[-,thick,darkred] (-.2,0) to[out=-90,in=180] (0,-0.2);
  \draw[-,thick,darkred] (0,-0.2) to[out=0,in=-90] (0.2,0);
 \node at (0.3,0) {$\scriptstyle{i}$};
   \node at (-0.2,0) {$\color{darkred}\bullet$};
   \node at (-1,0) {$\color{darkred}\scriptstyle{-r-s-t-3}$};
   \node at (0.23,0.43) {$\color{darkred}\bullet$};
   \node at (0.43,0.43) {$\color{darkred}\scriptstyle{s}$};
	\draw[-,thick,darkred] (0.3,-.7) to[out=90, in=0] (0,-0.3);
	\draw[->,thick,darkred] (0,-0.3) to[out = 180, in = 90] (-0.3,-.7);
    \node at (0.3,-.8) {$\scriptstyle{i}$};
   \node at (0.25,-0.5) {$\color{darkred}\bullet$};
   \node at (.4,-.5) {$\color{darkred}\scriptstyle{r}$};
	\draw[->,thick,darkred] (-.98,-0.7) to (-.98,0.7);
   \node at (-0.98,0.4) {$\color{darkred}\bullet$};
   \node at (-1.15,0.4) {$\color{darkred}\scriptstyle{t}$};
\end{tikzpicture}
}\right).
\end{align*}
\end{itemize}

\begin{remark}\label{rotated}
A useful consequence of these relations is that $\UU(\g)$ can be 
presented equivalently with
generating 2-morphisms 
$\mathord{
\begin{tikzpicture}[baseline = 1]
	\draw[<-,thick,darkred] (0.08,-.15) to (0.08,.3);
      \node at (0.08,0.13) {$\color{darkred}\bullet$};
   \node at (0.08,.4) {$\scriptstyle{i}$};
\end{tikzpicture}
}
{\scriptstyle\lambda}$,
$\mathord{
\begin{tikzpicture}[baseline = 1]
	\draw[<-,thick,darkred] (0.18,-.15) to (-0.18,.3);
	\draw[<-,thick,darkred] (-0.18,-.15) to (0.18,.3);
   \node at (-0.18,.4) {$\scriptstyle{i}$};
   \node at (0.18,.4) {$\scriptstyle{j}$};
\end{tikzpicture}
}
{\scriptstyle\lambda}$,
$\mathord{
\begin{tikzpicture}[baseline = -2]
	\draw[-,thick,darkred] (0.3,0.2) to[out=-90, in=0] (0.1,-0.1);
	\draw[->,thick,darkred] (0.1,-0.1) to[out = 180, in = -90] (-0.1,0.2);
    \node at (0.3,.3) {$\scriptstyle{i}$};
\end{tikzpicture}
}\!
{\scriptstyle\lambda}$ and
$\mathord{
\begin{tikzpicture}[baseline = -3]
	\draw[-,thick,darkred] (0.3,-0.1) to[out=90, in=0] (0.1,0.2);
	\draw[->,thick,darkred] (0.1,0.2) to[out = 180, in = 90] (-0.1,-0.1);
    \node at (0.3,-.2) {$\scriptstyle{i}$};
\end{tikzpicture}
}\!
{\scriptstyle\lambda}$
subject to
the negative quiver Hecke relations from (K1), the left adjunction
relations from (K2), and
replacing the original inversion relations by the dual inversion
relations from (K5)
where
$$
\mathord{
\begin{tikzpicture}[baseline = 0]
	\draw[->,thick,darkred] (0.28,-.3) to (-0.28,.4);
	\draw[<-,thick,darkred] (-0.28,-.3) to (0.28,.4);
   \node at (0.28,-.4) {$\scriptstyle{i}$};
   \node at (0.28,.5) {$\scriptstyle{j}$};
   \node at (.4,.05) {$\scriptstyle{\lambda}$};
\end{tikzpicture}
}
:=
\mathord{
\begin{tikzpicture}[baseline = 0]
	\draw[<-,thick,darkred] (0.3,-.5) to (-0.3,.5);
	\draw[-,thick,darkred] (-0.2,-.2) to (0.2,.3);
        \draw[-,thick,darkred] (0.2,.3) to[out=50,in=180] (0.5,.5);
        \draw[-,thick,darkred] (0.5,.5) to[out=0,in=90] (0.8,-.5);
        \draw[-,thick,darkred] (-0.2,-.2) to[out=230,in=0] (-0.5,-.5);
        \draw[->,thick,darkred] (-0.5,-.5) to[out=180,in=-90] (-0.8,.5);
  \node at (-0.3,.6) {$\scriptstyle{j}$};
   \node at (0.8,-.6) {$\scriptstyle{i}$};
   \node at (1.05,.05) {$\scriptstyle{\lambda}$};
\end{tikzpicture}
}.
$$
In pictures, this new presentation is the original presentation
from Definition~\ref{def1} rotated
through 180$^\circ$.
\end{remark}

It still seems somewhat remarkable that all of the above relations
(K1)--(K8) can be derived from Rouquier's minimalistic definition. 
Independently, Khovanov
and Lauda \cite{KL3} introduced a strict 2-category
incorporating extra  generators and relations essentially matching the ones
above; see also \cite{CL} which extended the definition in \cite{KL3}
to more general parameters.
The following is a consequence of (K1)--(K8).

\begin{theorem}[{\cite[Main Theorem]{B}}]\label{bt}
Rouquier's Kac-Moody 2-category $\UU(\g)$ 
from Definition~\ref{def1} is isomorphic to the 
2-category introduced in \cite{KL3, CL}.
\end{theorem}

Khovanov and Lauda exploited their extra generators and
relations to construct some explicit sets of 2-morphisms\footnote{They
worked with a restricted choice of parameters compared to here, but
it is clear how to extend their constructions to the general case using (K1)--(K8).} 
which they showed span the 2-morphism spaces in $\UU(\g)$;
see \cite[Proposition 3.11]{KL3}.
They then conjectured that these spanning sets actually give 
bases for the 2-morphism spaces in $\UU(\g)$.
This is the {\em Nondegeneracy Condition} from \cite[$\S$3.2.3]{KL3}.
For $\g = \mathfrak{sl}_n$, it is known that
the Nondegeneracy Condition holds, thanks to \cite[$\S$6.4]{KL3}.

To explain the significance of the Nondegeneracy Condition, 
let $\dot\UU(\g)$ be the additive Karoubi envelope of $\UU(\g)$
and $K_0(\dot\UU(\g)) = \bigoplus_{\lambda, \mu \in P} 1_\mu
K_0(\dot\UU(\g)) 1_\lambda$ be its Grothendieck ring as in (\ref{gring}).
Also let $\dot{U}(\g)_\Z$ be the Kostant
$\Z$-form for $\dot{U}(\g)$, i.e. the subring generated by the idempotents
$1_\lambda$ and the divided powers $e_i^{(r)} 1_\lambda, f_i^{(r)}
1_\lambda$ for $\lambda \in P, i \in I$ and $r \geq 1$. 
Then the arguments used to prove \cite[Theorem 1.1]{KL3} (which are based
ultimately on Lemmas~\ref{divpower} and
\ref{catserre}) 
show that there is a unique surjective
locally unital homomorphism
\begin{equation}\label{gamma}
\gamma:\dot{U}(\g)_\Z \twoheadrightarrow K_0(\dot\UU(\g))
\end{equation}
which sends $e_i 1_\lambda$ and $f_i
1_\lambda$ to $[E_i 1_\lambda]$ and $[F_i
1_\lambda]$, respectively.
Moreover, {\em assuming the parameters satisfy the homogeneity property}
\begin{equation}\label{hc}
s_{ij}^{pq} \neq 0 
 \Rightarrow 
p d_{ji} + q d_{ij} = d_{ij}d_{ji}
\end{equation}
and
{\em providing that the Nondegeneracy
Condition holds}, \cite[Theorem 1.2]{KL3} implies that
$\gamma$ is an isomorphism.
This makes precise the sense in which $\UU(\g)$ should categorify the
universal enveloping algebra of $\g$.

\begin{remark}
In finite type, it is known that $\gamma$ is an isomorphism
(regardless of whether (\ref{hc}) holds); see Corollary~\ref{gr} below.
\end{remark}

\begin{remark}\label{dr}
As we were finalizing this article, Webster released a preliminary
version of \cite{erratum}. In this work, he
appears to have found a general proof of the Nondegeneracy Condition
valid for all types and all choices of parameters satisfying (\ref{hc}). 
\end{remark}

\subsection{Gradings}
In this subsection, we discuss the graded version 
$\UU_q(\g)$ of
$\UU(\g)$ and its connection to quantum groups.
Our language is based on \cite[$\S$6]{BE}, and is
slightly different to that of
\cite{KL3, Rou}.

Let $\GVec$ be the symmetric monoidal category of (small) $\Z$-graded vector spaces and
degree-preserving linear maps. The 
grading shift functor gives an automorphism
$$
Q:\GVec \rightarrow \GVec,
$$ 
our convention for this being that $(QV)_n = V_{n-1}$.
By a {\em graded category},
%\footnote{Elsewhere in the literature, a graded category means a category such as $\GVec$
%which is equipped with an
%automorphism $Q$ (e.g. see \cite[Definition
%5.1]{BLW}).}, 
we mean a
category enriched in $\GVec$.
Thus, the morphism spaces in a graded category are equipped with
a $\Z$-grading in a way that
is compatible with composition. 
If $\C$ is any graded category, the {\em underlying category}
$\underline{\C}$ is the category with the same objects as $\C$, but only the
homogeneous morphisms of degree zero.
For example, $\GVec$ is the underlying category of the graded category
$\mathcal{GV}ec$ whose objects are (small) $\Z$-graded vector spaces and whose
morphisms are sums of homogeneous linear maps of various degrees.

Let $\GCat$ be the category of all (small)
graded categories. It is monoidal with product
$\boxtimes$ defined just like in $\S$\ref{S2c}. 
A {\em strict graded 2-category} is a category enriched in $\GCat$;
cf. Definition~\ref{strict2cat}.
If $\CC$ is any graded 2-category, the {\em underlying
  2-category} $\underline{\CC}$
is the 2-category with the same
objects and 1-morphisms as $\CC$, but only the homogeneous 2-morphisms
of degree zero. 
Also let $\dot{\underline{\CC}}$ be the 
additive Karoubi envelope of $\underline{\CC}$,
and $K_0(\dot{\underline{\CC}})$ be its Grothendieck ring defined as in
(\ref{gring}).

For any graded 2-category $\CC$, there is a universal
construction of another graded 2-category $\CC_q$, which we call the
{\em $Q$-envelope} of $\CC$.
It has the same object set as $\CC$. Given objects $\lambda,\mu$, the
1-morphisms
$\lambda \rightarrow \mu$ in $\CC_q$ are 
defined formally to be symbols $Q^n F$ for all 1-morphisms
$F:\lambda\rightarrow\mu$ in $\CC$ and all $n \in \Z$.
Then the 2-morphisms in $\CC_q$ are defined from
$$
\Hom_{\CC_q}(Q^n F, Q^m G) := Q^{m-n} \Hom_{\CC}(F, G),
$$
where the $Q^{m-n}$ on the right hand side is the grading shift in
$\GVec$.
Horizontal composition of 1-morphisms in $\CC_q$ is induced 
by the horizontal composition in $\CC$ so that
$(Q^n F) (Q^m G) := Q^{n+m} F G$.
Similarly, the horizontal and vertical compositions of 2-morphisms in
$\CC_q$ are induced in an obvious way 
by the horizontal and vertical compositions in $\CC$.
Note also for each 1-morphism $F$ of $\CC$ and $n \in \Z$
that $1_F$ defines a canonical 2-isomorphism $Q^n F
\stackrel{\sim}{\rightarrow} F$ in $\CC_q$ that is homogeneous of degree $-n$.

The point of the construction of $Q$-envelope 
is that $\CC_q$ (hence, the underlying 2-category $\underline{\CC}_q$)
is equipped with 
distinguished 1-morphisms $q_\lambda := Q 1_\lambda:\lambda\rightarrow
\lambda$ for each object
$\lambda$, 
such that $q_\mu F = F q_\lambda$ for each 1-morphism
$F:\lambda\rightarrow \mu$.
In particular, the Grothendieck ring $K_0(\dot{\underline{\CC}}_q)$ is actually a
$\Z[q,q^{-1}]$-algebra, with $q$ acting on $1_\mu K_0(\dot{\underline{\CC}}_q)
1_\lambda$ by left multiplication by $[q_\mu]$ ($=$ right
multiplication by $[q_\lambda]$).

\begin{definition}
Assume that the parameters fixed in
$\S$\ref{snow} satisfy (\ref{hc}).
Then the relations defining the Kac-Moody 2-category $\UU(\g)$
are all homogeneous, so we can make $\UU(\g)$ into a graded 2-category by declaring that the generating
2-morphisms
$
\mathord{
\begin{tikzpicture}[baseline = -2]
	\draw[->,thick,darkred] (0.08,-.15) to (0.08,.3);
      \node at (0.08,0.05) {$\color{darkred}\bullet$};
   \node at (0.08,-.25) {$\scriptstyle{i}$};
\end{tikzpicture}
}
{\scriptstyle\lambda}
$,
$\mathord{
\begin{tikzpicture}[baseline = -2]
	\draw[->,thick,darkred] (0.18,-.15) to (-0.18,.3);
	\draw[->,thick,darkred] (-0.18,-.15) to (0.18,.3);
   \node at (-0.18,-.25) {$\scriptstyle{i}$};
   \node at (0.18,-.25) {$\scriptstyle{j}$};
\end{tikzpicture}
}
{\scriptstyle\lambda}
$,
$\mathord{
\begin{tikzpicture}[baseline = 0]
	\draw[<-,thick,darkred] (0.3,0.2) to[out=-90, in=0] (0.1,-0.1);
	\draw[-,thick,darkred] (0.1,-0.1) to[out = 180, in = -90] (-0.1,0.2);
    \node at (-0.1,.3) {$\scriptstyle{i}$};
\end{tikzpicture}
}\,
{\scriptstyle\lambda}
$ and
$\mathord{
\begin{tikzpicture}[baseline = -2]
	\draw[<-,thick,darkred] (0.3,-0.1) to[out=90, in=0] (0.1,0.2);
	\draw[-,thick,darkred] (0.1,0.2) to[out = 180, in = 90] (-0.1,-0.1);
    \node at (-0.1,-.2) {$\scriptstyle{i}$};
\end{tikzpicture}
}
\,{\scriptstyle\lambda}$
are of degrees $2d_i$, $d_i d_{ij}$, $d_i(1+\langle
h_i,\lambda\rangle)$ and $d_i(1-\langle h_i,\lambda\rangle)$,
respectively.
Then we define the {\em graded Kac-Moody 2-category} $\UU_q(\g)$ to be
the $Q$-envelope of this graded 2-category.
\end{definition}

Using Theorem~\ref{bt}, it is easy to see that the underlying
2-category $\underline{\UU}_q(\g)$ is isomorphic to the graded version
of the Kac-Moody 2-category as defined in \cite{Rou, KL3}.
Then \cite[Theorem 1.1]{KL3}
shows that there is a surjective $\Z[q,q^{-1}]$-algebra homomorphism
\begin{equation}\label{gammaq}
\gamma_q:
\dot U_q(\g)_\Z
\twoheadrightarrow K_0(\dot{\underline{\UU}}_q(\g)),
\end{equation}
where $\dot U_q(\g)_\Z$ denotes the Kostant-Lusztig $\Z$-form of the
idempotented version of the quantized enveloping algebra $U_q(\g)$.
Moreover, {\em providing that the Nondegeneracy Condition holds}, 
\cite[Theorem 1.2]{KL3} shows 
that $\gamma_q$ is an isomorphism.
The injectivity of $\gamma$ claimed earlier in the
ungraded setting follows from this on specializing
$q$ to $1$.

\section{Categorical actions and crystals}

In this final section, we focus on 2-representations of Kac-Moody
2-categories. We will review the existing results mostly
following Rouquier \cite{R2}.
After that, we focus on the locally Schurian
case, explaining various results 
in that setting which generalize aspects of \cite{CR, Rou}.

\subsection{2-representations}
We keep the choices of Kac-Moody data and parameters as in the previous section.
The following is \cite[Definition 5.1.1]{Rou}.

\begin{definition}\label{2rep}
A {\em 2-representation} of $\UU(\g)$
is the following data:
\begin{itemize}
\item[(M1)]
a category $\RR$ with a given decomposition
$\RR = \coprod_{\lambda \in P} \RR_\lambda$ into {\em weight subcategories};
\item[(M2)]
endofunctors $E_i$ and $F_i$ of $\RR$ for each $i \in I$,
such that $E_i|_{\RR_\lambda}:\RR_\la \rightarrow \RR_{\la+\alpha_i}$
and $F_i|_{\RR_\la}:\RR_\la \rightarrow \RR_{\la-\alpha_i}$;
\item[(M3)]
natural transformations $x_i:E_i \rightarrow E_i$
and $\tau_{ij}:E_i E_j \rightarrow E_j E_i$
satisfying the positive quiver Hecke relations from Definition~\ref{pd}, i.e. so that
there is a
strict monoidal functor $\Phi:\H \rightarrow \mathcal{E}\!nd(\RR)$ with
$\Phi(i) = E_i$,
$\Phi\big(\mathord{
\begin{tikzpicture}[baseline = -2]
	\draw[->,thick,darkred] (0.08,-.15) to (0.08,.3);
      \node at (0.08,0.05) {$\color{darkred}\bullet$};
   \node at (0.08,-.25) {$\scriptstyle{i}$};
\end{tikzpicture}
}\big) = x_i$ and $\Phi\big(\mathord{
\begin{tikzpicture}[baseline = -2]
	\draw[->,thick,darkred] (0.18,-.15) to (-0.18,.3);
	\draw[->,thick,darkred] (-0.18,-.15) to (0.18,.3);
   \node at (-0.18,-.25) {$\scriptstyle{i}$};
   \node at (0.18,-.25) {$\scriptstyle{j}$};
\end{tikzpicture}
}\big) = \tau_{ij}$
for all $i,j \in I$;
\item[(M4)]
the unit $\eta_i:1_\RR \rightarrow F_i E_i$ and counit $\eps_i:E_i F_i
\rightarrow 1_\RR$ of an adjunction making $F_i$ into a right adjoint
of $E_i$.
\end{itemize}
Then we require that the following axiom holds:
\begin{itemize}
\item[(M5)]
all of the natural transformations $\sigma_{ij}\:(i \neq j)$
and $\rho_{i,\la}$ are invertible, where
\begin{align*}
\sigma_{ij} &:= F_i E_j \eps_i \circ F_i \tau_{ij} F_i  \circ \eta_i
E_j F_i:E_j F_i \rightarrow F_i E_j,\\
\rho_{i,\la} &:=
\left\{
\begin{array}{ll}
\displaystyle\sigma_{ii} \oplus \bigoplus_{n=0}^{\hhh-1} \eps_i
\circ (x_i F_i)^{\circ n}:E_i
F_i|_{\RR_\lambda} \rightarrow F_i E_i|_{\RR_\lambda} \oplus 1_{\RR_\lambda}^{\oplus
  \hhh}
&\text{if $\hhh \geq 0$,}\\
\displaystyle\sigma_{ii} \oplus \bigoplus_{m=0}^{-\hhh-1}
(F_i x_i)^{\circ m}\circ \eta_i:E_i F_i|_{\RR_\la} \oplus 1_{\RR_\lambda}^{\oplus
  -\hhh} \rightarrow F_i E_i|_{\RR_\la}
&\text{if $\hhh \leq 0$,}
\end{array}\right.
\end{align*}
for $h := \langle h_i, \lambda
\rangle$.
\end{itemize}
\end{definition}

We say that a 2-representation $\RR$ is {\em small}, {\em finite-dimensional}, {\em additive},
{\em Abelian} etc... if all of the categories $\RR_\lambda\:(\lambda
\in P)$ are small, finite-dimensional, additive, Abelian etc...
In the additive case, the functors $E_i$ and $F_i$ extend 
to $\bigoplus_{\lambda \in P} \RR_\lambda$, and it is more convenient to
denote this by $\RR$ in place of $\coprod_{\lambda \in P} \RR_\lambda$.
If $\RR$ is not additive, one can always replace it by its 
additive envelope, or indeed its additive Karoubi envelope $\dot\RR$; the 
endofunctors $E_i$ and $F_i$ extend canonically to make these into
2-representations too.

The point of Definition~\ref{2rep}
is that a small 2-representation $\RR$ is exactly the same data as a
strict 2-functor $\mathbb{R}:\UU(\g) \rightarrow \CCat$.
The dictionary for going between the two notions is given by
$\RR_\lambda = \mathbb{R}(\lambda)$,
$E_i|_{\RR_\lambda} = \mathbb{R}(E_i 1_\lambda)$,
$F_i|_{\RR_\lambda} = \mathbb{R}(F_i 1_\lambda)$,
$x_i|_{E_i|_{\RR_\lambda}} = \mathbb{R}\big(
\mathord{
\begin{tikzpicture}[baseline = -2]
	\draw[->,thick,darkred] (0.08,-.15) to (0.08,.3);
      \node at (0.08,0.05) {$\color{darkred}\bullet$};
   \node at (0.08,-.25) {$\scriptstyle{i}$};
\end{tikzpicture}
}
{\scriptstyle\lambda}
\big)$,
$
\tau_{ij}|_{E_i E_j|_{\RR_\lambda}}
= 
\mathbb{R}\big(\mathord{
\begin{tikzpicture}[baseline = -2]
	\draw[->,thick,darkred] (0.18,-.15) to (-0.18,.3);
	\draw[->,thick,darkred] (-0.18,-.15) to (0.18,.3);
   \node at (-0.18,-.25) {$\scriptstyle{i}$};
   \node at (0.18,-.25) {$\scriptstyle{j}$};
\end{tikzpicture}
}
{\scriptstyle\lambda}\big)$,
$\eta_i|_{1_{\RR_\lambda}}
= \mathbb{R}\big(\mathord{
\begin{tikzpicture}[baseline = 0]
	\draw[<-,thick,darkred] (0.3,0.2) to[out=-90, in=0] (0.1,-0.1);
	\draw[-,thick,darkred] (0.1,-0.1) to[out = 180, in = -90] (-0.1,0.2);
    \node at (-0.1,.3) {$\scriptstyle{i}$};
\end{tikzpicture}
}\,
{\scriptstyle\lambda}\big)$
and
$\eps_i|_{F_i E_i|_{\RR_\lambda}}
= \mathbb{R}\big(\mathord{
\begin{tikzpicture}[baseline = -2]
	\draw[<-,thick,darkred] (0.3,-0.1) to[out=90, in=0] (0.1,0.2);
	\draw[-,thick,darkred] (0.1,0.2) to[out = 180, in = 90] (-0.1,-0.1);
    \node at (-0.1,-.2) {$\scriptstyle{i}$};
\end{tikzpicture}
}
\,{\scriptstyle\lambda}\big)$.

The following strengthens \cite[Theorem 5.16]{Rou}.

\begin{lemma}\label{nec}
Suppose that $\RR$ is a 2-representation in the sense of Definition~\ref{2rep}.
Then there is a canonical
%\footnote{Although natural in $\RR$ in the appropriate
%sense to warrant being called ``canonical,'' we stress that it depends on the fixed choice of scalars
%$c_{\lambda;i}$ made in (K2).} 
choice for the unit
$\eta_i':1_\RR \rightarrow E_i F_i$ and counit 
$\eps_i':F_i E_i \rightarrow 1_\RR$
of an adjunction making $F_i$ into a
left adjoint of $E_i$.
\end{lemma}

\begin{proof}
We may assume that $\RR$ is small, so that there is a
corresponding strict 2-functor $\mathbb{R}:\UU(\g) \rightarrow \CCat$.
Then let $\eta'_i|_{1_{\RR_\lambda}} := \mathbb{R}\big(\:
\mathord{
\begin{tikzpicture}[baseline = -2]
	\draw[-,thick,darkred] (0.3,0.2) to[out=-90, in=0] (0.1,-0.1);
	\draw[->,thick,darkred] (0.1,-0.1) to[out = 180, in = -90] (-0.1,0.2);
    \node at (0.3,.3) {$\scriptstyle{i}$};
\end{tikzpicture}
}\!\!
{\scriptstyle\lambda}
\big)$
and $\eps'_i|_{F_i E_i|_{\RR_\lambda}} := \mathbb{R}\big(\:\mathord{
\begin{tikzpicture}[baseline = -3]
	\draw[-,thick,darkred] (0.3,-0.1) to[out=90, in=0] (0.1,0.2);
	\draw[->,thick,darkred] (0.1,0.2) to[out = 180, in = 90] (-0.1,-0.1);
    \node at (0.3,-.2) {$\scriptstyle{i}$};
\end{tikzpicture}
}
{\scriptstyle\lambda}\big)$, where notation is as in (K2) from $\S$\ref{ji}.
\end{proof}

The other generators and relations from (K1)--(K8)
can be transported to any 2-representation $\RR$ in a similar way.
For example, from the images of the downward dots and crossings from
(K1),  
one obtains canonical natural transformations
$x'_i:F_i \rightarrow F_i$ and $\tau_{ij}':F_i F_j \rightarrow F_j
F_i$ which satisfy the negative quiver Hecke relations.

\begin{remark}\label{2repb}
Definition~\ref{2rep} 
can be formulated equivalently by
replacing (M3)--(M5) by
\begin{itemize}
\item[(M3$'$)]
natural transformations $x'_i:F_i \rightarrow F_i$
and $\tau'_{ij}:F_i F_j \rightarrow F_j F_i$
satisfying the negative quiver Hecke relations, 
i.e. so that there is a strict monoidal functor
$\Psi:\H' \rightarrow \mathcal{E}\!nd(\RR)$
with 
$\Psi(i) = F_i$,
$\Psi\big(\mathord{
\begin{tikzpicture}[baseline = 0]
	\draw[<-,thick,darkred] (0.08,-.15) to (0.08,.3);
      \node at (0.08,0.15) {$\color{darkred}\bullet$};
   \node at (0.08,.35) {$\scriptstyle{i}$};
\end{tikzpicture}
}\big) = x'_i$ and $\Phi\big(\mathord{
\begin{tikzpicture}[baseline = 0]
	\draw[<-,thick,darkred] (0.18,-.15) to (-0.18,.3);
	\draw[<-,thick,darkred] (-0.18,-.15) to (0.18,.3);
   \node at (-0.18,.4) {$\scriptstyle{j}$};
   \node at (0.18,.4) {$\scriptstyle{i}$};
\end{tikzpicture}
}\big) = \tau'_{ij}$
for all $i,j \in I$;
\item[(M4$'$)]
the unit $\eta'_i:1_\RR \rightarrow E_i F_i$ and counit $\eps'_i:F_i E_i
\rightarrow 1_\RR$ of an adjunction making $F_i$ into a left adjoint
of $E_i$;
\item[(M5$'$)]
all of the natural transformations $\sigma'_{ij}\:(i \neq j)$ and $\rho'_{i,\la}$ are
invertible, where
\begin{align*}
\sigma'_{ij} &:= E_i F_j \eps'_i \circ E_i \tau_{ij}' E_i  \circ \eta_i'
F_j E_i:F_j E_i \rightarrow E_i F_j,\\
\rho'_{i,\la} &:=
\left\{
\begin{array}{ll}
\displaystyle\sigma'_{ii} \oplus \bigoplus_{n=0}^{\hhh-1}(E_i x'_i)^{\circ n} \circ \eta'_i:F_i
E_i|_{\RR_\lambda} 
\oplus 1_{\RR_\lambda}^{\oplus
  \hhh}
\rightarrow E_i F_i|_{\RR_\lambda} 
&\text{if $\hhh \geq 0$,}\\
\displaystyle\sigma'_{ii} \oplus \bigoplus_{m=0}^{-\hhh-1}
\eps'_i \circ (x'_i E_i)^{\circ m}:F_i E_i|_{\RR_\la}\rightarrow E_i
F_i|_{\RR_\la}
 \oplus 1_{\RR_\lambda}^{\oplus
  -\hhh} 
&\text{if $\hhh \leq 0$,}
\end{array}\right.
\end{align*}
for $h := \langle h_i, \lambda \rangle$.
\end{itemize}
This follows because, in view of the alternative presentation of
$\UU(\g)$ from Remark~\ref{rotated}, 
the new formulation is also the data of a strict 2-functor.
\end{remark}

If $\RR$ is a 2-representation, the endofunctors $E_i$ and $F_i$
induce endomorphisms $[E_i]$ and $[F_i]$ of
the split Grothendieck group
\begin{equation}\label{wd}
K_0(\dot \RR) = \bigoplus_{\lambda \in P} K_0(\dot\RR_\lambda)
\end{equation}
of its additive Karoubi envelope.

\begin{lemma}\label{gg}
Given a 2-representation $\RR$,
there is a unique way to make $K_0(\dot\RR)$ 
into a module over the Kostant $\Z$-form $U(\g)_\Z$ of the universal enveloping algebra of $\g$
so that the Chevalley generators $e_i, f_i$ 
act as $[E_i], [F_i]$, respectively, 
and (\ref{wd}) is its decomposition into weight spaces.
\end{lemma}

\begin{proof}
We may assume that $\RR$ (hence, $\dot \RR$) is small, so that there is
a corresponding strict 2-functor
$\dot{\mathbb{R}}:\dot\UU(\g) \rightarrow \CCat$
with $\dot{\mathbb{R}}(\lambda) = \dot\RR_\lambda$, etc... 
The definition of strict 2-functor then ensures that $K_0(\dot\RR)$ is a
module over $K_0(\dot\UU(\g))$:
for $F \in \ob \mathcal{H}om_{\dot\UU(\g)}(\lambda,\mu)$
defining $[F] \in K_0(\dot\UU(\g))$, and $P \in \ob \dot\RR_\la$ 
defining $[P] \in K_0(\dot\RR_\lambda)$, we set 
$[F][P] := [\dot{\mathbb{R}}(F)(P)] \in K_0(\dot\RR_\mu)$.
It remains to lift the action of $K_0(\dot\UU(\g))$ to $\dot U(\g)_\Z$
using the homomorphism $\gamma$ from (\ref{gamma});
this does not depend on the injectivity of
$\gamma$. 
\end{proof}

Lemma~\ref{gg} shows for any 2-representation $\RR$ that $\mathbb{C}
\otimes_\Z K_0(\dot\RR)$ is a $\g$-module in a canonical way.
Typically, it is an {\em integrable} $\g$-module in view of the next lemma.

\begin{lemma}\label{maid}
If $\RR$ is a finite-dimensional 2-representation, then
$\mathbb{C}\otimes_{\Z} K_0(\dot\RR)$ is an integrable $\g$-module.
\end{lemma}

\begin{proof}
The local nilpotency of $e_i$ follows from Lemma~\ref{proud}.
Similar arguments with the negative nil Hecke category show that $f_i$
is locally nilpotent too.
\end{proof}

\begin{definition}\label{f}
Let $\RR$ and $\SS$ be 2-representations.
A {\em strongly equivariant functor}
$G:\RR \rightarrow \SS$
is a functor such that $G|_{\RR_\lambda}:\RR_\lambda \rightarrow
\SS_\lambda$ for each $\lambda \in P$,
plus
natural isomorphisms
$\zeta_{i}= \zeta_{G,i}:E_iG \Rightarrow G E_i$,
such that the following hold for each $i \in I$:
\begin{itemize}
\item[(E1)]
the natural transformation $F_i G \eps_i \circ F_i \zeta_{i} F_i \circ
\eta_i G F_i:G F_i \Rightarrow F_i G$ is invertible,
with inverse denoted $\zeta_i':F_i G \Rightarrow G F_i$;
\item[(E2)]
we have that $G x_i \circ \zeta_{i} = \zeta_{i} \circ x_i G$;
\item[(E3)] we have that $G \tau_{ij} \circ \zeta_{i} E_j \circ E_i \zeta_{j} = \zeta_{j}
  E_i \circ E_j \zeta_{i} \circ \tau_{ij} G$.
\end{itemize}
There is an obvious way to make the composition $GH$ of two strongly
equivariant functors into a strongly equivariant functor in its own
right: set $\zeta_{GH,i} := \zeta_{G,i} H \circ G \zeta_{H,i}$. 
Also the identity functor $1$ is strongly equivariant with
$\zeta_{1,i} := 1_{E_i}$.
Let $\mathcal{R}ep(\UU(\g))$ be the resulting  category of (small)
2-representations and strongly equivariant functors.
\end{definition}

\begin{definition}\label{sees}
A {\em strongly equivariant equivalence} is 
an equivalence of categories $G:\RR\rightarrow\SS$ such that $G|_{\mathcal R_\lambda}:\RR_\lambda
\rightarrow \SS_\lambda$ for each $\lambda$, plus isomorphisms
$\zeta_i:E_i G \Rightarrow G E_i$ satisfying (E2)--(E3) for each $i
\in I$.
\end{definition}

\begin{remark}
In the situation of Definition~\ref{sees}, the axiom (E1) holds automatically, i.e. strongly
equivariant equivalences are strongly equivariant functors.
To see this, fix a right
adjoint $H$ to $G$. The given  isomorphism $\zeta_i:E_i G \Rightarrow G E_i$ induces a canonical isomorphism
$F_i H \Rightarrow H F_i$ between the right adjoints of $GE_i$ and
$E_i G$.
Horizontally composing on the left and right with $G$ and using the isomorphisms
$1 \Rightarrow HG$ and $GH \Rightarrow 1$ defined by the adjunction,
we get an isomorphism $G F_i \Rightarrow F_i G$. This is precisely the
natural transformation written down in (E1).
\end{remark}

It is helpful to 
interpret Definition~\ref{f} in terms of the
string calculus for 2-morphisms in $\CCat$.
We represent $\zeta_{i}$ by 
$\mathord{
\begin{tikzpicture}[baseline = -2]
	\draw[-,thick,dotted,darkred] (0.18,-.15) to (-0.18,.3);
	\draw[->,thick,darkred] (-0.18,-.15) to (0.18,.3);
   \node at (-0.18,-.25) {$\scriptstyle{i}$};
%   \node at (0.18,-.25) {$\scriptstyle{G}$};
\end{tikzpicture}
}\:
$ (so the dotted line is the identity morphism $1_G$, to the right of which is the
category $\RR$ and to the left is $\SS$).
Then it is natural to denote its inverse by
$\mathord{
\begin{tikzpicture}[baseline = -2]
	\draw[-,thick,dotted,darkred] (-0.18,-.15) to (0.18,.3);
	\draw[->,thick,darkred] (0.18,-.15) to (-0.18,.3);
   \node at (0.18,-.25) {$\scriptstyle{i}$};
%   \node at (-0.18,-.25) {$\scriptstyle{G}$};
\end{tikzpicture}
}
$.
The 2-morphism in (E1) is 
$
\mathord{
\begin{tikzpicture}[baseline = 0]
	\draw[<-,thick,darkred] (0.18,-.15) to (-0.18,.3);
	\draw[-,thick,dotted,darkred] (-0.18,-.15) to (0.18,.3);
   \node at (-0.18,.4) {$\scriptstyle{i}$};
%   \node at (0.18,-.25) {$\scriptstyle{G}$};
\end{tikzpicture}
}:=
\mathord{
\begin{tikzpicture}[baseline = -2]
	\draw[-,dotted,thick,darkred] (0.2,-.4) to (-0.2,.4);
	\draw[-,thick,darkred] (-0.15,-.2) to (0.15,.3);
        \draw[-,thick,darkred] (0.15,.3) to[out=50,in=180] (0.3,.4);
        \draw[->,thick,darkred] (0.3,.4) to[out=0,in=90] (0.6,-.4);
        \draw[-,thick,darkred] (-0.15,-.2) to[out=230,in=0] (-0.4,-.4);
        \draw[-,thick,darkred] (-0.4,-.4) to[out=180,in=-90] (-0.6,.4);
  \node at (-0.6,.5) {$\scriptstyle{i}$};
\end{tikzpicture}
}\,$.
We denote its inverse $\zeta_i'$ by 
$\mathord{
\begin{tikzpicture}[baseline = 0]
	\draw[<-,thick,darkred] (-0.18,-.15) to (0.18,.3);
	\draw[-,thick,dotted,darkred] (0.18,-.15) to (-0.18,.3);
   \node at (0.18,.4) {$\scriptstyle{i}$};
%   \node at (-0.18,-.25) {$\scriptstyle{G}$};
\end{tikzpicture}
}
$.
The axioms (E2) and (E3) are the identities
\begin{equation}\label{shower1}
\mathord{
\begin{tikzpicture}[baseline = 0]
	\draw[<-,thick,darkred] (0.25,.6) to (-0.25,-.2);
	\draw[-,dotted,thick,darkred] (0.25,-.2) to (-0.25,.6);
  \node at (-0.25,-.26) {$\scriptstyle{i}$};
      \node at (0.13,0.42) {$\color{darkred}\bullet$};
\end{tikzpicture}
}
=
\mathord{
\begin{tikzpicture}[baseline = 0]
	\draw[<-,thick,darkred] (0.25,.6) to (-0.25,-.2);
	\draw[-,dotted,thick,darkred] (0.25,-.2) to (-0.25,.6);
  \node at (-0.25,-.26) {$\scriptstyle{i}$};
      \node at (-0.13,-0.02) {$\color{darkred}\bullet$};
\end{tikzpicture}
}\:,
\qquad
\mathord{
\begin{tikzpicture}[baseline = 0]
	\draw[<-,thick,darkred] (0.45,.8) to (-0.45,-.4);
	\draw[-,dotted,thick,darkred] (0.45,-.4) to (-0.45,.8);
        \draw[-,thick,darkred] (0,-.4) to[out=90,in=-90] (.45,0.2);
        \draw[->,thick,darkred] (0.45,0.2) to[out=90,in=-90] (0,0.8);
   \node at (-0.45,-.45) {$\scriptstyle{i}$};
   \node at (0,-.45) {$\scriptstyle{j}$};
\end{tikzpicture}
}=
\mathord{
\begin{tikzpicture}[baseline = 0]
	\draw[<-,thick,darkred] (0.45,.8) to (-0.45,-.4);
	\draw[-,dotted,thick,darkred] (0.45,-.4) to (-0.45,.8);
        \draw[-,thick,darkred] (0,-.4) to[out=90,in=-90] (-.45,0.2);
        \draw[->,thick,darkred] (-0.45,0.2) to[out=90,in=-90] (0,0.8);
   \node at (-0.45,-.45) {$\scriptstyle{i}$};
   \node at (0,-.45) {$\scriptstyle{j}$};
\end{tikzpicture}
}
\:\,,
\end{equation}
where the dots and solid crossings represent $x_i$ and $\tau_{ij}$, respectively.
Representing $\eta_i,\eta_i', \eps_i$ and $\eps_i'$ by
oriented cups
and caps as usual, one can check further that
\begin{equation}\label{shower2}
\mathord{
\begin{tikzpicture}[baseline = 0]
\draw[-,thick,dotted,darkred](-.5,.4) to (0,-.3);
	\draw[<-,thick,darkred] (0.3,-0.3) to[out=90, in=0] (0,0.2);
	\draw[-,thick,darkred] (0,0.2) to[out = -180, in = 40] (-0.5,-0.3);
    \node at (-0.5,-.4) {$\scriptstyle{i}$};
\end{tikzpicture}
}=\mathord{
\begin{tikzpicture}[baseline = 0]
\draw[-,dotted,thick,darkred](.6,.4) to (.1,-.3);
	\draw[<-,thick,darkred] (0.6,-0.3) to[out=140, in=0] (0.1,0.2);
	\draw[-,thick,darkred] (0.1,0.2) to[out = -180, in = 90] (-0.2,-0.3);
    \node at (-0.2,-.4) {$\scriptstyle{i}$};
\end{tikzpicture}
}\:,
\quad
\mathord{
\begin{tikzpicture}[baseline = 5]
\draw[-,dotted,thick,darkred](-.5,-.3) to (0,.4);
	\draw[<-,thick,darkred] (0.3,0.4) to[out=-90, in=0] (0,-0.1);
	\draw[-,thick,darkred] (0,-0.1) to[out = 180, in = -40] (-0.5,0.4);
    \node at (-0.5,.5) {$\scriptstyle{i}$};
\end{tikzpicture}
}=\mathord{
\begin{tikzpicture}[baseline = 5]
\draw[-,dotted,thick,darkred](.6,-.3) to (.1,.4);
	\draw[<-,thick,darkred] (0.6,0.4) to[out=-140, in=0] (0.1,-0.1);
	\draw[-,thick,darkred] (0.1,-0.1) to[out = 180, in = -90] (-0.2,0.4);
    \node at (-0.2,.5) {$\scriptstyle{i}$};
\end{tikzpicture}
}\,,\quad
\mathord{
\begin{tikzpicture}[baseline = 0]
\draw[-,dotted,thick,darkred](-.5,.4) to (0,-.3);
	\draw[-,thick,darkred] (0.3,-0.3) to[out=90, in=0] (0,0.2);
	\draw[->,thick,darkred] (0,0.2) to[out = -180, in = 40] (-0.5,-0.3);
    \node at (0.3,-.4) {$\scriptstyle{i}$};
\end{tikzpicture}
}=
\mathord{
\begin{tikzpicture}[baseline = 0]
\draw[-,dotted,thick,darkred](.6,.4) to (.1,-.3);
	\draw[-,thick,darkred] (0.6,-0.3) to[out=140, in=0] (0.1,0.2);
	\draw[->,thick,darkred] (0.1,0.2) to[out = -180, in = 90] (-0.2,-0.3);
    \node at (0.6,-.4) {$\scriptstyle{i}$};
\end{tikzpicture}
}\,,\quad
\mathord{
\begin{tikzpicture}[baseline = 5]
\draw[-,dotted,thick,darkred](-.5,-.3) to (0,.4);
	\draw[-,thick,darkred] (0.3,0.4) to[out=-90, in=0] (0,-0.1);
	\draw[->,thick,darkred] (0,-0.1) to[out = 180, in = -40] (-0.5,0.4);
    \node at (0.3,.5) {$\scriptstyle{i}$};
\end{tikzpicture}
}=\mathord{
\begin{tikzpicture}[baseline = 5]
\draw[-,dotted,thick,darkred](.6,-.3) to (.1,.4);
	\draw[-,thick,darkred] (0.6,0.4) to[out=-140, in=0] (0.1,-0.1);
	\draw[->,thick,darkred] (0.1,-0.1) to[out = 180, in = -90] (-0.2,0.4);
    \node at (0.6,.5) {$\scriptstyle{i}$};
\end{tikzpicture}
}.
\end{equation}
The first two of these are almost immediate; the second two follow
using the inversion relations and the definitions of the leftward cups and
caps in $\UU(\g)$.

\begin{remark}\label{yyy}
If $\RR$, $\SS$ are small 2-representations with associated
2-functors $\R$, $\SSS$,
the data of a strongly equivariant functor $G:\RR
\rightarrow \SS$ is the same as the data of a morphism of 2-functors
$G:\R \rightarrow \SSS$
as in \cite[Definition 2.3]{Rou}
(with $G(\lambda) = G|_{\RR_\lambda}$).
Indeed, given any 1-morphism $u:\lambda\rightarrow \mu$ in $\UU(\g)$,
i.e. an appropriate composition of several $E_i$ and $F_j$ applied to $1_\lambda$, 
the corresponding horizontal composition of $\zeta_i$ and $\zeta_j'$
defines a natural isomorphism $\zeta(u):\SSS(u) G(\lambda)
\Rightarrow 
G(\mu)
\R(u)$ satisfying the axioms of a morphism of 2-functors.
The non-trivial part about this assertion is the naturality of $\zeta(u)$,
i.e. the statement that 
$G(\mu) \R(\xi) \circ \zeta(u) = \zeta(v) \circ \SSS(\xi)
G(\lambda)$
for all 2-morphisms $\xi:u \Rightarrow v$. In pictures:
$$
\scriptstyle{\mu}
\mathord{
\begin{tikzpicture}[baseline = 4]
	\draw[-,dotted,thick,darkred] (0.5,-.4) to (-0.12,.8);
	\draw[-,thick,darkred] (-0.14,-.4) to (0.27,.4);
	\draw[-,thick,darkred] (-0.10,-.4) to (0.31,.4);
	\draw[-,thick,darkred] (0.45,.8) to (0.37,.64);
	\draw[-,thick,darkred] (0.49,.8) to (0.41,.64);
      \draw[thick,darkred] (0.34,0.52) circle (4pt);
   \node at (0.34,0.52) {\color{darkred}$\scriptstyle{\xi}$};
   \node at (-0.12,-.53) {$\scriptstyle{u}$};
   \node at (0.5,.9) {$\scriptstyle{v}$};
\end{tikzpicture}
}
\scriptstyle{\lambda}
\quad
=
\quad
\scriptstyle{\mu}
\mathord{
\begin{tikzpicture}[baseline = 4]
	\draw[-,dotted,thick,darkred] (0.5,-.4) to (-0.12,.8);
	\draw[-,thick,darkred] (-0.09,-.4) to (0.01,-.2);
	\draw[-,thick,darkred] (-0.13,-.4) to (-0.03,-.2);
	\draw[-,thick,darkred] (0.52,.8) to (0.11,.04);
	\draw[-,thick,darkred] (0.48,.8) to (0.07,.04);
      \draw[thick,darkred] (0.04,-0.08) circle (4pt);
   \node at (0.04,-0.08) {\color{darkred}$\scriptstyle{\xi}$};
   \node at (-0.12,-.53) {$\scriptstyle{u}$};
   \node at (0.5,.9) {$\scriptstyle{v}$};
\end{tikzpicture}
}
\scriptstyle{\lambda}\:.
$$
The proof of this reduces to checking it in case $\xi$ is an upward dot or
crossing or any cup or cap, since these generate all 2-morphisms in
$\UU(\g)$ thanks to (K5). These cases are covered by (\ref{shower1})--(\ref{shower2}).
\end{remark}

\begin{remark}
Using (\ref{shower2}),
we see in particular that 
$\mathord{
\begin{tikzpicture}[baseline = 0]
	\draw[-,thick,dotted,darkred] (0.18,-.15) to (-0.18,.3);
	\draw[->,thick,darkred] (-0.18,-.15) to (0.18,.3);
   \node at (-0.18,-.25) {$\scriptstyle{i}$};
%   \node at (0.18,-.25) {$\scriptstyle{G}$};
\end{tikzpicture}
}=
\bigg(\mathord{
\begin{tikzpicture}[baseline = -2]
	\draw[-,dotted,thick,darkred] (0.2,-.4) to (-0.2,.4);
	\draw[-,thick,darkred] (-0.15,-.2) to (0.15,.3);
        \draw[-,thick,darkred] (0.15,.3) to[out=50,in=180] (0.3,.4);
        \draw[-,thick,darkred] (0.3,.4) to[out=0,in=90] (0.6,-.4);
        \draw[-,thick,darkred] (-0.15,-.2) to[out=230,in=0] (-0.4,-.4);
        \draw[->,thick,darkred] (-0.4,-.4) to[out=180,in=-90] (-0.6,.4);
  \node at (0.6,-.5) {$\scriptstyle{i}$};
\end{tikzpicture}
}\bigg)^{-1}$, i.e.
the natural transformation
$\zeta_i'$ determines $\zeta_i$.
Indeed, using also Remark~\ref{yyy}, 
one can reformulate
Definition~\ref{f} equivalently
in terms of isomorphisms $\zeta_i':F_i G \Rightarrow G F_i$.
The axioms (E1)--(E3) become:
\begin{itemize}
\item[(E1$'$)]
the natural transformation $E_i G \eps'_i \circ E_i \zeta'_{i} E_i \circ
\eta'_i G E_i:G E_i \Rightarrow E_i G$ is invertible,
with inverse denoted $\zeta_i:E_i G \Rightarrow G E_i$;
\item[(E2$'$)]
we have that $G x'_i \circ \zeta'_{i} = \zeta'_{i} \circ x'_i G$;
\item[(E3$'$)] we have that $G \tau'_{ij} \circ \zeta'_{i} F_j \circ F_i \zeta'_{j} = \zeta'_{j}
  F_i \circ F_j \zeta'_{i} \circ \tau'_{ij} G$.
\end{itemize}
This version is compatible with the definition 
of 2-representation
from Remark~\ref{2repb}.
\end{remark}

\begin{definition}
A {\em strongly equivariant natural transformation}
between strongly equivariant functors $G, H:\RR \rightarrow
\SS$ is a natural transformation $\pi:G \Rightarrow H$ such
that
$\pi E_i \circ \zeta_{G,i} = \zeta_{H,i} \circ E_i \pi:E_i G
\Rightarrow H E_i$.
Let $\mathfrak{R}\mathrm{ep}(\UU(\g))$ be the strict 2-category of (small)
2-representations, strongly equivariant functors and strongly
equivariant natural transformations. We denote the morphism categories
in this 2-category by $\mathcal{H}om_{\UU(\g)}(\RR, \SS)$.
\end{definition}

\begin{remark}
In the setup of Remark~\ref{yyy},
a strongly equivariant natural transformation is the same as a modification between morphisms of 2-functors in the
sense of \cite[Definition 2.4]{Rou}.
\end{remark}

\subsection{Generalized cyclotomic quotients}
In this subsection, we define some important examples of
2-representations. We need a couple more basic notions to prepare for this.

Fix a weight $\kappa \in P$.
Then there is a 2-representation\footnote{In \cite[$\S$4.3.3]{R2}, Rouquier denotes this
  by $\M(\kappa)$, but that seems confusing notation since it is 
a categorification of the left ideal $\dot U(\g) 1_\kappa$
rather than the Verma module $M(\kappa)$.}
 $\RR(\kappa)$ of $\UU(\g)$ defined as follows:
$\RR(\kappa)_\lambda
:= \mathcal{H}om_{\UU(\g)}(\kappa,\lambda)$;
$E_i$ (resp. $F_i$)
is the functor
defined by horizontally composing 
$1$-morphisms on the left by $E_i$ (resp. $F_i$)
and $2$-morphisms on the left by
$
\mathord{
\begin{tikzpicture}[baseline = -2]
	\draw[->,thick,darkred] (0.08,-.15) to (0.08,.3);
   \node at (0.08,-.25) {$\scriptstyle{i}$};
\end{tikzpicture}
}
{\scriptstyle\lambda}
$
(resp. $\mathord{
\begin{tikzpicture}[baseline = 2]
	\draw[<-,thick,darkred] (0.08,-.15) to (0.08,.3);
   \node at (0.08,.4) {$\scriptstyle{i}$};
\end{tikzpicture}
}
{\scriptstyle\lambda}$);
$x_i, \tau_{ij}, \eta_i$ and $\eps_i$
are the natural transformations 
defined by horizontally composing on the left by
$\mathord{
\begin{tikzpicture}[baseline = -2]
	\draw[->,thick,darkred] (0.08,-.15) to (0.08,.3);
      \node at (0.08,0.05) {$\color{darkred}\bullet$};
   \node at (0.08,-.25) {$\scriptstyle{i}$};
\end{tikzpicture}
}$,
$\mathord{
\begin{tikzpicture}[baseline = -2]
	\draw[->,thick,darkred] (0.18,-.15) to (-0.18,.3);
	\draw[->,thick,darkred] (-0.18,-.15) to (0.18,.3);
   \node at (-0.18,-.25) {$\scriptstyle{i}$};
   \node at (0.18,-.25) {$\scriptstyle{j}$};
\end{tikzpicture}
}
$,
$\mathord{
\begin{tikzpicture}[baseline = 0]
	\draw[<-,thick,darkred] (0.3,0.2) to[out=-90, in=0] (0.1,-0.1);
	\draw[-,thick,darkred] (0.1,-0.1) to[out = 180, in = -90] (-0.1,0.2);
    \node at (-0.1,.3) {$\scriptstyle{i}$};
\end{tikzpicture}
}\,
$
and
$\mathord{
\begin{tikzpicture}[baseline = -2]
	\draw[<-,thick,darkred] (0.3,-0.1) to[out=90, in=0] (0.1,0.2);
	\draw[-,thick,darkred] (0.1,0.2) to[out = 180, in = 90] (-0.1,-0.1);
    \node at (-0.1,-.2) {$\scriptstyle{i}$};
\end{tikzpicture}
}\,$, respectively.

An {\em invariant ideal}\footnote{In Rouquier's language, an invariant
  ideal is the data of a full sub-2-representation.} $\I$ of a 2-representation $\RR$ is a
family of subspaces $\I(b,c) \leq \Hom_{\RR}(b,c)$
for each $b, c \in \ob \RR$
such that
\begin{itemize}
\item
$f \in \Hom_{\RR}(a,b)\text{ and }g \in \I(b,c)
\Rightarrow g \circ f \in \I(a,c)$;
\item
$h \in \Hom_{\RR}(c,d)\text{ and }g \in \I(b,c)
\Rightarrow
h \circ g \in \I(b,d)$;
\item
$g \in \I(b,c)\text{ and }i \in I 
\Rightarrow E_i(g) \in \I(E_i(b), E_i(c))\text{ and }F_i(g) \in \I(F_i(b), F_i(c))$.
\end{itemize}
Given an invariant ideal $\I$, the quotient category $\RR / \I$ is the
category with the same objects as $\RR$ and morphisms
$\Hom_{\RR / \I}(b,c) := \Hom_\RR(b,c) / \I(b,c)$.
It  has a naturally induced structure of 2-representation in its own
right.
%(If $\RR$ is additive, it is natural to require also that $\I$ is
%closed under finite direct sums; then $\RR / \I$ is
%additive too.)

Now we are ready for the main construction.

\begin{example}\label{bad}
Fix weights $\kappa \in P^+$, $\kappa' \in -P^+$.
Let $k_i := \langle h_i, \kappa\rangle$, $k_i' := -\langle h_i,
\kappa'\rangle$, and take
a family of indeterminates $\{z_{i,r}, z_{j,s}'\:|\:i,j \in I, 
1 \leq r \leq k_i, 1 \leq s \leq k_j'\}$.
Let
$\k[z] := \k[z_{i,r}, z_{j,s}'\:|\:i,j \in I, 1 \leq r \leq k_i, 1
\leq k_j']$
be the corresponding polynomial algebra.
Adopting the convention that $z_{i,0}=z_{i,0}' := 1$, 
we define new variables $\delta_{i,s}, \delta_{i,s}' \in \k[z]$ for $s
\geq 0$
from the generating functions
\begin{align}\label{vip}
\sum_{s \geq 0} \delta_{i,s} t^{s}  &:=
c_{\kappa+\kappa';i}
\frac{ \sum_{r=0}^{k'_i} z_{i,r}'t^{r} }{ \sum_{r=0}^{k_i} z_{i,r}t^{r} },
&
\sum_{s \geq 0} \delta'_{i,s} t^{s}  &:=
c_{\kappa+\kappa';i}^{-1}
\frac{ \sum_{r=0}^{k_i} z_{i,r}t ^{r} }{ \sum_{r=0}^{k'_i} z_{i,r}' t^{r} }.
\end{align}
Here, 
we are working in $\k[z][[t]]$ where
$t$ is a formal parameter.
Let $\RR := \RR(\kappa+\kappa') \otimes_{\k} \k[z]$
be the $\k[z]$-linear 2-representation 
obtained from
$\RR(\kappa+\kappa')$ by extending scalars
in the obvious way. Let
$\I$ be the $\k[z]$-linear invariant ideal of $\RR$ generated by
the morphisms
\begin{align}
&\sum_{r=0}^{k_i}
\bigg(\mathord{
\begin{tikzpicture}[baseline = 2]
	\draw[<-,darkred,thick] (0.08,-.3) to (0.08,.4);
      \node at (0.08,0.1) {$\color{darkred}\bullet$};
   \node at (0.1,.5) {$\scriptstyle{i}$};
   \node at (-0.36,.1) {$\color{darkred}\scriptstyle{k_i-r}$};
\end{tikzpicture}
}
{\scriptstyle\kappa+\kappa'}
\bigg) z_{i,r},\label{buh1}\\
& 
\mathord{
\begin{tikzpicture}[baseline = 0]
  \draw[<-,thick,darkred] (0,0.4) to[out=180,in=90] (-.2,0.2);
  \draw[-,thick,darkred] (0.2,0.2) to[out=90,in=0] (0,.4);
 \draw[-,thick,darkred] (-.2,0.2) to[out=-90,in=180] (0,0);
  \draw[-,thick,darkred] (0,0) to[out=0,in=-90] (0.2,0.2);
 \node at (0,-.1) {$\scriptstyle{i}$};
   \node at (-0.65,0.2) {$\scriptstyle{\kappa+\kappa'}$};
   \node at (0.2,0.2) {$\color{darkred}\bullet$};
   \node at (0.6,0.2) {$\color{darkred}\scriptstyle{s+*}$};
\end{tikzpicture}
}
- 1_{\kappa+\kappa'} \delta_{i,s}\label{huh1}
 \end{align}
for all $i \in I$ and $s=1,\dots,k_i'$, using the shorthand (\ref{hfg}).
Taking the quotient, we obtain the $\k[z]$-linear 2-representation
\begin{equation}
\L(\kappa'|\kappa) := \RR / \I.
\end{equation}
%Given any field $\K$ that is a $\k[z]$-algebra,
%we can specialize this to obtain a
%$\K$-linear 2-representation
%$\L(\kappa'|\kappa) \otimes_{\k[z]} \K$, i.e. a 2-representation
%of $\UU(\g) \otimes_{\k}\K$.
Finally, we viewing
$\k$ as a $\k[z]$-algebra so that each $z_{i,r},
z_{i,r}'$ act as zero, we have the {\em minimal specialization}
\begin{equation}\label{fault}
\Lm(\kappa'|\kappa) := 
\L(\kappa'|\kappa) \otimes_{\k[z]} \k.
\end{equation}
\end{example}

\begin{lemma}\label{new}
The ideal $\I$ in Construction~\ref{bad} is generated also by the 
morphisms
\begin{align}
&\sum_{r=0}^{k_i'}
\bigg(
\mathord{
\begin{tikzpicture}[baseline = -2]
	\draw[->,darkred,thick] (0.08,-.3) to (0.08,.4);
      \node at (0.08,0) {$\color{darkred}\bullet$};
   \node at (0.1,-.4) {$\scriptstyle{i}$};
   \node at (-0.36,0) {$\color{darkred}\scriptstyle{k'_i-r}$};
\end{tikzpicture}
}
{\scriptstyle\kappa+\kappa'}
\bigg) z'_{i,r},\label{buh2}\\
&
\mathord{
\begin{tikzpicture}[baseline = 0]
  \draw[-,thick,darkred] (0,0.4) to[out=180,in=90] (-.2,0.2);
  \draw[->,thick,darkred] (0.2,0.2) to[out=90,in=0] (0,.4);
 \draw[-,thick,darkred] (-.2,0.2) to[out=-90,in=180] (0,0);
  \draw[-,thick,darkred] (0,0) to[out=0,in=-90] (0.2,0.2);
 \node at (0,-.1) {$\scriptstyle{i}$};
   \node at (-0.65,0.2) {$\scriptstyle{\kappa+\kappa'}$};
   \node at (0.2,0.2) {$\color{darkred}\bullet$};
   \node at (0.6,0.2) {$\color{darkred}\scriptstyle{s+*}$};
\end{tikzpicture}
}
- 1_{\kappa+\kappa'} \delta_{i,s}',\label{huh2}
\end{align}
for all $i \in I$ and $s=1,\dots,k_i$.
Moreover, it contains (\ref{huh1}) 
and (\ref{huh2}) 
for all $s \geq 0$.
\end{lemma}

\begin{proof}
We first show that
the images of the elements (\ref{huh1}) are zero in
$\L(\kappa'|\kappa)$ for all $s \geq 0$.
For the induction step, we may assume that $s > k_i'$.
Note by the definition from (K7) that 
$$
\mathord{
\begin{tikzpicture}[baseline = 0]
  \draw[<-,thick,darkred] (0,0.4) to[out=180,in=90] (-.2,0.2);
  \draw[-,thick,darkred] (0.2,0.2) to[out=90,in=0] (0,.4);
 \draw[-,thick,darkred] (-.2,0.2) to[out=-90,in=180] (0,0);
  \draw[-,thick,darkred] (0,0) to[out=0,in=-90] (0.2,0.2);
 \node at (0,-.1) {$\scriptstyle{i}$};
   \node at (-0.65,0.2) {$\scriptstyle{\kappa+\kappa'}$};
   \node at (0.2,0.2) {$\color{darkred}\bullet$};
   \node at (0.6,0.2) {$\color{darkred}\scriptstyle{s+*}$};
\end{tikzpicture}
}
=
\mathord{
\begin{tikzpicture}[baseline = 0]
  \draw[<-,thick,darkred] (0,0.4) to[out=180,in=90] (-.2,0.2);
  \draw[-,thick,darkred] (0.2,0.2) to[out=90,in=0] (0,.4);
 \draw[-,thick,darkred] (-.2,0.2) to[out=-90,in=180] (0,0);
  \draw[-,thick,darkred] (0,0) to[out=0,in=-90] (0.2,0.2);
 \node at (0,-.1) {$\scriptstyle{i}$};
   \node at (-0.65,0.2) {$\scriptstyle{\kappa+\kappa'}$};
   \node at (0.2,0.2) {$\color{darkred}\bullet$};
   \node at (1.05,0.2) {$\color{darkred}\scriptstyle{s+k_i - k_i' -1}$};
\end{tikzpicture}
},
$$
so for $s > k_i'$ there are $\geq k_i$ dots on the right hand side
here.
Therefore using the relation (\ref{buh1}), we get that
$$
\sum_{r=0}^{k_i} 
\mathord{
\begin{tikzpicture}[baseline = 0]
  \draw[<-,thick,darkred] (0,0.4) to[out=180,in=90] (-.2,0.2);
  \draw[-,thick,darkred] (0.2,0.2) to[out=90,in=0] (0,.4);
 \draw[-,thick,darkred] (-.2,0.2) to[out=-90,in=180] (0,0);
  \draw[-,thick,darkred] (0,0) to[out=0,in=-90] (0.2,0.2);
 \node at (0,-.1) {$\scriptstyle{i}$};
   \node at (-0.65,0.2) {$\scriptstyle{\kappa+\kappa'}$};
   \node at (0.2,0.2) {$\color{darkred}\bullet$};
   \node at (0.8,0.2) {$\color{darkred}\scriptstyle{s-r+*}$};
\end{tikzpicture}
}
z_{i,r} = 0,
$$
working in the quotient $\L(\kappa'|\kappa)$.
Applying the inductive hypothesis, we deduce that
$$
\mathord{
\begin{tikzpicture}[baseline = 0]
  \draw[<-,thick,darkred] (0,0.4) to[out=180,in=90] (-.2,0.2);
  \draw[-,thick,darkred] (0.2,0.2) to[out=90,in=0] (0,.4);
 \draw[-,thick,darkred] (-.2,0.2) to[out=-90,in=180] (0,0);
  \draw[-,thick,darkred] (0,0) to[out=0,in=-90] (0.2,0.2);
 \node at (0,-.1) {$\scriptstyle{i}$};
   \node at (-0.65,0.2) {$\scriptstyle{\kappa+\kappa'}$};
   \node at (0.2,0.2) {$\color{darkred}\bullet$};
   \node at (0.6,0.2) {$\color{darkred}\scriptstyle{s+*}$};
\end{tikzpicture}
}
+\sum_{r=1}^{k_i} 
1_{\kappa+\kappa'}\delta_{i,s-r} z_{i,r}=0.
$$
It remains to observe that
$\sum_{r=0}^{k_i} \delta_{i,s-r} z_{i,r} = 0 $ already in $\k[z]$ when
$s > k_i'$.
This follows because $\big(\sum_{s \geq 0} \delta_{i,s} t^{s}\big)
\:\big(\sum_{r=0}^{k_i} z_{i,r}t^{r}\big)$ 
is a polynomial of
degree $k_i'$ by the definition (\ref{vip}).

Now let 
$\e(t) := \sum_{r \geq
  0} \e_r t^r$
and $\h(t) := \sum_{r \geq 0} \h_r t^r$, so that $\e(-t)\h(t) = 1$ by
(\ref{igr}); remember also the definitions (\ref{ee})--(\ref{hh}).
Also set 
$\delta_i(t) := \sum_{s \geq 0} \delta_{i,s} t^s$ and
$\delta'_i(t) := \sum_{s \geq 0} \delta'_{i,s} t^s$,
so that $\delta_i'(t) \delta_i(t)  = 1$ by (\ref{vip}).
In the previous paragraph, we have shown that the image of
$\beta_{\kappa+\kappa';i}(\h(t))$
is $c_{\kappa+\kappa';i}^{-1} 1_{\kappa+\kappa'}\delta_i(t) $.
Hence, the image of $\beta_{\kappa+\kappa';i}(\e(-t))$
is $c_{\kappa+\kappa';i} 1_{\kappa+\kappa'} \delta'_i(t)$. 
This shows that (\ref{huh2}) belongs to $\I$ for all $s \geq 0$.

In this paragraph we show that (\ref{buh2}) belongs to $\I$ too.
Working in $\L(\kappa'|\kappa)$ once again, we have by (\ref{buh1})
and (K6) that
$$
0=\sum_{r=0}^{k_i}
\mathord{
\begin{tikzpicture}[baseline = 0]
	\draw[<-,thick,darkred] (0,0.6) to (0,0.3);
	\draw[-,thick,darkred] (0,0.3) to [out=-90,in=180] (.3,-0.2);
	\draw[-,thick,darkred] (0.3,-0.2) to [out=0,in=-90](.5,0);
	\draw[-,thick,darkred] (0.5,0) to [out=90,in=0](.3,0.2);
	\draw[-,thick,darkred] (0.3,.2) to [out=180,in=90](0,-0.3);
	\draw[-,thick,darkred] (0,-0.3) to (0,-0.6);
   \node at (0.9,0) {$\color{darkred}\scriptstyle{k_i-r}$};
      \node at (0.5,0) {$\color{darkred}\bullet$};
   \node at (0,-.7) {$\scriptstyle{i}$};
   \node at (.7,-0.4) {$\scriptstyle{\kappa+\kappa'}$};
\end{tikzpicture}
}
z_{i,r}= -
\sum_{r= 0}^{k_i} \sum_{s=0}^{k_i'-r}
\mathord{
\begin{tikzpicture}[baseline = 3]
	\draw[->,thick,darkred] (0.08,-.25) to (0.08,.55);
     \node at (0.08,-.35) {$\scriptstyle{i}$};
   \node at (-0.3,0.1) {$\color{darkred}\scriptstyle{s}$};
      \node at (0.08,0.1) {$\color{darkred}\bullet$};
\end{tikzpicture}
}
\mathord{\begin{tikzpicture}[baseline = 0]
  \draw[<-,thick,darkred] (0,0.2) to[out=180,in=90] (-.2,0);
  \draw[-,thick,darkred] (0.2,0) to[out=90,in=0] (0,.2);
 \draw[-,thick,darkred] (-.2,0) to[out=-90,in=180] (0,-0.2);
  \draw[-,thick,darkred] (0,-0.2) to[out=0,in=-90] (0.2,0);
 \node at (0,-.28) {$\scriptstyle{i}$};
   \node at (-1,0) {$\color{darkred}\scriptstyle{k_i'-r-s+*}$};
      \node at (-.2,0) {$\color{darkred}\bullet$};
\end{tikzpicture}
}\:
{\scriptstyle\kappa+\kappa'}
\:\:\:z_{i,r}.
$$
Changing the summation using also (\ref{huh1}), we have shown that
$$
\sum_{r=0}^{k_i'} 
\mathord{
\begin{tikzpicture}[baseline = 3]
	\draw[->,thick,darkred] (0.08,-.25) to (0.08,.55);
     \node at (0.08,-.35) {$\scriptstyle{i}$};
   \node at (-0.4,0.1) {$\color{darkred}\scriptstyle{k_i'-r}$};
      \node at (0.08,0.1) {$\color{darkred}\bullet$};
\end{tikzpicture}
}
{\scriptstyle\kappa+\kappa'}
\Bigg(
\sum_{s= 0}^{r} 
\delta_{i,r-s}
\:z_{i,s}\Bigg) = 0.
$$
It remains to apply (\ref{vip}) to simplify this to (\ref{buh2}).

Conversely, one checks by similar arguments that the $\k[z]$-linear invariant ideal $\I'$ generated by the
elements (\ref{buh2})--(\ref{huh2})
contains (\ref{buh1})--(\ref{huh1}).
\end{proof}

\begin{lemma}\label{hom}
There is a unique $\g$-module homomorphism
\begin{equation*}
L(\kappa'|\kappa) \rightarrow
\mathbb{C}\otimes_{\Z}
K_0(\dotLm(\kappa'|\kappa)),
\qquad
\bar 1_{\kappa+\kappa'}
\mapsto [1_{\kappa+\kappa'}].
\end{equation*}
\end{lemma}

\begin{proof}
We need to show that the homomorphism
$\dot{U}(\g) 1_{\kappa+\kappa'} \rightarrow
\mathbb{C}\otimes_{\Z}
K_0(\dotLm(\kappa'|\kappa))$ sending $1_{\kappa+\kappa'} \mapsto
[1_{\kappa+\kappa'}]$
factors through the quotient $L(\kappa'|\kappa)$ from
(\ref{mid}).
This follows because $E_i^{(1+k_i')} 1_{\kappa+\kappa'} = 0 =
F_i^{(1+k_i)} 1_{\kappa+\kappa'}$
in $\dotLm(\kappa'|\kappa)$.
The first equality here follows from Lemma~\ref{proud} and the
defining relation (\ref{buh2});
the second one follows similarly using (\ref{buh1}) and a rotated
version of Lemma~\ref{proud}.
\end{proof}

\begin{lemma}\label{nilp}
The 2-representation
$\Lm(\kappa'|\kappa)$ is 
{\em nilpotent} in the sense that $(x_i)_u
\in \End_{\Lm(\kappa'|\kappa)}(E_i u)$
and $(x_i')_u
\in \End_{\Lm(\kappa'|\kappa)}(F_i u)$ are
nilpotent for all $i \in I$ and $u \in \ob \Lm(\kappa'|\kappa)$.
\end{lemma}

\begin{proof}
We show by induction on $r$ that
$(x_i)_u$ is nilpotent for
any object $u$ that is a monomial obtained by applying $r$ of the
generating $E$'s and $F$'s to $1_{\kappa+\kappa'}$;
similar arguments give the nilpotency of $(x_i')_u$ too.
The base case $r=0$ follows from (\ref{buh1}) and (\ref{buh2}), since
they show that
$
\mathord{
\begin{tikzpicture}[baseline = 2]
	\draw[<-,darkred,thick] (0.08,-.15) to (0.08,.3);
      \node at (0.08,0.15) {\color{darkred}$\bullet$};
   \node at (0.1,.4) {$\scriptstyle{i}$};
   \node at (-0.2,0.1) {\color{darkred}$\scriptstyle{k_i}$};
\end{tikzpicture}
}
{\scriptstyle\kappa+\kappa'}
=\mathord{
\begin{tikzpicture}[baseline = -2]
	\draw[->,darkred,thick] (0.08,-.15) to (0.08,.3);
      \node at (0.08,0.05) {\color{darkred}$\bullet$};
   \node at (0.1,-.25) {$\scriptstyle{i}$};
   \node at (-0.2,0) {\color{darkred}$\scriptstyle{k'_i}$};
\end{tikzpicture}
}
{\scriptstyle\kappa+\kappa'}
=0$
in $\Lm(\kappa'|\kappa)$.
For the induction step, we consider $(x_i)_u$ for a monomial $u$ of length $(r+1)$.
There are three cases:

\vspace{2mm}

\noindent
{\em Case one:  $u = E_i 1_\lambda v$.}
Introduce the {\em intertwiner}
$
\mathord{
\begin{tikzpicture}[baseline = -2]
	\draw[-,thick,darkred] (0.1,.05) to (-0.1,0.05);
	\draw[-,thick,darkred] (0.18,-.15) to (0.1,0.05);
	\draw[-,thick,darkred] (-0.18,-.15) to (-0.1,0.05);
	\draw[->,thick,darkred] (.1,.05) to (0.18,.3);
	\draw[->,thick,darkred] (-0.1,.05) to (-0.18,.3);
   \node at (-0.18,-.25) {$\scriptstyle{i}$};
   \node at (0.18,-.25) {$\scriptstyle{i}$};
\end{tikzpicture}
}
{\scriptstyle\lambda}
:=
\mathord{
\begin{tikzpicture}[baseline = -2]
	\draw[->,thick,darkred] (0.18,-.15) to (-0.18,.3);
	\draw[->,thick,darkred] (-0.18,-.15) to (0.18,.3);
   \node at (-0.18,-.25) {$\scriptstyle{i}$};
   \node at (0.18,-.25) {$\scriptstyle{i}$};
   \node at (0.084,.16) {$\color{darkred}\bullet$};
\end{tikzpicture}
}
{\scriptstyle\lambda}
-
\mathord{
\begin{tikzpicture}[baseline = -2]
	\draw[->,thick,darkred] (0.18,-.15) to (-0.18,.3);
	\draw[->,thick,darkred] (-0.18,-.15) to (0.18,.3);
   \node at (-0.18,-.25) {$\scriptstyle{i}$};
   \node at (0.14,-.25) {$\scriptstyle{i}$};
   \node at (0.09,-.05) {$\color{darkred}\bullet$};
\end{tikzpicture}
}
{\scriptstyle\lambda}.$
Using the relations, one checks 
that
$
\mathord{
\begin{tikzpicture}[baseline = -7]
	\draw[-,thick,darkred] (0.1,.05) to (-0.1,0.05);
	\draw[-,thick,darkred] (0.18,-.15) to [out=90,in=-70](0.1,0.05);
	\draw[-,thick,darkred] (-0.18,-.15) to [out=90,in=-110](-0.1,0.05);
	\draw[->,thick,darkred] (.1,.05) to (0.18,.3);
	\draw[->,thick,darkred] (-0.1,.05) to (-0.18,.3);
	\draw[-,thick,darkred] (0.1,.-.38) to (-0.1,-.38);
	\draw[-,thick,darkred] (0.18,-.6) to (0.1,-.38);
	\draw[-,thick,darkred] (-0.18,-.6) to (-0.1,-.38);
	\draw[-,thick,darkred] (.1,-.38) to [in=-90,out=70](0.18,-.15);
	\draw[-,thick,darkred] (-0.1,-.38) to [in=-90,out=110](-0.18,-.15);
   \node at (-0.18,-.7) {$\scriptstyle{i}$};
   \node at (0.18,-.7) {$\scriptstyle{i}$};
\end{tikzpicture}
}
{\scriptstyle\lambda}
=
\mathord{
\begin{tikzpicture}[baseline = -2]
	\draw[->,thick,darkred] (0.13,-.15) to (0.13,.3);
	\draw[->,thick,darkred] (-0.13,-.15) to (-0.13,.3);
   \node at (-0.13,-.25) {$\scriptstyle{i}$};
   \node at (0.13,-.25) {$\scriptstyle{i}$};
\end{tikzpicture}
}
{\scriptstyle\lambda}$ and
$\mathord{
\begin{tikzpicture}[baseline = -7]
	\draw[-,thick,darkred] (0.1,.05) to (-0.1,0.05);
	\draw[-,thick,darkred] (0.18,-.15) to [out=90,in=-70](0.1,0.05);
	\draw[-,thick,darkred] (-0.18,-.15) to [out=90,in=-110](-0.1,0.05);
	\draw[->,thick,darkred] (.1,.05) to (0.18,.3);
	\draw[->,thick,darkred] (-0.1,.05) to (-0.18,.3);
	\draw[-,thick,darkred] (0.1,.-.38) to (-0.1,-.38);
	\draw[-,thick,darkred] (0.18,-.6) to (0.1,-.38);
	\draw[-,thick,darkred] (-0.18,-.6) to (-0.1,-.38);
	\draw[-,thick,darkred] (.1,-.38) to [in=-90,out=70](0.18,-.15);
	\draw[-,thick,darkred] (-0.1,-.38) to [in=-90,out=110](-0.18,-.15);
   \node at (-0.18,-.7) {$\scriptstyle{i}$};
   \node at (0.18,-.7) {$\scriptstyle{i}$};
   \node at (0.18,-.15) {$\color{darkred}\bullet$};
\end{tikzpicture}
}
{\scriptstyle\lambda}
=
\mathord{
\begin{tikzpicture}[baseline = -2]
	\draw[->,thick,darkred] (0.13,-.15) to (0.13,.3);
	\draw[->,thick,darkred] (-0.13,-.15) to (-0.13,.3);
   \node at (-0.13,-.25) {$\scriptstyle{i}$};
   \node at (0.13,-.25) {$\scriptstyle{i}$};
   \node at (-0.12,.03) {$\color{darkred}\bullet$};
\end{tikzpicture}
}
{\scriptstyle\lambda}\,$.
Hence,
$\mathord{
\begin{tikzpicture}[baseline = -7]
	\draw[-,thick,darkred] (0.1,.05) to (-0.1,0.05);
	\draw[-,thick,darkred] (0.18,-.15) to [out=90,in=-70](0.1,0.05);
	\draw[-,thick,darkred] (-0.18,-.15) to [out=90,in=-110](-0.1,0.05);
	\draw[->,thick,darkred] (.1,.05) to (0.18,.3);
	\draw[->,thick,darkred] (-0.1,.05) to (-0.18,.3);
	\draw[-,thick,darkred] (0.1,.-.38) to (-0.1,-.38);
	\draw[-,thick,darkred] (0.18,-.6) to (0.1,-.38);
	\draw[-,thick,darkred] (-0.18,-.6) to (-0.1,-.38);
	\draw[-,thick,darkred] (.1,-.38) to [in=-90,out=70](0.18,-.15);
	\draw[-,thick,darkred] (-0.1,-.38) to [in=-90,out=110](-0.18,-.15);
   \node at (-0.18,-.7) {$\scriptstyle{i}$};
   \node at (0.18,-.7) {$\scriptstyle{i}$};
   \node at (0.18,-.15) {$\color{darkred}\bullet$};
   \node at (0.38,-0.15) {\color{darkred}$\scriptstyle{n}$};
\end{tikzpicture}
}
{\scriptstyle\lambda}
=
\mathord{
\begin{tikzpicture}[baseline = -2]
	\draw[->,thick,darkred] (0.13,-.15) to (0.13,.3);
	\draw[->,thick,darkred] (-0.13,-.15) to (-0.13,.3);
   \node at (-0.13,-.25) {$\scriptstyle{i}$};
   \node at (0.13,-.25) {$\scriptstyle{i}$};
   \node at (-0.12,.03) {$\color{darkred}\bullet$};
   \node at (-0.32,0.03) {\color{darkred}$\scriptstyle{n}$};
\end{tikzpicture}
}
{\scriptstyle\lambda}\,$.
As $(x_i^n)_v \in \End_{\Lm(\kappa'|\kappa)}(E_i v)$ is zero for some $n$ by
induction, we deduce that $(x_i^n)_u
\in \End_{\Lm(\kappa'|\kappa)}(E_i u)$ is zero too.

\vspace{2mm}

\noindent
{\em Case two: $u = E_j 1_\lambda v$ for $j
\neq i$.}
By induction, we have that $(x_i^n)_v = (x_j^n)_v = 0$ for some $n$.
Let $m := (d_{ij}+1) n$.
Then we can use the defining relations to rewrite
$\mathord{
\begin{tikzpicture}[baseline = -2]
	\draw[->,thick,darkred] (0.13,-.15) to (0.13,.3);
	\draw[->,thick,darkred] (-0.13,-.15) to (-0.13,.3);
   \node at (-0.13,-.25) {$\scriptstyle{i}$};
   \node at (0.13,-.25) {$\scriptstyle{j}$};
   \node at (-0.12,.03) {$\color{darkred}\bullet$};
   \node at (-0.34,0.03) {\color{darkred}$\scriptstyle{m}$};
   \node at (0.4,0) {$\scriptstyle{\lambda}$};
\end{tikzpicture}
}$
as a linear combination of
$\mathord{
\begin{tikzpicture}[baseline = -3]
	\draw[->,thick,darkred] (0.14,0) to[out=90,in=-90] (-0.18,.4);
	\draw[->,thick,darkred] (-0.14,0) to[out=90,in=-90] (0.18,.4);
	\draw[-,thick,darkred] (0.18,-.4) to[out=90,in=-90] (-0.14,0);
	\draw[-,thick,darkred] (-0.18,-.4) to[out=90,in=-90] (0.14,0);
  \node at (-0.18,-.5) {$\scriptstyle{i}$};
  \node at (0.18,-.5) {$\scriptstyle{j}$};
   \node at (0.62,0) {$\scriptstyle{\lambda}$};
   \node at (0.18,0) {$\color{darkred}\bullet$};
   \node at (0.38,0) {\color{darkred}$\scriptstyle{m}$};
\end{tikzpicture}
}$
and terms of the form 
$\mathord{
\begin{tikzpicture}[baseline = -2]
	\draw[->,thick,darkred] (0.13,-.15) to (0.13,.3);
	\draw[->,thick,darkred] (-0.13,-.15) to (-0.13,.3);
   \node at (-0.13,-.25) {$\scriptstyle{i}$};
   \node at (0.13,-.25) {$\scriptstyle{j}$};
   \node at (-0.12,.03) {$\color{darkred}\bullet$};
   \node at (-0.32,0.03) {\color{darkred}$\scriptstyle{p}$};
   \node at (0.12,.03) {$\color{darkred}\bullet$};
   \node at (0.32,0.03) {\color{darkred}$\scriptstyle{q}$};
   \node at (0.6,0) {$\scriptstyle{\lambda}$};
\end{tikzpicture}
}$
with $p \geq n d_{ij}$ and $q \geq 1$.
Repeating the calculation $(n-1)$ more times, we get that
$\mathord{
\begin{tikzpicture}[baseline = -2]
	\draw[->,thick,darkred] (0.13,-.15) to (0.13,.3);
	\draw[->,thick,darkred] (-0.13,-.15) to (-0.13,.3);
   \node at (-0.13,-.25) {$\scriptstyle{i}$};
   \node at (0.13,-.25) {$\scriptstyle{j}$};
   \node at (-0.12,.03) {$\color{darkred}\bullet$};
   \node at (-0.34,0.03) {\color{darkred}$\scriptstyle{m}$};
   \node at (0.4,0) {$\scriptstyle{\lambda}$};
\end{tikzpicture}
}$
is a linear combination of
$\mathord{
\begin{tikzpicture}[baseline = -3]
	\draw[->,thick,darkred] (0.14,0) to[out=90,in=-90] (-0.18,.4);
	\draw[->,thick,darkred] (-0.14,0) to[out=90,in=-90] (0.18,.4);
	\draw[-,thick,darkred] (0.18,-.4) to[out=90,in=-90] (-0.14,0);
	\draw[-,thick,darkred] (-0.18,-.4) to[out=90,in=-90] (0.14,0);
  \node at (-0.18,-.5) {$\scriptstyle{i}$};
  \node at (0.18,-.5) {$\scriptstyle{j}$};
   \node at (0.62,0) {$\scriptstyle{\lambda}$};
   \node at (0.14,0) {$\color{darkred}\bullet$};
   \node at (0.34,0) {\color{darkred}$\scriptstyle{r}$};
   \node at (0.13,-.28) {$\color{darkred}\bullet$};
   \node at (0.33,-.28) {\color{darkred}$\scriptstyle{s}$};
\end{tikzpicture}
}$
and terms of the form 
$\mathord{
\begin{tikzpicture}[baseline = -2]
	\draw[->,thick,darkred] (0.13,-.15) to (0.13,.3);
	\draw[->,thick,darkred] (-0.13,-.15) to (-0.13,.3);
   \node at (-0.13,-.25) {$\scriptstyle{i}$};
   \node at (0.13,-.25) {$\scriptstyle{j}$};
   \node at (-0.12,.03) {$\color{darkred}\bullet$};
   \node at (-0.32,0.03) {\color{darkred}$\scriptstyle{p}$};
   \node at (0.12,.03) {$\color{darkred}\bullet$};
   \node at (0.32,0.03) {\color{darkred}$\scriptstyle{q}$};
   \node at (0.6,0) {$\scriptstyle{\lambda}$};
\end{tikzpicture}
}$
with $r \geq n, s \geq 0, p \geq d_{ij}$ and $q \geq n$.
It remains to act on $v$. All the terms vanish
and we have shown that $(x_i^m)_u = 0$.

\vspace{2mm}

\noindent
{\em Case three: $u = F_j 1_\lambda v$.}
Use the first alternating crossing relation from (K8) to move dots
to the right in a similar way.
\end{proof}

\begin{corollary}\label{fdaty}
$\Lm(\kappa'|\kappa)$ is 
a finite-dimensional category.
\end{corollary}

\begin{proof}
We take $u, v \in \ob 
\Lm(\kappa'|\kappa)_\lambda$, i.e. 1-morphisms
$\kappa+\kappa' \rightarrow \lambda$ in $\UU(\g)$ for some
$\lambda \in P$, and 
consider the explicit spanning set for
$\Hom_{\UU(\g)}(u,v)$ constructed as in
\cite[Proposition 3.11]{KL3}; we arrange this so that all the dotted
bubbles appear at the right hand side of the diagrams.
We need to show that the image of this set 
spans a finite-dimensional vector space
when we pass to the quotient $\Lm(\kappa'|\kappa)$.
This is clear because any diagram with a bubble vanishes in the
quotient by the relations (\ref{huh1}) and (\ref{huh2}), and any diagram with too many dots
on any given strand vanishes by the nilpotency established in
Lemma~\ref{nilp}.
\end{proof}

\begin{remark}\label{fd}
\rm
The category $\L(\kappa'|\kappa)$ is not finite-dimensional. 
However, 
each of its morphism spaces 
is {\em finitely generated} as a $\k[z]$-module.
This follows 
by a similar but more delicate inductive argument compared to the proof of Corollary~\ref{fdaty}.
%Hence any specialization of the form $\L(\kappa'|\kappa)
%\otimes_{\k[z]} \K$ is finite-dimensional over $\K$
%(for any field $\K$ that is a $\k[z]$-algebra).
\end{remark}

Taking $\kappa' = 0$ in Construction~\ref{bad}, we obtain the $\k[z]$-linear 2-representation
\begin{equation}
\L(\kappa) := \L(0|\kappa).
\end{equation}
Note in this situation
that $\k[z] = \k[z_{i,r}\:|\:i\in I, r=1,\dots,k_i]$ and
$z_{i,r} = c_{\kappa;i} \delta_{i,r}'$
for $r=1,\dots,k_i$.
Since each $k_i' = 0$, Lemma~\ref{new} shows that $\L(\kappa)$ is the quotient of
$\RR := \RR(\kappa)\otimes_\k \k[z]$ by the $\k[z]$-linear
invariant ideal generated by the morphisms
\begin{align}\label{poll1}
&\mathord{
\begin{tikzpicture}[baseline = -2]
	\draw[->,darkred,thick] (0.08,-.3) to (0.08,.4);
   \node at (0.1,-.4) {$\scriptstyle{i}$};
\end{tikzpicture}
}\!
{\scriptstyle\kappa},\\
\label{poll2}
1_{\kappa} z_{i,r} 
&- 
c_{\kappa;i} 
\mathord{
\begin{tikzpicture}[baseline = 0]
  \draw[-,thick,darkred] (0,0.4) to[out=180,in=90] (-.2,0.2);
  \draw[->,thick,darkred] (0.2,0.2) to[out=90,in=0] (0,.4);
 \draw[-,thick,darkred] (-.2,0.2) to[out=-90,in=180] (0,0);
  \draw[-,thick,darkred] (0,0) to[out=0,in=-90] (0.2,0.2);
 \node at (0,-.1) {$\scriptstyle{i}$};
   \node at (-0.4,0.2) {$\scriptstyle{\kappa}$};
   \node at (0.2,0.2) {$\color{darkred}\bullet$};
   \node at (0.6,0.2) {$\color{darkred}\scriptstyle{r+*}$};
\end{tikzpicture}
},
\end{align}
for $i \in I$ and $r = 1,\dots,k_i$.
In view of (\ref{poll2}), there is no need to extend scalars
to $\k[z]$ after all: we
could equivalently define 
$\L(\kappa)$ to be the quotient of $\RR(\kappa)$ by the invariant ideal
generated by the morphisms
(\ref{poll1})
for all $i \in I$, viewing it as a $\k[z]$-linear 2-representation
so that each $z_{i,r}$ acts 
by horizontally composing on the right with
$
%\beta_{\kappa;i}((-1)^r \e_r)  =
c_{\kappa;i} 
\mathord{
\begin{tikzpicture}[baseline = 2]
  \draw[-,thick,darkred] (0,0.4) to[out=180,in=90] (-.2,0.2);
  \draw[->,thick,darkred] (0.2,0.2) to[out=90,in=0] (0,.4);
 \draw[-,thick,darkred] (-.2,0.2) to[out=-90,in=180] (0,0);
  \draw[-,thick,darkred] (0,0) to[out=0,in=-90] (0.2,0.2);
 \node at (0,-.1) {$\scriptstyle{i}$};
   \node at (-0.4,0.2) {$\scriptstyle{\kappa}$};
   \node at (0.2,0.2) {$\color{darkred}\bullet$};
   \node at (0.6,0.2) {$\color{darkred}\scriptstyle{r+*}$};
\end{tikzpicture}
}$.

The discussion in the previous paragraph shows that $\L(\kappa)$ is Rouquier's {\em universal categorification} of
$L(\kappa)$ from 
\cite[$\S$4.3.3]{R2}.
These 2-representations play a fundamental role
in his general structure theory for upper integrable
2-representations.
To start with, they formally satisfy the following universal property:
for any 2-representation $\V$, let 
$$
\V_\kappa^{\hw} := \{V \in \ob
\V_\kappa\:|\:E_i V = 0\text{ for all }i \in I\}
$$
which is a full subcategory of $\V_\kappa$;
then there is an equivalence of categories
\begin{equation}\label{up}
\operatorname{ev}_{1_\kappa}:\mathcal{H}om_{\UU(\g)}(\L(\kappa), \V)
\rightarrow
\V^\hw_\kappa,
\qquad
G \mapsto G 1_\kappa,
\quad
\pi \mapsto \pi_{1_\kappa}.
\quad
\end{equation}
This is a key ingredient in \cite[Theorem 4.22]{R2}, which shows that any
upper integrable, additive, idempotent-complete
2-representation has a finite filtration whose
sections are specializations of the $\dot\L(\kappa)$'s
ordered in some way refining the dominance order (most dominant at the bottom). This result is a categorical analog of the
filtration of a based module constructed by Lusztig in
\cite[Ch. 27]{Lubook}.

Using a fundamental theorem of Kang and Kashiwara \cite[Theorem 5.2]{KK},
Rouquier has also given an equivalent realization of $\dot\L(\kappa)$
as follows.
Recalling Definition~\ref{pd}, introduce the (negative) {\em cyclotomic quiver Hecke category}
$\H'(\kappa)$, namely, the quotient of the $\k[z]$-linear monoidal category
$\H' \otimes_\k \k[z]$ by the $\k[z]$-linear left tensor ideal generated by the morphisms
\begin{equation}
\sum_{r=0}^{k_i}
\bigg(\mathord{
\begin{tikzpicture}[baseline = 2]
	\draw[<-,darkred,thick] (0.08,-.3) to (0.08,.4);
      \node at (0.08,0.1) {$\color{darkred}\bullet$};
   \node at (0.1,.5) {$\scriptstyle{i}$};
   \node at (-0.36,.1) {$\color{darkred}\scriptstyle{k_i-r}$};
\end{tikzpicture}
}
\bigg) z_{i,r}\qquad(i \in I).
\end{equation}
The endomorphism algebras 
\begin{equation}
\widehat{H}'_n(\kappa):= \bigoplus_{\bi, \bj \in I^{\otimes
    n}}\Hom_{\H'(\kappa)}(\bi, \bj)
\end{equation}
are Rouquier's
{\em deformed cyclotomic quiver Hecke algebras};
in particular, $\widehat{H}_0'(\kappa)$, the endomorphism algebra of
the unit object $\varnothing$ of $\H'(\kappa)$,
is $\k[z]$.
The theorem of Kang and Kashiwara shows that the additive Karoubi
envelope 
$\dot\H'(\kappa)$ can be endowed with the structure of a
$\k[z]$-linear 2-representation
with $E_i$ and $F_i$ arising from certain restriction and
induction functors\footnote{The equivalent formulation of the definition of
2-representation from Remark~\ref{2repb} is 
convenient here since induction is obviously left adjoint to
restriction; Lemma~\ref{nec} implies that it is right adjoint too,
but this is far from clear from the outset.}.
Applying the universal property from (\ref{up}), 
we get 
a canonical
strongly equivariant functor
$G:\L(\kappa) \rightarrow \dot\H'(\kappa)$
such that $\ev_{1_\kappa} (G) = \varnothing$.

\begin{theorem}[{\cite[Theorem 4.25]{R2}}]\label{rrr}
The functor $G$ is $\k[z]$-linear
and it induces a strongly equivariant equivalence
$G:\dot\L(\kappa)\rightarrow\dot\H'(\kappa)$.
Hence,
$\End_{\L(\kappa)}(1_\kappa)\cong\k[z]$.
%Moreover
%$\mathbb{C} \otimes_{\Z} K_0(\dot\L(\kappa)) \cong L(\kappa)$
%as $\g$-modules.
\end{theorem}

It follows immediately that 
the {\em minimal categorification}
\begin{equation}
\Lm(\kappa) := \Lm(0|\kappa)
\end{equation}
is Morita equivalent to the quotient $\Hm'(\kappa)$ of $\H'$
by the left tensor ideal generated by
$\mathord{
\begin{tikzpicture}[baseline = 2]
	\draw[<-,darkred,thick] (0.08,-.1) to (0.08,.3);
      \node at (0.08,0.15) {$\color{darkred}\bullet$};
   \node at (0.1,.4) {$\scriptstyle{i}$};
   \node at (-0.2,.18) {$\color{darkred}\scriptstyle{k_i}$};
\end{tikzpicture}
}\:(i \in I)$.
The endomorphism algebras 
\begin{equation}
H_n'(\kappa):=
\bigoplus_{\bi, \bj \in I^{\otimes
    n}}\Hom_{\Hm'(\kappa)}(\bi, \bj)
\end{equation}
are the
{\em cyclotomic quiver Hecke algebras} 
introduced by Khovanov and Lauda in \cite{KL1}.
They are finite-dimensional algebras, so
the blocks of $\Lm(\kappa)$ are finite-dimensional
algebras too; in particular, $\Lm(\kappa)$ is Artinian.
Moreover, $\End_{\Lm(\kappa)}(1_\kappa) \cong H_0'(\kappa) \cong \k$.
This shows that $1_\kappa$ is non-zero in
$\Lm(\kappa)$, which is the crucial
point needed in order to deduce the following theorem,
which was established already in \cite[Theorem 6.2]{KK}.
Note also that Webster has given a different proof of all of these results in \cite[$\S$3]{Web}.

\begin{theorem}[{\cite{KK, Web}}]\label{t}
For any $\kappa \in P^+$, the homomorphism 
$$
L(\kappa) \rightarrow
\mathbb{C}\otimes_{\Z}
K_0(\dotLm(\kappa))
$$
from Lemma~\ref{hom} is an isomorphism.
\end{theorem}

\begin{proof}
We just saw that the homomorphism is non-zero. Hence, it is injective.
It is surjective too by a standard argument recalled in Corollary~\ref{lastone} below. 
\end{proof}

\begin{corollary}\label{gr}
If $\g$ is of finite type, then the 
homomorphism $\gamma:\dot U(\g)_\Z \rightarrow K_0(\dot\UU(\g))$
from (\ref{gamma}) is an isomorphism; similarly, so is $\gamma_q$
from (\ref{gammaq}).
\end{corollary}

\begin{proof}
It remains to show that $\gamma$ is injective.
Take $u \in \dot U(\g)_\Z$ with $\gamma(u) = 0$.
By its definition, the isomorphism
$L(\kappa) \stackrel{\sim}{\rightarrow}
\mathbb{C}\otimes_{\Z}
K_0(\dotLm(\kappa))$
from Theorem~\ref{t}
intertwines the action of $u$ on the left hand space with the action
of $\gamma(u)$ on the right hand space.
Hence, $u$ annihilates $L(\kappa)$. This is true for each $\kappa \in
P^+$, so we get that $u=0$ by the second statement of Lemma~\ref{faith}.
\end{proof}

\begin{remark}
The results just explained obviously have lowest weight analogs too.
Taking $\kappa = 0$ 
in Construction~\ref{bad},
the universal and minimal
 categorifications of $L'(\kappa')$ are
$\Lp(\kappa') := \L(\kappa'|0)$
and $\Lmp(\kappa') := \Lm(\kappa'|0)$, respectively.
They are Morita equivalent to
analogously defined cyclotomic
quotients $\H(\kappa')$ and $\Hm(\kappa')$ 
of the quiver Hecke category $\H$.
%, e.g. the definition of $\H(\kappa')$ mirrors (\ref{buh2}).
Moreover $\End_{\L'(\kappa')}(1_{\kappa'}) \cong \k[z]$, the
finite-dimensional category
$\Lmp(\kappa')$ is Artinian, and
$\mathbb{C}\otimes_{\Z} K_0(\dotLmp(\kappa')) \cong L(\kappa')$.
\end{remark}

\begin{remark}\label{crazy}\rm
The general definition of the 2-representations $\L(\kappa'|\kappa)$ for
$\kappa,\kappa' \neq 0$ given in Construction~\ref{bad} is new, but
their minimal specializations 
$\Lm(\kappa'|\kappa)$ 
appear already\footnote{Note that Webster's diagrams are the mirror
images of ours in a vertical axis.} 
in \cite[Proposition 5.6]{Wcan}. 
In \cite[Theorem 5.13]{Wcan}, Webster asserts that
the homomorphism in Lemma~\ref{hom} is an isomorphism for all $\kappa,\kappa'$,
so that $\Lm(\kappa'|\kappa)$ is a categorification of
$L(\kappa'|\kappa) \cong L'(\kappa')\otimes L(\kappa)$ in general.
Like in the proof of Corollary~\ref{gr} (using the first statement of
Lemma~\ref{faith} instead of the second), such a result implies that 
$\dot U(\g)_\Z \cong K_0(\dot\UU(\g))$ 
for arbitrary $\g$.
However, Webster's proof of this is intertwined with his new
approach in \cite{erratum} to verifying the
Nondegeneracy Condition; cf. Remark~\ref{dr}.
Actually, \cite{erratum} is based on some even
more general deformations, which should be closely related to our
$\L(\kappa'|\kappa)$ when there is just one lowest and one highest
weight tensor factor.
\end{remark}

\begin{remark}
\rm
The finite-dimensional category
$\Lm(\kappa'|\kappa)$ is not Artinian in general (outside of finite type). We conjecture that
it is always Noetherian.
\end{remark}

\iffalse
RIGHT NOW I DON'T BELIEVE WEBSTER'S DEFORMATION IS CORRECT ACTUALLY
\begin{remark}\rm
In \cite{erratum}, Webster introduces some more general deformations of the tensor
product algebras of \cite{Web}.
In the case that there is one red string labelled $\kappa$ (with
Webster's corresponding parameter denoted $z$) and one
blue string labelled $\kappa'$ (with parameter $z'$), 
we expect that Webster's deformed category is equivalent to our
$\L(\kappa'|\kappa)$
with $z_{i,r}$ and $z_{i,r'}$ specialized according to 
$$
\sum_{r=0}^{k_i} z_{i,r}
t^r \mapsto (1-t z)^{k_i},
\qquad
\sum_{r=0}^{k_i} z'_{i,r}
t^r \mapsto (1-t z')^{k_i'}.
$$
This should follow from similar arguments to the proof of
\cite[Proposition 5.6]{Wcan}.
\end{remark}
\fi

\subsection{Categorical actions}
Henceforth, $\C$ denotes a locally Schurian category in the sense of Definition~\ref{kitty}.
We fix a set of representatives $\{L(b)\:|\:b \in \B\}$ for the
isomorphism classes of irreducible objects, and let $P(b)$ be a
projective cover of $L(b)$.
Let $K_0(\pC)$ denote the split
Grothendieck group of the additive category $\pC$ ($=$ finitely
generated projectives in $\C$).
The classes
$\{[P(b)]\:|\:b \in \B\}$
give a distinguished basis for $\mathbb{C}
\otimes_{\Z} K_0(\pC)$.

\begin{definition}\label{catact}
A {\em categorical action} of $\g$ on $\C$ is the following additional
data:
\begin{itemize}
\item[(A1)] a partition $\B = \bigsqcup_{\lambda \in P}
  \B_\lambda$
inducing a decomposition $\C = \prod_{\lambda \in P} \C_\lambda$ as in (L10) from $\S\ref{fc}$;
\item[(A2)]
sweet endofunctors $E_i$ of $\C$ for each $i \in I$
such that 
$E_i|_{\C_\lambda}:\C_\lambda \rightarrow
\C_{\lambda+\alpha_i}$ (recall Definition~\ref{swe});
\item[(A3)] a strict monoidal functor $\Phi:\H \rightarrow
  \mathcal{E}\!nd(\C)$
  with $\Phi(i) = E_i$ for each $i$, where $\H$ is the quiver Hecke
category from Definition~\ref{pd}.
\end{itemize}
Let
$x_i := \Phi\big(\mathord{
\begin{tikzpicture}[baseline = -2]
	\draw[->,thick,darkred] (0.08,-.15) to (0.08,.3);
      \node at (0.08,0.05) {$\color{darkred}\bullet$};
   \node at (0.08,-.25) {$\scriptstyle{i}$};
\end{tikzpicture}
}\big)$ and $\tau_{ij} := \Phi\big(\mathord{
\begin{tikzpicture}[baseline = -2]
	\draw[->,thick,darkred] (0.18,-.15) to (-0.18,.3);
	\draw[->,thick,darkred] (-0.18,-.15) to (0.18,.3);
   \node at (-0.18,-.25) {$\scriptstyle{i}$};
   \node at (0.18,-.25) {$\scriptstyle{j}$};
\end{tikzpicture}
}\big)$.
We also fix the choice of a right adjoint $F_i$ to $E_i$ for each $i
\in I$, and set $e_i := [E_i]$ and $f_i := [F_i]$, which are
endomorphisms of $K_0(\pC)$.
Then we impose the following axiom:
\begin{itemize}
\item[(A4)] 
for each $i,j \in I$ and $\la \in P$,
 the commutator $[e_i, f_j]$ acts on $K_0(\pC_\lambda)$
as multiplication by the scalar $\delta_{i,j} \langle h_i,\lambda\rangle$.
\end{itemize}
%, letting $\eta_i:1_\C \rightarrow F_i E_i$ and $\eps_i:E_i F_i
%\rightarrow 1_\C$ be the unit and counit of the adjunction.
We say that the categorical action is {\em nilpotent}\footnote{Just as
  good would be to assume there is $c \in \k$ that $(x_i-c 1)_V$ is nilpotent for
all $i$ and $V$.}
if $(x_i)_{V}$ is a nilpotent element of the finite-dimensional
algebra
$\End_\C(E_i V)$
for all $i \in I$ and $V \in \ob\fC$.
\end{definition}

The following gives a general recipe producing a
categorical action on the locally Schurian category $\rMod \A$
for any finite-dimensional 2-representation $\A$.
In particular, applying it to the finite-dimensional 2-representation
$\Lm(\kappa'|\kappa)$ from Construction~\ref{bad}, 
this shows that 
$\rMod \Lm(\kappa'|\kappa)$ admits a categorical action; this example
is also nilpotent thanks to Lemma~\ref{nilp}.

\begin{example}\label{right}
Let $\A = \coprod_{\lambda \in P} \A_\lambda$ be a finite-dimensional 2-representation.
Let 
$A = \bigoplus_{\lambda \in P} A_\lambda$ 
be the associated locally unital algebra
as in Remark~\ref{cats};
the distinguished idempotents in $A$ are indexed by
$X = \bigsqcup_{\lambda \in P} X_\lambda$ where $X_\lambda := \ob \A_\lambda$.
We define a categorical action on $\C := \rMod A$ as follows.
\begin{itemize}
\item
Fix representatives $\{L(b)\:|\:b \in \B_\lambda\}$ for the isomorphism
classes of irreducible $A_\lambda$-modules, and let $\B :=
\bigsqcup_{\lambda \in P} \B_\lambda$.
This partition induces the decomposition
$\C = \prod_{\lambda \in P} \C_\lambda$ required for (A1); of
course, we have that $\C_\lambda = \rMod
A_\lambda$.
\item
The functor $E_i$ 
defines locally unital homomorphisms
$e_i:A_\lambda \rightarrow A_{\lambda+\alpha_i}$ 
for each $\lambda \in P$; moreover,
$e_i(1_u) = 1_{E_i u}$ for each $u \in X_\lambda$.
Let $e_i^* A_{\lambda+\alpha_i}$ be
the $(A_{\lambda}, A_{\lambda+\alpha_i})$-bimodule obtained from
$A_{\lambda+\alpha_i}$ by restricting the
natural left action through this homomorphism.
Tensoring on the right with this bimodule defines a functor
$\hat E_i:\rMod A_\lambda \rightarrow \rMod A_{\lambda+\alpha_i}$
for each $\lambda \in P$.
This is the data required for (A2). The endofunctor $\hat E_i$ is sweet because
the functor $F_i$ extends similarly to a functor
$\hat F_i:\rMod A_{\lambda+\alpha_i} \rightarrow \rMod A_{\lambda}$ which
is biadjoint to $\hat E_i$ thanks to Lemma~\ref{nec}.
\item
The natural transformation $x_i:E_i
  \rightarrow E_i$
on each $u \in X_\lambda$ 
produces a family of elements
$x_{i;u} \in 1_{e_i u} A_{\lambda+\alpha_i} 1_{e_i u}\:(u \in
X_\lambda)$ 
such that $e_i(f) x_{i;u} = x_{i;v} e_i(f)$ for all $f \in 1_v A_\lambda 1_u$.
Hence, there is a bimodule homomorphism
$e_i^* A_{\lambda+\alpha_i} \rightarrow e_i^*
A_{\lambda+\alpha_i}$ defined on
$1_{e_i u} A_{\lambda+\alpha_i}$ by
left multiplication by $x_{i;u}$,
from which we get 
$\hat x_i:\hat E_i \rightarrow \hat E_i$.
Similarly, $\tau_{ij}:E_i E_j \rightarrow E_{j} E_i$
translates to $\tau_{ij;u} \in 1_{e_je_i u} A_{\lambda+\alpha_i+\alpha_j}
1_{e_i e_j u}$
such that $(e_je_i(f)) \tau_{ij;u} = \tau_{ji;v} (e_ie_j(f))$
for all $f \in 1_v A_\lambda 1_u$ and $u,v \in X_\lambda$.
Left multiplication by these elements defines a bimodule homomorphism
$(e_i e_j)^* A_{\lambda+\alpha_i+\alpha_j}
\rightarrow (e_je_i)^* A_{\lambda+\alpha_i+\alpha_j}$.
The composite functor $\hat E_i \hat E_j$ is defined by tensoring with $
(e_i e_j)^* A_{\lambda+\alpha_i+\alpha_j}  \cong
e_j^* A_{\lambda+\alpha_j} \otimes_{A_{\lambda+\alpha_j}}
  e_i^* A_{\lambda+\alpha_i+\alpha_j}$, 
so this is what we need to get 
$\hat \tau_{ij}:\hat E_i \hat E_j \rightarrow \hat E_j \hat E_i$.
Thus we have the data for (A3).
\item Finally the axiom (A4) follows from Lemma~\ref{gg} and (\ref{yoneda2}).
\end{itemize}
\end{example}

\iffalse
\begin{remark}
Assume $\g$ is of type $A_\infty$ if $\operatorname{char} \k = 0$
or affine type $A_{p-1}^{(1)}$ if $\operatorname{char} \k = p$,
and that $\kappa+\kappa'$ is of level zero, i.e. $\sum_{i \in
  I} \langle h_i,\kappa+\kappa' \rangle = 0$.
Then we expect that the algebras
$A_\lambda$ from Construction~\ref{right} 
can also be realized as blocks of
the cyclotomic oriented Brauer categories
$\mathcal{OB}(\kappa'|\kappa)$ 
from \cite{BCNR}.
They are definitely not Artinian in general. We conjecture that they
are Noetherian, and moreover 
that the functors $E_i$ and $F_i$ send
irreducible modules to modules of finite length.
\end{remark}
\fi

\iffalse
\begin{remark}
Conversely, if we are given a categorical action of $\g$ on a locally
Schurian category $\C$, then $\pC$ is a finite-dimensional 2-representation.
\end{remark}
\fi

In fact, assuming nilpotency, the notion of a
finite-dimensional 2-representation is equivalent to the notion of
categorical action on a locally Schurian category.
This depends on the following theorem, which is a variation on \cite[Theorem 5.27]{Rou}.
It is a remarkable example of relations on the Grothendieck
group (specifically, axiom (A4)) implying relations between 2-morphisms (specifically,
axiom (M5)); 
Rouquier refers to this as ``control by $K_0$.''
It is very useful since (M5) can be very difficult to check directly.

\begin{theorem}\label{jolly}
Suppose that we are given a {nilpotent} categorical action on
some locally Schurian category $\C$.
Then the natural transformations $\sigma_{ij}\:(i \neq j)$ and
$\rho_{i,\lambda}$ from 
Definition~\ref{2rep} are invertible. Hence, $\C$ is a locally
Schurian 2-representation.
%; equivalently, $\pC$ is a finite-dimensional
%2-representation.
%Moreover, 
%$\mathbb{C}\otimes_{\Z} K_0(\pC)$
%is an integrable $\g$-module.
\end{theorem}

\begin{proof}
Note to start with
that
$\mathbb{C}\otimes_{\Z} K_0(\pC)$
is an integrable $\g$-module, by the same argument as in the proof of Lemma~\ref{maid}.

In this paragraph, we explain how to see that $\rho_{i,\lambda}$ is
invertible (for fixed $i$ and $\lambda$).
By Lemma~\ref{swee}, it suffices to show that $\rho_{i,\lambda}$ is
invertible on $L(b)$ for $b \in \B_\lambda$.
In the Artinian case, this follows
%\footnote{
%We note that the argument in {\em loc. cit.} depends on the nilpotency
%assumption, as its proof invokes
%\cite[Theorem 5.27]{CR} in which nilpotency is implicit.} 
immediately from 
\cite[Theorem 5.22]{Rou}.
The proof in general reduces to the Artinian case as follows.
By (A4) and integrability, the set 
$$
\B'' := \{a \in \B\:|\:P(a)\text{ is a summand
of some sequence of $E_i$ and $F_i$ applied to }P(b)\}
$$
is finite.
Of course, $b \in\B''$.
Let $\B' := \B \setminus \B''$ and $\pi: \C \rightarrow \C / \C'$ be the
corresponding Serre quotient as in $\S$\ref{sq}.
The isomorphism classes of 
irreducible objects of $\C / \C'$ are represented by $\{\pi L(a)\:|\:a
\in \B''\}$; in particular,
$\C / \C'$ is finite. Moreover, using the left adjoint functor $\pi^!$, we
may identify its 
complexified $K_0$
with
the subspace of $\mathbb{C}\otimes_{\Z} K_0(\pC)$ spanned by
$\{[P(a)]\:|\:a \in \B''\}$. The definition of $\B''$ ensures that
this subspace is stable under the action of $e_i$ and $f_i$.
Consequently,
the functors $E_i$ and $F_i$ preserve the subcategory $\C'$;
for example, for $E_i$, this follows because
$\Hom_{\C}(P(c), E_i L(d))
\cong \Hom_{\C}(F_i P(c), L(d)) = 0$
for all $c \in \B'', d \in \B'$.
Hence, $E_i$ and $F_i$
induce endofunctors of $\C / \C'$, 
showing that
$\C / \C'$ admits a categorical
$\mathfrak{sl}_2$-action.
By \cite[Theorem 5.22]{Rou}, $\rho_{i,\lambda}$ is invertible on
$\pi L(b) \in \ob \C / \C'$. 
It just remains to invoke Lemma~\ref{qlem} to deduce that $\rho_{i,\lambda}$
is invertible on $L(b) \in \ob \C$ too.
In order to check the hypotheses
of Lemma~\ref{qlem} here, we 
should note that if $X$ is any functor obtained by taking a finite
composition and/or direct sum of the
categorification functors $E_i$ and $F_i$,
then $X L(b)$ is finitely generated and cogenerated by
Theorem~\ref{sef}.
Moreover all constituents of $\operatorname{soc}(X L(b))$ and
$\operatorname{hd}(X L(b))$ are $\cong L(c)$ for
$c \in \B''$; for example, to see this for the head, 
the right adjoint $Y$ to $X$
preserves $\C'$ as before, so for $c \in \B'$ we get that
$\Hom_\C(X L(b), L(c)) \cong \Hom_\C(L(b), Y L(c)) = 0$.

It remains to show that $\sigma_{ij} (i \neq j)$ is invertible. For
this, we 
appeal to the proof of \cite[Theorem 5.25]{Rou} to get that
$\sigma_{ij}$ is invertible on 
$E_i^r K$ for all $r \geq 0$ and any irreducible $K$ with $F_i K = 0$;
this is a very general result which requires no finiteness assumptions about
$\C$ other than integrability.
To deduce the invertibility of $\sigma_{ij}$ on arbitrary objects,
we claim that every irreducible object $L\in \ob \C$ can be realized as a
quotient (resp.\ subobject) of some such object $E_i^r K$. 
Using the claim and naturality, the invertibility of $\sigma_{ij}$ on
$E_i^r K$ implies
the surjectivity (resp. injectivity) of $\sigma_{ij}$ on $L$ too.
Then we apply Lemma~\ref{swee} to get that $\sigma_{ij}$ is invertible on
arbitary objects.
Finally, we must prove the claim.
By integrability, there is a unique $r \geq 0$ such that $F_i^r L \neq
0$ but $F_i^{r+1} L = 0$. Then we let $K$ be any irreducible
constituent of the socle (resp. head)
of $F_i^r L$, so that $K \hookrightarrow F_i^r L$ (resp. $F_i^r L
\twoheadrightarrow K$); this relies on the fact that $F_i^r L$
is finitely cogenerated (resp.\ generated) according to Theorem~\ref{sef}.
Applying adjointness, we get that there is a non-zero
homomorphism
$E_i^r K \twoheadrightarrow L$ (resp. $L \hookrightarrow E_i^r K$), as
required.
\end{proof}

\begin{remark}
The proof of Theorem~\ref{jolly} relies ultimately on \cite[Theorem 5.27]{CR}, in which
nilpotency is certainly assumed.
However, we expect that this result can be generalized, so that
the nilpotency assumption in the statement of Theorem~\ref{jolly}
(and in the remainder of this subsection) should actually be
unnecessary. 
\end{remark}

\begin{example}\label{left}
Let $\C$ be a locally Schurian category admitting a nilpotent
categorical action. Fix a set $X_0$ indexing finitely generated
projective objects $(P_x)_{x \in X_0}$ such that each $P_x$
belongs to some weight subcategory of $\C$.
For $n \geq 1$, define $X_n$ and $(P_x)_{x \in X_n}$ recursively by
letting $X_n$ consist of the symbols $e_i x, f_i x$ for all $x \in
X_{n-1}$ and $i \in I$, and setting $P_{e_i x} := E_i P_x, P_{f_i x}
:= F_i P_x$.
Let $X := \bigsqcup_{n \geq 0} X_n$.
We assume further that $(P_x)_{x \in X}$ is a projective generating family
for $\C$.
Having made this choice, we can define a finite-dimensional
2-representation $\mathcal A$  as follows.
\begin{itemize}
\item
Let $\mathcal A$ be the finite-dimensional category with object set
$X$, $\Hom_{\mathcal A}(x,y) := \Hom_{\mathcal C}(P_x,P_y)$,
and composition induced by composition in $\mathcal C$.
Note that $\mathcal A = \coprod_{\lambda \in P} \mathcal A_\lambda$
where
$\mathcal A_\lambda$ is the full subcategory generated by 
$X_\lambda := \{x \in X\:|\:P_x \in \ob \C_\lambda\}$.
\item 
Let $\bar E_i, \bar F_i:\mathcal A \rightarrow \mathcal A$ be the endofunctors
defined on objects by $\bar E_i x := e_i x, \bar F_i x := f_i x$.
On morphisms, $\bar E_i$ and $\bar F_i$ are defined by
applying the given categorification functors $E_i$ and $F_i$ in $\mathcal C$.
\item 
Let $\bar x_i:\bar E_i \Rightarrow \bar E_i, \bar \tau_{ij} :\bar E_i
\bar E_j \Rightarrow \bar E_j \bar E_i, \bar \eta_i:1_{\mathcal A}
\Rightarrow \bar F_i \bar E_i$ and $\bar \eps_i:\bar E_i \bar F_i
\Rightarrow 1_{\mathcal A}$ be the natural transformations obtained
by restricting $x_i, \tau_{ij}, \eta_i$ and $\eps_i$ in the obvious way.
\end{itemize}
This produces all of the data required by
Definition~\ref{2rep}(M1)--(M4). The final axiom (M5) is satisfied thanks
to Theorem~\ref{jolly}.
\end{example}

We leave it as an exercise for the reader to show that
Constructions~\ref{right} and \ref{left} are quasi-inverses in the appropriate sense.
In particular, if one starts
with $\mathcal C$ equipped with a nilpotent categorical action, 
applies Construction~\ref{left} to obtain $\mathcal
A$,  then applies Construction~\ref{right} to define a categorical action
on $\rMod \mathcal A$, then $\mathcal C$ and $\rMod \mathcal A$ are
strongly equivariantly equivalent.

\subsection{Associated crystals}
Finally we recall a definition of Kashiwara; e.g. see \cite{Kas}.
%Following \cite[$\S$5]{BeK}
%we adopt the following language:
%a {\em partial bijection} $f:X \dashrightarrow Y$ 
%is a bijection $f:\dom(f) \stackrel{\sim}{\rightarrow} \codom(f)$ 
%between 
%subsets $\dom(f) \subseteq X$ and $\codom(f)
%\subseteq Y$. 
%The inverse $f^{-1}:Y \dashrightarrow X$ of a partial bijection 
%is $f^{-1}:\codom(f) \stackrel{\sim}{\rightarrow} \dom(f)$.
%Also the composition of two partial bijections $f:X \dashrightarrow Y$
%and $g:Y \dashrightarrow Z$ is 
%$g \circ f:\dom(f) \cap f^{-1}(\dom(g)) \stackrel{\sim}{\rightarrow}
%\codom(g)\cap g(\codom(f))$.
%Note that $\dom(g \circ f) = \codom(f^{-1}\circ g^{-1})$.

\begin{definition}\label{nc}
A {\em normal crystal} is a set $\B$ with a decomposition  $\B =
\bigsqcup_{\lambda \in P} \B_\lambda$, plus \textit{crystal operators}
$\tilde{e}_i, \tilde{f}_i : \B \to \B \sqcup \{0\}$ for each $i \in I$
satisfying the following axioms:
\begin{itemize}
\item[(C1)] for every $\lambda \in P$, the crystal operator $\tilde{e}_i$ restricts to a map $\B_\lambda \to \B_{\lambda + \alpha_i} \sqcup \{0\}$;
\item[(C2)] 
for $b \in \B$, we have that
$\tilde{e}_i(b) = b' \neq 0 \text{ if and only if } \tilde{f}_i(b') = b \neq 0;$
\item[(C3)] for every $b \in \B$, there is an $r \in \N$ such that
  $\tilde{e}_i^{r}(b) = \tilde{f}_i^r(b) = 0$.
\end{itemize}
For each $i$, define functions $\eps_i, \varphi_i : \B \to \N$ by
$$\eps_i(b) = \max \{r \in \N \:|\: \tilde{e}_i^r(b) \neq 0 \}, \hspace{0.25in}
\varphi_i(b) = \max \{r \in \N \:|\: \tilde{f}_i^r(b) \neq 0 \}.$$
Then
we also require that
\begin{itemize}
\item[(C4)] $\varphi_i(b) - \eps_i (b)
= \langle h_i , \lambda \rangle$ for each $b \in \B_\lambda$ and $i \in I$.
\end{itemize}
\end{definition}

The following theorem is essentially due to Chuang and Rouquier \cite{CR}, but
it has its origins in \cite{GV, G}.
It shows that every nilpotent categorical action has a canonical {\em associated
crystal}.

\begin{theorem}\label{ac}
Suppose that we are given a nilpotent categorical action on a locally
Schurian category $\C$ with irreducible objects $\{L(b)\:|\:b \in
\B\}$
as in Definition~\ref{catact}.
There is a unique structure of normal crystal on 
$\B = \bigsqcup_{\lambda \in P} \B_\lambda$
such that
\begin{enumerate}
\item 
$\tilde e_i b \neq 0\Leftrightarrow E_i L(b) \neq 0\Rightarrow
\operatorname{soc}(E_i L(b)) \cong \operatorname{hd}(E_i L(b)) \cong L(\tilde e_i b)$;
\item 
$\tilde f_i b \neq 0\Leftrightarrow F_i L(b) \neq 0\Rightarrow
\operatorname{soc}(F_i L(b)) \cong \operatorname{hd}(F_i L(b)) \cong L(\tilde f_i b)$.
\end{enumerate}
Moreover, the following hold for any $b \in \B$, $i \in I$ and $0 \leq
n \leq m := \eps_i(b)$:
\begin{enumerate}
\item[(3)]
$E_i^{(n)} L(b)$ has irreducible socle and head both isomorphic to
$L(\tilde e_i^n b)$;
\item[(4)]
$[E_i^{(n)} L(b):L(\tilde e_i^n b)] = \binom{m}{n}$,
and 
all irreducible subquotients of $E_i^{(n)} L(b)$ other than
 $L(\tilde e_i^n b)$ are of the form $L(c)$ for $c$ with $\eps_i(c) < m - n$;
\item[(5)]
the natural action of $Z(NH_n)$ on 
$E_i^{(n)} L(b)$
induces an isomorphism 
$$
\Sym_n / \langle
\h_{m-n+1},\h_{m-n+2},\dots,\h_m\rangle 
\stackrel{\sim}{\rightarrow}
\End_{\C}(E_i^{(n)} L(b)),
$$
where $\h_r$ denotes the $r$th elementary symmetric polynomial in
$X_1,\dots,X_n$.
\end{enumerate}
Analogous statements to (3)--(5) with $E$ and $\eps$ replaced by $F$ and
$\varphi$ also hold.
\end{theorem}

\begin{proof}
Note this only involves some fixed $i$, so we are reduced
immediately to the case that $\g$ is of rank one. In the Artinian
case, it suffices to work in the Schurian category $\fC$, and then
everything that we need is a consequence of \cite[Proposition
5.20]{CR} and the construction of \cite[Theorem 5.22]{Rou}. To give a little more detail, {\em loc. cit.} shows that
$E_i L(b)$ is either zero, or it has irreducible head and socle
which are isomorphic, and similarly for $F_i L(b)$.
Hence we can use (1)--(2) 
to define 
$\tilde e_i, \tilde f_i:\B \rightarrow \B \sqcup \{0\}$.
The axiom (C1) is clear, while (C2) follows by an adjunction argument.
Temporarily redefining $\eps_i(b)$ from 
$\eps_i(b) := \max\{r \in \N\:|\:E_i^{(r)} L(b)
\neq 0\}$, 
we get that properties (3) and (4)
hold by \cite[Proposition 3.20]{CR} again.
Using them, an easy induction on $\eps_i(b)$ shows that (C3) holds
and that
$\eps_i(b)$ agrees with the function from Definition~\ref{nc}.
Similarly,
$$
\phi_i(b) :=\max\{n \in \N\:|\:F_i^{(n)} L(b)
\neq 0\}=
\max\{n \in \N\:|\:\tilde f_i^{n} b
\neq 0\}.
$$
Now we can establish the final axiom (C4). Suppose that 
$b \in \B_\lambda$ and set $c := \tilde e_i^m b$.
We have that $\eps_i(c) = 0$, hence $E_i L(c) = 0$.
Thus, in the Grothendieck group, the class of $L(c)$ is an
$\mathfrak{sl}_2$-highest weight vector.
By $\mathfrak{sl}_2$-theory, we deduce that $\phi_i(c) = \langle
h_i,\lambda +m \alpha_i\rangle$.
Hence, $\phi_i(b) - \eps_i(b) = \phi_i(c) - 2m = \langle h_i,\lambda\rangle$ as required.
Finally, for (5),
the proof of \cite[Proposition 3.20]{CR} shows that the
natural action of $NH_n$ on $E_i^n L(b)$ induces an isomorphism
$NH_n / \langle X_1^m \rangle
\stackrel{\sim}{\rightarrow} \End_\C(E_i^n L(b))$.
By an elementary relation chase (omitted), the two-sided ideal of $NH_n$
generated by $X_1^m$ is also generated by $\h_{m-n+1},\dots,\h_m$. 
Recalling that $NH_n$ is a matrix algebra over its center $\Sym_n$, 
we deduce 
on truncating with the idempotent $\pi_{i,n}$
that $\End_\C(E_i^{(n)} L(b)) \cong \Sym_n / \langle
\h_{m-n+1},\dots,\h_m\rangle$.

To extend the result to the general locally Schurian case, we make a
reduction similar to the one made in the second paragraph of
the proof of Theorem~\ref{jolly}.
Fix $b \in \B$ and define $\B'', \B', \C'$ and
the quotient functor
$\pi:\C \rightarrow \C / \C'$ exactly as there.
As we explained already, $E_i$ and $F_i$ induce endofunctors of $\C / \C'$,
hence giving a categorical $\mathfrak{sl}_2$-action on $\C / \C'$,
which is finite. Moreover all $E_i^{(n)} L(b)$ and $F_i^{(n)} L(b)$
are finitely generated and cogenerated, and their socles and heads
have constituents only of the form $L(c)$ for $c \in \B''$.
So we can use Lemma~\ref{qlem} to transport the results from
$\C / \C'$ established in the previous paragraph to $\C$, and the
general result
follows.
Perhaps the only statement that requires additional comment is the
second assertion of (4). For this, the other properties imply that
$E_i^{(m-n)}$ annihilates all composition factors $L(c)$ of $E_i^{(n)} L(b)$
different from $L(\tilde e_i^n b)$, hence we get that
$\eps_i(c) < m-n$.
\end{proof}

\begin{remark}
By a classical result, the algebra $\Sym_n / \langle
\h_{m-n+1},\h_{m-n+2},\dots,\h_m\rangle$ in Theorem~\ref{ac}(5) is isomorphic to 
the cohomology of the Grassmannian $\operatorname{Gr}_{n,m}$. This is
explained in \cite[$\S$3.3.2]{CR}.
\end{remark}

\begin{remark}
It is interesting to consider what happens in Theorem~\ref{ac} if the
nilpotency assumption is dropped. In general, one still obtains a
crystal structure on $\B$, but for a certain unfurling
$\widetilde{\mathfrak{g}}$ of $\mathfrak{g}$ 
in the sense of \cite{erratum}. This is a consequence of the isomorphism
theorem established in \cite[$\S$3]{erratum}.
\end{remark}

The following well-known corollary is a first application; again this argument appeared
already in a special case in \cite{G}.

\begin{corollary}\label{lastone}
For $\kappa \in P^+$,
the Grothedieck group $K_0(\dotLm(\kappa))$ is the $\Z$-span of the vectors
$f_{i_n}^{(r_n)} \cdots f_{i_1}^{(r_1)} [1_\kappa]$
for $i_1,\dots,i_n \in I$ and $r_1,\dots,r_n \geq 1$.
\end{corollary}

\begin{proof}
We apply Theorem~\ref{ac} with $\C := \rMod \Lm(\kappa)$; cf. Construction~\ref{right}.
Let $M$ be the span of all the vectors
$f_{i_n}^{(r_n)} \cdots f_{i_1}^{(r_1)} [1_\kappa]$.
Proceeding by downward induction on weight, consider some $\lambda < \kappa$.
We need to show for each $b \in \B_\lambda$
that $[P(b)] \in M$.
Pick $i$ so that $m := \eps_i(b) \neq 0$.
Using 
Theorem~\ref{ac} and an argument with adjunctions, one shows that
$F_i^{(m)} P(\tilde e_i^m(b))
\cong P(b) \oplus (*)$
where $(*)$ is a direct sum of projectives of the form $P(b')$ for $b'
\in \B_\lambda$ with $\eps_i(b') > \eps_i(b)$.
By downward induction on $\eps_i(b)$, we may assume that all of these 
$[P(b')]$ lie in $M$. Hence, we get that $[P(b)] \in M$ too.
\end{proof}

The proof of Corollary~\ref{lastone} implicitly uses the defining property of
a {\em dual perfect basis} from \cite[Definition 4.2]{KKKS}.
In fact, Theorem~\ref{ac} easily implies for any nilpotent locally
Schurian categorical action that
$\{
{[{P(b)}]} \:|\: b \in \B
\}$ is a
{dual perfect basis} for $\mathbb{C}\otimes_{\Z} K_0(\pC)$.
In particular, we recover the following well-known result on appealing also to 
\cite[Theorem 6.1]{KKKS};
this was originally proved in \cite{LV} by a different method.

\begin{corollary}\label{thec}
For $\kappa \in P^+$,
the crystal associated to the 
minimal categorification $\Lm(\kappa)$
is a copy of Kashiwara's highest weight crystal $\B(\kappa)$.
\end{corollary}

\begin{remark}
If $\C$ is Artinian, one can also show that $\{[L(b)]\:|\:b \in
\B\}$
is a {\em perfect basis} for the
complexified Grothendieck group of the Schurian category $\fC$
in the (older) sense of \cite[Definition 5.49]{BeK}.
This was observed originally by Shan \cite[Proposition
6.2]{S}. Combined with 
\cite[Theorem 5.37]{BeK} (in place of \cite[Theorem 6.1]{KKKS}), 
Corollary~\ref{thec} may also be deduced from this; cf. 
\cite[Remark 10.3.6]{Kbook}.
However, perfect bases are not a natural thing to consider in the
locally Schurian setup: in general it is not even clear that $E_i$
and $F_i$ send irreducible objects to objects of finite length.
\end{remark}

\begin{remark}
It is natural to expect that the crystal associated to
$\Lm(\kappa'|\kappa)$
is Kashiwara's tensor product $\B'(\kappa') \otimes \B(\kappa)$ of the lowest
weight crystal $\B'(\kappa')$ with the highest weight crystal
$\B(\kappa)$.
We hope to prove this in subsequent work using some of Losev's techniques from \cite{Losev}.
\end{remark}
\iffalse
We can prove this using the general framework of {\em tensor product categorifications}
from \cite[Remark 3.6]{LW}, but at present this relies on Remark~\ref{crazy}.
In fact, the associated crystal
of any locally Schurian tensor product categorification 
can be computed via Kashiwara's rule for crystal tensor products.
This assertion is explained in detail in \cite{D}, building
on Losev's ideas from \cite{Losev} (which were extended once already in
\cite[$\S$7]{LW}), 
In \cite{D}, the results discussed in this subsection have been extended
to the super
Kac-Moody 2-categories of \cite{BE2}.
\end{remark}
\fi

\end{document}